\documentclass[12pt,reqno]{amsart}
\usepackage{fullpage}

\usepackage{amssymb,amsfonts,bbm}
\usepackage[all,arc]{xy}
\usepackage{enumerate}
\usepackage{mathrsfs}
\usepackage{color}   
\usepackage{hyperref}
\hypersetup{
    colorlinks=black, 
    linktoc=all,     
    linkcolor=black,  
    citecolor=black,
}
\usepackage{amsthm}
\usepackage{amsmath}
\usepackage{graphicx}
\usepackage{wasysym}
\usepackage{pgf,tikz}
\usetikzlibrary{matrix,arrows,decorations.pathmorphing}
\usetikzlibrary{positioning}

\newcommand{\oG}{\overline{G}}
\newcommand{\oGL}{\overline{\mathrm{GL}}}
\newcommand{\oH}{\overline{H}}
\newcommand{\oM}{\overline{M}}
\newcommand{\oP}{\overline{P}}

\newcommand{\oT}{\overline{T}}

\newcommand{\bB}{\mathbf{B}}
\newcommand{\bG}{\mathbf{G}}
\newcommand{\bH}{\mathbf{H}}
\newcommand{\bK}{\mathbf{K}}
\newcommand{\bM}{\mathbf{M}}
\newcommand{\bN}{\mathbf{N}}
\newcommand{\bP}{\mathbf{P}}
\newcommand{\bT}{\mathbf{T}}
\newcommand{\bU}{\mathbf{U}}
\newcommand{\bZ}{\mathbb{Z}}
\newcommand{\bGL}{\mathbf{GL}}
\newcommand{\bSL}{\mathbf{SL}}
\newcommand{\bSO}{\mathbf{SO}}
\newcommand{\bSp}{\mathbf{Sp}}
\newcommand{\bSpin}{\mathbf{Spin}}

\DeclareMathOperator{\SO}{SO}

\newcommand{\obG}{\overline{\mathbf{G}}}

\newcommand{\obGsk}{\overline{\mathbf{G}^{\square,k}}}
\newcommand{\obH}{\overline{\mathbf{H}}}

\newcommand{\obT}{\overline{\mathbf{T}}}

\newcommand{\BD}{\mathbf{BD}}

\newcommand{\CExt}{\mathbf{CExt}}
\newcommand{\Irr}{\mathrm{Irr}}
\newcommand{\Int}{\mathrm{Int}}
\newcommand{\pr}{\mathrm{pr}}
\newcommand{\Res}{\mathrm{Res}}

\newcommand{\St}{\mathrm{St}}
\newcommand{\Wh}{\mathrm{Wh}}
\newcommand{\Spec}{\mathrm{Spec}}
\newcommand{\op}{\mathrm{op}}

\newcommand{\sca}{\mathcal{A}}

\newcommand{\sce}{\mathcal{E}}

\newcommand{\sco}{\mathcal{O}}
\newcommand{\scp}{\mathcal{P}}

\newcommand{\scw}{\mathcal{W}}
\newcommand{\scy}{\mathcal{Y}}

\newcommand{\bD}{\mathbf{D}}
\newcommand{\bx}{\boldsymbol{x}}
\newcommand{\by}{\boldsymbol{y}}

\newcommand{\bc}{\mathbb C}

\newcommand{\bq}{\mathbb Q}

\newcommand{\ba}{\mathbb A}

\newcommand{\la}{\langle}
\newcommand{\ra}{\rangle}

\newcommand{\bs}{\backslash}

\newcommand{\al}{\alpha}

\newcommand{\lam}{\lambda}

\DeclareMathOperator{\Gal}{Gal}

\DeclareMathOperator{\End}{End}

\DeclareMathOperator{\GL}{GL}
\DeclareMathOperator{\Hom}{Hom}
\DeclareMathOperator{\Ind}{Ind}

\DeclareMathOperator{\M}{M}

\newtheorem{Thm}{Theorem}[section]
\newtheorem{Prop}[Thm]{Proposition}
\newtheorem{Lem}[Thm]{Lemma}
\newtheorem{Cor}[Thm]{Corollary}

\theoremstyle{definition}
\newtheorem{Def}[Thm]{Definition}

\theoremstyle{remark}
\newtheorem{Rem}[Thm]{Remark}

\theoremstyle{definition}

\begin{document}

\title{Twisted doubling integrals for Brylinski-Deligne extensions of classical groups}
\author{Yuanqing Cai}

\begin{abstract}
We explain how to develop the twisted doubling integrals for Brylinski-Deligne extensions of connected classical groups. This gives a family of global integrals which represent Euler products for this class of non-linear extensions.
\end{abstract}

\date\today
\address{Faculty of Mathematics and Physics, Institute of Science and Engineering, Kanazawa University, Kakumamachi, Kanazawa, Ishikawa, 920-1192, Japan}
\email{cai@se.kanazawa-u.ac.jp}
\maketitle
\tableofcontents

\section{Introduction}

One of the goals in the Langlands program is to study analytic properties of automorphic $L$-functions. To this problem, a basic approach is to find a global integral that represents the automorphic $L$-function one would like to study. There are some questions to be addressed in this approach. First, one needs to show that the global integral is Eulerian. For this, one usually uses certain multiplicity one results, such as uniqueness of Whittaker models. To establish the desired properties of the $L$-functions, one uses the corresponding properties of an Eisenstein series used in the global integral, or some spectral theory results such as the Poisson summation formula. For reductive groups, the approach of global integrals is successful in several important cases. The purpose of this paper is to describe a family of global integrals for a class of non-linear covers of reductive groups.

\subsection{Brylinski-Deligne covering groups}

Let $\bG$ be a connected reductive group over a number field $F$. In \cite{BD01}, Brylinski and Deligne considered the category of multiplicative $\bK_2$-torsors on $\bG$; these are extensions of $\bG$ by the sheaf $\bK_2$ of Quillen's $K_2$ group in the category of sheaves of groups on the big Zariski site of $\Spec(F)$:
\[
1\to \bK_2 \to \obG \to \bG \to 1.
\]
Brylinski and Deligne gave an elegant and functorial classification of this category in terms of enhanced root-theoretic data, similar to the classification of split connected reductive groups by their root data.

We now assume that the base field $F$ contains a full set of $n$-th roots of unity. Then at every local place $v$, there is a functor from the category of multiplicative $\bK_2$-torsors $\obG$ on $\bG$ to the category of topological central extensions:
\[
1\to \mu_n \to \oG_v \to G_v=\bG(F_v) \to 1,
\]
which glues to a central extension of the adelic group
\[
1\to \mu_n \to \oG(\ba) \to G(\ba) \to 1.
\]
The global extension is equipped with a natural splitting $G(F)\to \oG(\ba)$. This naturally leads to the notion of automorphic forms on this class of groups. These topological central extensions may be considered of ``algebraic origin'' and can be constructed using cocycles which are essentially algebraic in nature. This construction does not exhaust all topological central extensions, but it captures a sufficiently large class of such extensions, including all interesting examples which have been investigated so far.

\subsection{A Langlands program for covering groups}

There has been serious progress in extending the Langlands program to the Brylinski-Deligne extensions. We refer to \cite{GGW18} for a comprehensive discussion of the history of covering groups. From our point of view, there are several reasons to study automorphic forms on covering groups. The first is that a Langlands program for covering groups indeed provides supporting evidences to the Langlands program for linear groups. The second is to hope that we can relate automorphic forms on covering groups and linear groups, and use this to gain new information regarding linear groups. Indeed, the development of the Langlands program already uses covering groups. A notable example is the Weil representations, defined on double covers of the symplectic groups. Another example is the Rankin-Selberg integral for the symmetric square $L$-function for $\GL_n$, which uses the theta representations on double covers of general linear groups \cite{BG92}. The theory of Weyl group multiple Dirichlet series, which is closely tied to Whittaker coefficients of Eisenstein series, has important consequences in analytic number theory.

A natural question is to test whether the myriad of global integrals for various $L$-functions for linear groups have counterparts in the covering case. In the linear case, such a theory relies heavily on the uniqueness of certain models, in particular the Whittaker model. The failure of such multiplicity one results in the covering group case causes serious obstructions to some development for the theory, and it is fundamentally difficult to find integral representations for $L$-functions for covering groups.

Nevertheless, there are two recent constructions that set up the first step towards this area. The first is the calculation of constant terms of Eisenstein series on Brylinski-Deligne covering groups \cite{Gao18}. A consequence of this calculation is the meromorphic continuation of many interesting $L$-functions. Another example is a generalization of the doubling method (\cite{PSR87,CFGK19}) to covers of symplectic groups \cite{Kaplan} (see also \cite{CFGK2016} for a brief description of the method). (Also note that the case of the double covers of symplectic groups was already considered in the literature; for example, see \cite{Gan12}.) In \cite{Kaplan}, the unfolding argument and unramified calculation are carried out. As a consequence, the global integral represents the tensor product $L$-function for a cover of a symplectic group and a cover of a general linear group under certain assumptions. It also has the potential to establish a functional equation for the $L$-functions obtained.

\subsection{The doubling integrals}

The doubling integrals grew out of Rallis' work on the inner products of theta lifts -- the Rallis inner product formula. This gives a global integral for the standard $L$-function of a classical group. As a global integral, the doubling integral \cite{PSR87} and its generalizations \cite{CFGK19,Kaplan} enjoy the following nice properties:

\begin{enumerate}
\item It uses the matrix coefficient of a representation $\pi$ of a classical group $G(\ba)$. As a consequence, this construction works for all cuspidal representations of $G(\ba)$. This is a special feature of the construction since most of other constructions only work for some but not all cuspidal representations.

\item The global integral uses certain unique models for some degenerate representations. These representations can be viewed as the generalized Speh representations for covering groups. As we noted above, uniqueness of Whittaker models fails in general for covering groups. However, it is possible that for some `degenerate' representations on certain covering groups, uniqueness of Whittaker models still holds. A typical example of this type is the theta representation \cite{KP84}. Establishing the existence and verifying the expected properties are in fact the technical heart of the doubling method. 

\item One can write down a local version of these global integrals and use it to define local factors. Note that in the Langlands-Shahidi method for covering groups, one can define the local coefficients matrix instead of local coefficient, but it is not straightforward to have a definition of local factors. We refer the reader to \cite{GSS} for recent progress.
\end{enumerate}

\subsection{Main results}

The purpose of this article is to explain how to develop the twisted doubling integrals (see \cite{CFGK19,Kaplan}) for Brylinski-Deligne extensions of \textit{connected classical groups}. Here, we use the conceptual description in \cite{Cai21}, which works for all classical groups uniformly. A major part of this paper is to explain the necessary (non-trivial) modifications in order to develop the twisted doubling integrals for covering groups.

 We include almost all BD extensions for connected classical groups, but have to exclude some cases of unitary groups for some technical reasons (see Remark \ref{rem:commutative assumption}).  We can invoke some neat structures of BD extensions to minimize the use of cocycles as the construction is functorial in nature. Moreover, we only require the condition $\mu_n\subset F^\times$ (instead of $\mu_{2n}\subset F^\times$) in our construction. Once we have the correct setup, working with BD extensions allows us to study some basic properties of twisted doubling integrals just in the linear case. In particular, we carry out the unfolding part following the argument in \cite{Cai21}. 

The most important part in the setup is to find a suitable BD extension of a large classical group for the construction to work. This part is proved case-by-case in Section \ref{sec:BD data explicit}. This is relatively easy when the category of BD extensions is ``rigid''. There is additional complication in the non-simply-connected and non-split case, with the case of $\bSO_{2m}$ being the most sophisticated one. 

As we noted above, another ingredient in the twisted doubling integrals is the construction of some ``Speh'' type representations as the inducing data of some Siegel Eisenstein series. We believe that this is a very difficult question for covering groups and we only discuss this briefly in the Section \ref{sec:construction of Speh}. Indeed, following certain conjectures in \cite{Suzuki98} and \cite{Ginzburg}, we suggest that one should construct such representations from representations of $\GL_k(\ba)$, instead of a cover of $\GL_k(\ba)$. As a consequence, we expect that the twisted doubling integrals represent the tensor product $L$-functions for $\obG\times \bGL_k$. In the second factor we only taking a linear group instead of a cover of $\bGL_k$. This is different from the $L$-functions obtained in \cite{Gao18} and \cite{Kaplan}. It is an interesting question to relate these $L$-functions.

We also note that in \cite{BD01}, BD covers are only studied for connected reductive groups. So in this paper we exclude the orthogonal group $\mathbf{O}_m$ and consider $\bSO_m$ instead. (Here we use the term connected classical groups for these groups. Note that we do not assume any condition on the splitness of the classical groups.) We leave other investigations such as unramified calculations, local and global theory as well as the construction of the generalized Speh representations to future work.

\subsection{Structure of this paper}

The rest of this paper is organized as follows. We first recall some preliminaries on Brylinski-Deligne covering groups in Section \ref{sec:BD extensions} -- \ref{sec:pullback pushout baer sum}. We recall some construction related to the classical groups in Section \ref{sec:doubling variables}. Section \ref{sec:degenerate reps} highlights a family of degenerate representations on covers of general linear groups. Even if we do not how to construct the generalized Speh representations for covering groups, we still know the properties that will be used in the global integrals. We then discuss the basic setup of the twisted doubling integrals in the linear case in Section \ref{sec:Basic setup in the linear case} and discuss necessary assumptions for covering groups in Section \ref{sec:covering assumptions}. The long proofs in Section \ref{sec:covering assumptions} are deferred to Section \ref{sec:BD data explicit}. We then introduce the global integral and prove a global identity in Section \ref{sec:global integrals}. We briefly discuss the construction of the generalized Speh representations and the $L$-functions obtained in Section \ref{sec:construction of Speh}. 

\subsection*{Acknowledgement}
The author would like to thank the referee for very detailed and helpful comments for an earlier version of the paper. Part of this work was carried out when the author was a postdoctoral fellow at the Weizmann Institute of Science and at Kyoto University. The author would like to thank both institutes for providing an excellent working environment.
This research was supported by the ERC, StG grant number 637912, JSPS KAKENHI grant number 19F19019, and MEXT Leading Initiative for Excellent Young Researchers Grant Number JPMXS0320200394.

\subsection{Notation}

We now give a list of notations that are commonly used in this paper.
\begin{itemize}
\item For an algebraic group $\bG$ over $F$, we sometimes write $G(F)=\bG(F)$; if $F$ is a local field, we sometimes write $G=G(F)$. If $F$ is a global field, we usually write $[G]$ for $G(F)\bs G(\ba)$ where $\ba$ is the ring of adeles of $F$, and $G_v=\bG(F_v)$ for a local place $v$.
\item For an algebraic group $\bG$ over $F$ and a field extension $L/F$, we write $\bG_L:=\bG \otimes_F L$ to be the base change of $\bG$ to $L$.
\item $J^{\diamondsuit}$: for a subgroup $J$ of $G$, define $J^{\diamondsuit}=\{(g,g)\in G\times G \mid g\in J\}$.
\item $\Int(g)$: for $g\in G$ and a subgroup $H\subset G$, the conjugation map by induced by $g$ is denoted $\Int(g):H\mapsto gHg^{-1}$ .
\item For an abelian group $A$ with multiplication, we write $[m]$ for the homomorphism $A\to A, \ x\mapsto x^m$.
\end{itemize}

\section{Brylinski-Deligne extensions}\label{sec:BD extensions}

Let $F$ be a local or global field of characteristic $0$. Let $\bG$ be a connected reductive group over $F$. In this section, we discuss the Brylinski-Deligne extensions of $\bG$. The main references for this section are \cite{BD01} and \cite{GG18}.

\subsection{Some structural facts}
Let $F_s$ denote a separable closure of $F$ and $\Gamma=\Gal(F_s/F)$. Let $\bT$ be a maximal $F$-torus of $\bG$. Let
\[
\{
X, \ \Phi, ; \  Y, \ \Phi^\vee
\}
\]
be the resulting (absolute) root system. Here $X$ (respectively, $Y$) is the character lattice (respectively, cocharacter lattice) for $(\bG,\bT)$ over $F_s$. Both $X$ and $Y$ are naturally $\bZ[\Gamma]$-modules. The group $\Gamma$ acts on the root system $\Phi$ as well. Choose a set $\Delta\subset \Phi$ of simple roots from the set of roots $\Phi$, and $\Delta^\vee$ the corresponding simple coroots from $\Phi^\vee$. Write $Y^{sc}\subset Y$ for the sublattice generated by $\Phi^\vee$. Let $\bB\subset \bG_{F_s}$ be the Borel ($F_s$-)subgroup determined by $\Delta$.

For each $\al\in \Phi$, one has the associated root subgroup $\bU_\al\subset \bG_{F_s}$ which is normalized by $\bT_{F_s}$. We shall fix a Chevalley system of \'epinglage for $(\bG_{F_s},\bT_{F_s})$, so that for each $\al\in \Phi$, one has an $F_s$-isomorphism $x_\al:\bG_a\simeq \bU_\al$.

Denote by $W=W(\bG):=N(\bT_{F_s})/\bT_{F_s}$ the (absolute) Weyl group of $(\bG,\bT)$, which we identify with the Weyl group of the coroot system.

\subsection{Multiplicative $\bK_2$-torsors}

The reductive group $\bG$ defines a sheaf of groups on the big Zariski site on $\Spec(F)$. Let $\bK_2$ denote the sheaf of groups on $\Spec(F)$ associated to the $\bK_2$-group in Quillen's K-theory. Then a multiplicative $\bK_2$-torsor is an extension
\[
1\to \bK_2 \to \obG \to \bG \to 1
\]
of sheaves of groups on $\Spec(F)$. We consider the category $\CExt(\bG,\bK_2)$ of such extensions where the morphisms between objects are given by morphisms of extensions. Given two such central extensions, one can form the Baer sum: this equips $\CExt(\bG,\bK_2)$ with the structure of a commutative Picard category.

In \cite{BD01}, Brylinski and Deligne made a deep study of $\CExt(\bG,\bK_2)$ and obtained an elegant classification of this category when $\bG$ is a connected reductive group. We first recall their results.

\subsection{Tori}

Suppose that $\bT$ is an $F$-torus, with cocharacter lattice $Y=\Hom(\bG_m,\bT_{F_s})$ and character lattice $X=\Hom(\bT_{F_s},\bG_m)$.

\begin{Prop}
Let $\bT$ be a $F$-torus. The category $\CExt(\bT,\bK_2)$ is equivalent as a commutative Picard category (by an explicit functor) to the category whose objects are pairs $(Q,\sce)$, where
\begin{itemize}
\item $Q$ is a $\Gamma$-invariant $\bZ$-valued quadratic form on $Y$, with associated symmetric bilinear form
\[
B_Q(y_1,y_2)=Q(y_1+y_2)-Q(y_1)-Q(y_2);
\]
\item $\sce$ is a $\Gamma$-equivariant central extension of groups
\[
1\to F_s^\times \to \sce \to Y \to 1
\]
whose associated commutator map $[-.-]:Y\times Y\to F_s^\times$ is given by
\[
[y_1,y_2]=(-1)^{B_Q(y_1,y_2)}.
\]
The set of morphisms between $(Q,\sce)$ and $(Q',\sce')$ is empty unless $Q=Q'$, in which case it is given by the set of isomorphism of $\Gamma$-equivariant extensions from $\sce$ to $\sce'$. The Picard structure is defined by
\[
(Q,\sce)+(Q', \sce') = (Q+Q',\text{ Baer sum of } \sce \text{ and } \sce').
\]
\end{itemize}
\end{Prop}

The isomorphism class of the extension $\sce$ is completely determined by the commutator map and hence by the quadratic form $Q$. The extension $\sce$ is obtained from $\obT$ as follows. Let $F_s(\!(\tau)\!)$ denote the field the Laurent seires in the variable $\tau$ over $F_s$. Then one has
\[
1\to \bK_2(F_s(\!(\tau)\!)) \to \obT(F_s(\!(\tau)\!)) \to \bT(F_s(\!(\tau)\!)) = Y\otimes_{\bZ} F_s(\!(\tau)\!)^{\times}\to 1.
\]
The map $y\mapsto y(\tau)$ defines a group homomorphism $Y\to \bT(F_s(\!(\tau)\!))$. Pulling back by this morphism and pushing out by the residue map
\[
\mathrm{Res}:\bK_2(F_s(\!(\tau)\!)) \to \bK_1(F_s) = F_s^\times
\]
defined by
\[
\mathrm{Res}(f,g)=(-1)^{v(f)\cdot v(g)} \cdot \left(\dfrac{f^{v(g)}}{g^{v(f)}}(0)\right),
\]
one obtains the desired extension $\sce$.

\subsection{Simply-connected groups}

Suppose now that $\bG$ is a simply-connected semisimple gorup over $F$. Since $\bG$ is simply-connected, we have $Y^{sc}=Y$ and $\Delta^\vee$ is a basis for $Y$.

\begin{Prop}
The category $\CExt(\bG,\bK_2)$ is equivalent to (as commutative Picard categories) to the category whose objects are $\Gamma$-invariant $W$-invariant $\bZ$-valued quadratic form $Q$ on $Y$, and whose only morphisms are the identity morphisms on each object.
\end{Prop}

We have the following (\cite{Weissman11} Proposition 3.15).

\begin{Prop}
There is a unique $W$-invariant quadratic form $Q$ on $Y$, such that $Q(\al^\vee)=1$ for every short coroot $\al^\vee \in Y$, and every integer multiple of this quadratic form is $\Gamma$-invariant.
\end{Prop}

As a result, whenever we are given such a quadratic form $Q$ on $Y$, $Q$ gives rise to a multiplicative $\bK_2$-torsor $\obG_Q$ on $\bG$, unique up to unique isomorphism, which may be pulled back to a multiplicative $\bK_2$-torsor $\obT_Q$ on $\bT$ and hence gives rise to an extension $\sce_Q$ of $Y$ by $F_s^\times$. The automorphism group of the extension $\sce_Q$ is $\Hom_{\Gamma}(Y,F_s^\times)$. Following \cite{BD01} Section 11 and \cite{GG18} Section 2.4, one can rigidify $\sce_Q$ by giving it an extra structure.

\subsection{The Brylinski-Deligne liftings}

We assume that $\bG$ is simply-connected. We also have a fixed Chevalley system of \'epinglage for $(\bG_{F_s},\bT_{F_s})$. In particular, for each $\al\in \Phi$ with associated root subgroup $\bU_\al$, there is a fixed isomorphism over $F_s$:
\[
x_\al: \bG_a \to \bU_\al \subset \bG.
\]
Indeed, one has an embedding over $F_s$
\[
i_\al: \bSL_2 \hookrightarrow \bG
\]
which restricts to $x_{\pm\al}$ on the upper and lower triangular subgroup of unipotent matrices. By \cite{BD01}, one has a canonical lifting
\[
\tilde x_\al : \bG_a \to \overline{\bU}_\al \subset \obG.
\]
For $t\in \bG_m$, we set
\[
n_\al(t)=x_\al(t)\cdot x_{-\al}(-t^{-1}) \cdot x_\al(t)
=i_\al
\begin{pmatrix} 0 & t \\ -t^{-1} & 0 \end{pmatrix} \in N(\bT_Q),
\]
and
\[
\tilde n_\al(t)=\tilde x_\al(t)\cdot \tilde x_{-\al}(-t^{-1}) \cdot \tilde x_\al(t)\in \obG_Q.
\]
Then one has a map
\[
s_\al: \bT_\al := \al^\vee(\bG_m) \to \obT_Q
\]
given by
\[
\al^\vee(t) \mapsto \tilde n_\al(t) \cdot \tilde n_\al(-1).
\]
This is a section of $\obG_Q$ over $\obT_\al$, which is trivial at the identity element. This section is useful in describing the natural conjugation action of $N(\bT_Q)$ on $\obT_Q$. By \cite{BD01} Proposition 11.3, one has the nice formula:
\begin{equation}\label{eq:conjugation nice formula}
\tilde n_\al(1) \cdot \tilde t \cdot \tilde n_\al(1)^{-1}=\tilde t\cdot s_\al(\al^\vee(\al(t)^{-1})).
\end{equation}
We also use the following formula frequently (\cite{Gao18} (2)): for $t_1,t_2 \in \bG_m$,
\begin{equation}\label{eq:gao first formula}
s_\al(\al^\vee(t_1)) \cdot s_\al(\al^\vee(t_2)) = s_\al(\al^\vee(t_1 t_2))  \cdot \{t_1,t_2\}^{Q(\al^\vee)}.
\end{equation}
The collection of sections $\{s_\al: \al\in \Delta\}$ provides a collection of elements $s_\al(\al^\vee(a))\in \bT_Q$ with $a\in \bG_m$, and $\obT_Q$ is generated by $\bK_2$ and the collection of $s_\al(\al^\vee(a))$.

Taking points in $F(\!(\tau)\!)$, we have the element
\[
s_\al (\al^\vee(\tau)) \in \obT_Q(F(\!(\tau)\!)),
\]
which gives rise (via the construction of $\sce_Q$) to an element
\[
s_Q(\al^\vee) \in \sce_Q.
\]
Then we rigidify $\sce_Q$ by equipping it with the set $\{s_Q(\al^\vee)\mid \al^\vee \in \Delta^\vee\}$: there is a unique automorphism of $\sce_Q$ which fixes all these elements.

In the following, we shall fix a choice of the data $(\obG_Q,\obT_Q,\sce_Q)$ for each $\Gamma$-invariant $W$-invariant quadratic form $Q$ on $Y$ when $\bG$ is simply-connected. 

\subsection{Weyl group action on Brylinski-Deligne liftings}

Observe that $s_Q(\al^\vee)$ can be defined for every coroot (not necessarily simple coroot). We would like know how $s_Q(\al^\vee)$ behaves under the action of the Weyl group. Recall that for every $\al\in \Phi$, one can choose $w_\al:=n_\al(1) \in \bG(F_s)$ as a representative of the simple reflection $\mathbf{w}_\al$ corresponding to $\al$. Let $\mathbf{w}=\mathbf{w}_{\al_1} \cdots \mathbf{w}_{\al_{\ell}}$ be a minimum decomposition of $\mathbf{w}$, we choose the following representative of $\mathbf{w}$:
\[
w=w_{\al_1}\cdots w_{\al_{\ell}} \in \bG(F_s).
\]
This is independent of the minimum decomposition of $\mathbf{w}$. We choose $\tilde w_\al = \tilde n_\al(1)$ as a lift of $w_\al$ in $\bG(F_s)$. In any case, the conjugation action of $\obG(F_s)$ on $\obG(F_s)$ descents to an action of $\bG(F_s)$.

Fix a pair of roots $\al$ and $\beta$. Then we have a homomorphism
\[
\Int(w_\al)\circ x_{\beta}=w_\al \cdot x_\beta \cdot w_\al^{-1}: \bG_a \to \bU_{\mathbf{w}_\al(\beta)}.
\]
From \cite{BT84} Section 3.2.2, there is a sign $\epsilon_{\al,\beta}\in \{\pm 1\}$ associated to the Chevalley system of \'epinglage such that
\[
w_\al \cdot  x_\beta(t) \cdot w_\al^{-1}=x_{\mathbf{w}_\al(\beta)}(\epsilon_{\al,\beta}t),\qquad t \in \bG_m.
\]
This implies that for $t\in \bG_m$,
\begin{equation}\label{eq:conjugation of unipotent by weyl}
\tilde w_\al \cdot  \tilde x_{\beta}(t) \cdot \tilde w_\al^{-1}= \tilde x_{\mathbf{w}_\al(\beta)}(\epsilon_{\al,\beta} t).
\end{equation}

\begin{Lem}\label{lem:weyl action on E I}
For $t\in \bG_m$,
\[
\tilde w_\al \cdot  s_{\beta}(\beta^{\vee}(t)) \cdot \tilde w_\al^{-1}= \{\epsilon_{\al,\beta},t\}^{Q(\beta^\vee)} \cdot s_{\mathbf{w}_\al(\beta)}(\mathbf{w}_\al(\beta)^\vee (t)).
\]
\end{Lem}

\begin{proof}
The calculation here is identical to \cite{Gao18} Corollary 7.4.
For $t\in \bG_m$, we have
\[
\tilde w_\al\cdot \tilde n_\beta(t)\cdot \tilde w_\al^{-1}=(\tilde w_\al\cdot \tilde x_\beta(t)\cdot  \tilde w_\al^{-1}) (\tilde w_\al\cdot \tilde x_{-\beta}(-t^{-1})\tilde w_\al^{-1}) (\tilde w_\al \cdot  \tilde x_\beta(t)\cdot \tilde w_\al^{-1}) .
\]
By \eqref{eq:conjugation of unipotent by weyl}, this is
\[
\tilde x_{\mathbf{w}_{\al}(\beta)}(\epsilon_{\al,\beta}t) \cdot \tilde x_{-\mathbf{w}_{\al}(\beta)}(-\epsilon_{\al,\beta}t^{-1}) \cdot \tilde x_{\mathbf{w}_{\al}(\beta)}(\epsilon_{\al,\beta}t) = \tilde n_{\mathbf{w}_\al(\beta)}(\epsilon_{\al,\beta}t).
\]
From this we deduce that for any $t\in \bG_m$,
\[
\tilde w_\al \cdot  \tilde n_\beta(t) \cdot \tilde n_\beta(-1) \cdot \tilde w_{\al}^{-1}=  \tilde n_{\mathbf{w}_\al(\beta)}(\epsilon_{\al,\beta} t) \cdot \tilde n_{\mathbf{w}_\al(\beta)}(-\epsilon_{\al,\beta})= \{\epsilon_{\al,\beta}, t\}^{Q(\beta^\vee)} s_{\mathbf{w}_\al(\beta)}(\mathbf{w}_\al(\beta)^\vee(t)).
\]
The last equality follows from \cite{Gao18} (3).
\end{proof}

\begin{Lem}\label{lem:formula on BD lift}
We have the following results.
\begin{enumerate}
\item If $\la \al,\beta^\vee\ra = -1$, we have
\[
s_Q(\al^\vee+\beta^\vee)=s_Q(\beta^\vee) \cdot s_Q(\al^\vee) \cdot \epsilon_{\al,\beta}^{Q(\beta^\vee)}.
\]
\item We have $s_Q(-\al^\vee) \cdot s_Q(\al^\vee)= 1$.
\end{enumerate}
\end{Lem}

\begin{proof}
On the one hand, $\tilde w_{\al} \cdot s_{\beta}(\beta^{\vee}(\tau)) \cdot \tilde w_{\al}^{-1} =s_{\mathbf{w}_\al(\beta)}(\mathbf{w}_\al(\beta)^\vee(\tau)) \cdot \{\epsilon_{\al,\beta},\tau\}^{Q(\beta^\vee)}$. On the other hand, from \eqref{eq:conjugation nice formula}, we obtain
\[
\begin{aligned}
& \tilde w_{\al} \cdot s_{\beta}(\beta^{\vee}(\tau)) \cdot \tilde w_{\al}^{-1} =s_{\beta}(\beta^{\vee}(\tau)) \cdot s_{\al}(\al^{\vee}(\al(\beta^\vee(\tau))^{-1})) \\
=&s_{\beta}(\beta^{\vee}(\tau))\cdot s_{\al}(\al^{\vee}(\tau^{-\la\al,\beta^\vee\ra}))=s_{\beta}(\beta^{\vee}(\tau))\cdot s_{\al}(\al^{\vee}(\tau)).
\end{aligned}
\]
Observe that under the map $\bK_2(F) \to \bK_1(F)=F^\times$, $\mathrm{Res}\{\epsilon_{\al,\beta},\tau\}=\epsilon_{\al,\beta}$. Now the first statement follows.

The argument for the second statement is similar. We also need the following consequences of \eqref{eq:gao first formula}:
\[
\begin{aligned}
s_\al(\al^{\vee}(\tau)) \cdot s_\al(\al^{\vee}(\tau^{-1}))  = & \{\tau,\tau^{-1}\}^{Q(\al^\vee)}. \\
s_\al(\al^\vee(\tau))\cdot s_\al(\al^\vee(\tau^{-2})) = & s_\al(\al^\vee(\tau^{-1})) \{\tau, \tau^{-2}\}^{Q(\al^\vee)}.
\end{aligned}
\]
Note that
\[
\tilde w_{\al} \cdot s_{\al}(\al^{\vee}(\tau)) \cdot \tilde w_{\al}^{-1} =s_{-\al}(-\al^\vee(\tau)) \cdot \epsilon_{\al,-\al}^{Q(\al^\vee)}
=s_{-\al}(-\al^\vee(\tau)) \cdot (-1)^{Q(\al^\vee)}.
\]
Again, by \eqref{eq:conjugation nice formula},
\[
\tilde w_{\al} \cdot s_{\al}(\al^{\vee}(\tau)) \cdot \tilde w_{\al}^{-1} = s_{\al}(\al^{\vee}(\tau)) \cdot s_\al (\al^\vee(\tau^{-2})) = s_\al (\al^\vee(\tau^{-1}))\cdot \{\tau, \tau^{-2}\}^{Q(\al^\vee)}.
\]
From these two equations, we deduce that
\[
s_{-\al}(-\al^\vee(\tau)) \cdot (-1)^{Q(\al^\vee)}=s_\al (\al^\vee(\tau^{-1}))\cdot \{\tau, \tau^{-2}\}^{Q(\al^\vee)}.
\]
We now have
\[
s_{-\al}(-\al^\vee(\tau))\cdot s_\al(\al^\vee(\tau))= \{\tau,\tau^{-1}\}^{Q(\al^\vee)} \cdot (-1)^{Q(\al^\vee)} \cdot \{\tau, \tau^{-2}\}^{Q(\al^\vee)}.
\]
From this we deduce that
\[
s_Q(\al^\vee) \cdot s_Q(-\al^\vee) = 1.
\]
\end{proof}

We write $\al^\vee$ as a sum of simple coroots:
\[
\al^\vee = \al_{i_1}^\vee + \cdots  + \al_{i_n}^\vee.
\]
We say that this expression has property $(\ast)$ if the following holds:
\begin{equation}\label{eq:condition star}
(\ast): \al_{i_1}^\vee, \ \al_{i_1}^\vee+\al_2^\vee, \  \cdots, \ \al_{i_1}^\vee + \cdots  + \al_{i_n}^\vee \in \Phi^\vee.
\end{equation}

\begin{Lem}\label{lem:additivity of sQ}
Assume that $\epsilon_{\al,\beta}^{Q(\beta^\vee)}=1$ and $\la \al,\beta^\vee\ra = -1$ for all $\al,\beta$. If we write $\al^\vee$ as a sum of simple coroots such that property ($\ast$) holds, then
\[
s_{Q}(\al^\vee)=s_{Q}(\al_{i_n}^\vee)\cdots s_{Q}(\al_{i_2}^\vee) \cdot s_{Q}(\al_{i_1}^\vee).
\]
\end{Lem}

\begin{proof}
This follows from Lemma \ref{lem:formula on BD lift} and property $(\ast)$ by induction.
\end{proof}

\subsection{General reductive groups}

Let $\bG$ be a connected reductive group over $F$, with a fixed Chevalley system of \'epinglage for $(\bG,\bT)$. Let $X^{sc}\subset X\otimes_{\bZ} \bq$ be the dual lattice of $Y^{sc}$. Then the quadruple $(X^{sc}, \Delta, Y^{sc}, Y)$ is the root system of the simply-connected cover $\bG^{sc}$ of the derived group $\bG^{der}$ of $\bG$, and one has a natural map
\[
q:\bG^{sc}\to \bG^{der} \to \bG.
\]
Let $\bT^{sc}$ be the preimage of $\bT$ in $\bG^{sc}$. It is a maximal $F$-torus of $\bG^{sc}$ with cocharacter group $Y^{sc}\subset Y$.  so that one has a commutative diagram
\[
\begin{tikzpicture}[scale=1.5]
\node (A1) at (0,1) {$\bT^{sc}$};
\node (B1) at (1.5,1) {$\bG^{sc}$};
\node (A2) at (0,0) {$\bT$};
\node (B2) at (1.5,0) {$\bG$};
\path[->,font=\scriptsize,>=angle 90]
(A1) edge  (B1)
(A2) edge  (B2)
(A1) edge (A2)
(B1) edge (B2);
\end{tikzpicture}
\]
The restriction $Q^{sc}:=Q|_{Y^{sc}}$ gives an element $\obG^{sc}\in \CExt(\bG^{sc},\bK_2)$. It also gives the extension $\sce_{Q|_{Y^{sc}}}^{sc}$. For simplicity, we just write $\sce_{Q^{sc}}$ with no confusion caused.

\begin{Thm}
The category $\CExt(\bG,\bK_2)$ is equivalent to the category $\BD(\bG,\bT)$ whose objects are triples $(Q,\sce,f)$, where
\begin{itemize}
\item $Q:Y\to \bZ$ is a $\Gamma$-invariant $W$-invariant quadratic form;
\item $\sce$ is a $\Gamma$-equivariant extension of $Y$ by $F_s^\times$ with commutator map $[y_1,y_2]=(-1)^{B_Q(y_1,y_2)}$;
\item $f$ is a $\Gamma$-equivariant morphism from $\sce_{Q^{sc}}$ to $\sce$ such that the following diagram commute:
    \begin{equation}\label{eq:third invariant}
\begin{tikzpicture}[scale=1.5]
\node (A1) at (0,1) {$1$};
\node (B1) at (1.5,1) {$F_s^\times$};
\node (C1) at (3.5,1) {$\sce_{Q^{sc}}$};
\node (D1) at (5.5,1) {$Y^{sc}$};
\node (E1) at (7,1) {$1$};
\node (A2) at (0,0) {$1$};
\node (B2) at (1.5,0) {$F_s^\times$};
\node (C2) at (3.5,0) {$\sce$};
\node (D2) at (5.5,0) {$Y$};
\node (E2) at (7,0) {$1$};
\path[->,font=\scriptsize,>=angle 90]
(A1) edge  (B1)
(B1) edge  (C1)
(C1) edge  (D1)
(D1) edge  (E1)
(A2) edge  (B2)
(B2) edge  (C2)
(C2) edge  (D2)
(D2) edge  (E2)
(C1) edge (C2)
(D1) edge (D2)
(B1) edge (B2);
\end{tikzpicture}
\end{equation}
\end{itemize}
The set of morphisms from $(Q,\sce,f)$ to $(Q',\sce',f')$ is empty unless $Q=Q'$, in which case it consists of $\Gamma$-equivariant isomorphisms of extensions $\phi:\sce\to \sce'$ such that $f=f'\circ \phi$. 
\end{Thm}

\section{Topological covering groups}\label{sec:topological covering}

We now pass from the algebro-geometric world of multiplicative $\bK_2$-torsors to the world of topological central extensions. We first assume that $F$ is a local field. If $F$ is non-Archimedean, let $\sco$ denote its ring of integers and let $p$ be the residue characteristic.

\subsection{BD covering groups}

Start with a multiplicative $\bK_2$-torsor $\obG$ on $\bG$, with associated BD data $(Q,\sce,f)$. Since $H^1(F,\bK_2)=0$, by taking $F$-points, we obtain a short exact sequence of abstract groups
\[
1\to \bK_2(F) \to \obG(F) \to G=\bG(F) \to 1.
\]
Let $\mu(F)$ denotes the set of roots of unity contained in the local field $F\neq \bc$; when $F=\bc$, we let $\mu(F)$ to be the trivial group. Then the Hilbert symbol gives a map
\[
(-,-)_F: \bK_2(F) \to \mu(F).
\]
For any $n$ dividing $\# \mu(F)$, one has the $n$-th Hilbert symbol
\[
(-,-)_n : (-,-)_F^{\#\mu(F)/n}: \bK_2(F) \to \mu_n(F).
\]
By pushing out the exact sequence via the Hilbert symbol $\bK_2(F)\to \mu_n(F)$, we obtain an exact sequence of locally compact topological groups
\[
1\to \mu_n(F) \to \oG \to G \to 1.
\]
We call this the BD covering group associated to the BD data $(Q,\sce,f,n)$.

\subsection{Unipotent subgroups}

Let $\mathcal{N}_G$ be the set of all unipotent elements of $G$. Because a BD extension is uniquely split over any unipotent subgroup, one has unique splittings:
\[
\tilde x_\al: F\to \overline{U}_\al \text{ for each }\al \in \Phi.
\]
Indeed, as shown in \cite{MW95} Appendix I and \cite{Li14} Proposition 2.2.1, there is a unique section
\[
i:\mathcal{N}_G \to \oG
\]
satisfying:
\begin{itemize}
\item for each unipotent subgroup $\bU\subset \bG$, the restriction of $i$ to $U=\bU(F)$ is a group homomorphism;
\item the map $i$ is $G$-equivariant.
\end{itemize}

\subsection{Tori}

The following result is a consequence of \cite{BD01} Proposition 3.13. 

\begin{Prop}\label{prop:commutator of tori}
Let $L$ be any field containing $F$ over which $\bT$ splits. Let $\overline{\bT}(L)$ be the resulting central extension
\[
1\to \bK_2(L)\to \overline{\bT}(L)\to \bT(L)\to 1.
\]
Then the commutator of this extension satisfies
\[
\mathrm{Comm} (y_1(u_1), y_2(u_2))=\{u_1,u_2\}^{B_Q(y_1,y_2)},
\]
for all $y_1,y_2\in \Hom(\bG_m,\bT_L)$ and all $u_1,u_2\in L^{\times}$.
\end{Prop}

We would like to note the following useful observation. If $\bT=\bT_1\times \bT_2$, then $Y=Y_1\oplus Y_2$. We have the following consequence.

\begin{Lem}\label{lem:commute for torus}
If $B_Q(y_1,y_2)=0$ for all $y_1\in Y_1, y_2\in Y_2$, then $\oT_1$ and $\oT_2$ commute in $\oT$.
\end{Lem}

\begin{proof}
This is an immediate consequence of Proposition \ref{prop:commutator of tori}.
\end{proof}

Let $\bG=\bG_1\times \bG_2$ with maximal torus $\bT=\bT_1\times \bT_2$. Then there is a corresponding decomposition of cocharacter lattice $Y=Y_1\oplus Y_2$. Let $\obG\in \CExt(\bG,\bK_2)$. This gives $\obG_i \in \CExt(\bG_i,\bK_2)$ for $i=1,2$. We have inclusions $\oG_i\to \oG, i= 1, 2$.

\begin{Lem}\label{lem:commute for G}
If $B_Q(y_1,y_2)=0$ for all $y_1\in Y_1, y_2\in Y_2$, then $\oG_1$ and $\oG_2$ commute in $\oG$.
\end{Lem}

\begin{proof}
The group $\oG$ is generated by $\oT$ and $\mathcal{N}_G$. We only have to verify the following:
\begin{enumerate}
\item $\oT_1$ and $\oT_2$ commute;
\item $\oT_1$ and $\mathcal{N}_{G_2}$ commute;
\item $\oT_2$ and $\mathcal{N}_{G_1}$ commute;
\item $\mathcal{N}_{G_1}$ and $\mathcal{N}_{G_2}$ commute;
\end{enumerate}

The first statement is simply Lemma \ref{lem:commute for torus}. The second follows from the fact that $T_1$ and $\mathcal{N}_{G_2}$ commute in $G$, and the unipotent section is $G$-equivariant. The rest is similar.
\end{proof}

\subsection{The tame case}

We now discuss the splitting of maximal compact subgroups at unramified places.

Let $F$ be a non-Archimedean field with ring of integers $\sco$.
If $\bG$ is an unramified reductive group over $F$. Suppose that the group $\bG$ has an integral model $\underline{G}$ over $\sco$. The $\bK_2$-extension $\obG$ might not be defined over $\sco$. If it is, then there is a natural splitting of $\oG$ over $K=\bG(\sco)$.

The $\bK_2$-extension $\obG$ yields a short exact sequence
\[
1\to \bK_2(F)\to \obG(F)\to \bG(F)\to 1.
\]
If $p\nmid n$, then this ``tameness'' gives an exact sequence
\[
1\to \bK_2(\sco)\to \bK_2(F)\xrightarrow{\mathrm{Hilb}_n}\mu_n\to 1.
\]
This gives a commutative diagram
\[
\begin{tikzpicture}[scale=1.5]
\node (A1) at (0,1) {$1$};
\node (B1) at (1.5,1) {$\bK_2(\sco)$};
\node (C1) at (3,1) {$\obG(\sco)$};
\node (D1) at (4.5,1) {$\bG(\sco)$};
\node (E1) at (6,1) {$1$};
\node (A2) at (0,0) {$1$};
\node (B2) at (1.5,0) {$\mu_n$};
\node (C2) at (3,0) {$\oG$};
\node (D2) at (4.5,0) {$G$};
\node (E2) at (6,0) {$1$};
\path[->,font=\scriptsize,>=angle 90]
(A1) edge  (B1)
(B1) edge  (C1)
(C1) edge  (D1)
(D1) edge  (E1)
(A2) edge  (B2)
(B1) edge (B2)
(B2) edge  (C2)
(C2) edge  (D2)
(D2) edge  (E2)
(C1) edge (C2)
(D1) edge (D2);
\end{tikzpicture}.
\]
Thus the central extension $1\to \mu_n\to \oG\to G\to 1$ is endowed with a splitting over the hyperspecial maximal compact subgroup $\bG(\sco)$.

\subsection{Adelic BD covering}

In this section, $F$ is a global field. For a place $v$ of $F$, we write $F_v$ for the completion of $F$ at $v$.

Starting with a BD extension $\obG$ over $\Spec(F)$ and a positive integer $n$ such that $|\mu_n(F)|=n$, Brylinski and Deligne showed using results of \cite{Moore68} that one inherits the following data:
\begin{itemize}
\item for each place $v$ of $k$, a local BD covering group $\oG_v$ of degree $n$;
\item for almost all $v$, a splitting $s_v:\bG(\sco_v)\to \oG_v$;
\item a restricted direct product $\prod'_v \oG_v$ with respect to the family of subgroups $s_v(\bG(\sco_v))$, from which one can define
    \[
    \oG(\ba):=\prod_v{}^{'} \oG_v/\{(\zeta_v)\in\oplus_v \mu_n(k_v):\prod_v \zeta_v=1\},
    \]
    which gives a topological central extension
    \[
    1\to \mu_n(F)\to \oG(\ba) \to G(\ba) \to 1,
    \]
    called the adelic or global BD covering group;
\item a natural inclusion
\[
\begin{tikzpicture}[scale=1.5]
\node (A1) at (0,1) {$1$};
\node (B1) at (1.5,1) {$\mu_n(F_v)$};
\node (C1) at (3,1) {$\oG_v$};
\node (D1) at (4.5,1) {$\bG(F_v)$};
\node (E1) at (6,1) {$1$};
\node (A2) at (0,0) {$1$};
\node (B2) at (1.5,0) {$\mu_n(F)$};
\node (C2) at (3,0) {$\oG(\ba)$};
\node (D2) at (4.5,0) {$\bG(\ba)$};
\node (E2) at (6,0) {$1$};
\path[->,font=\scriptsize,>=angle 90]
(A1) edge  (B1)
(B1) edge  (C1)
(C1) edge  (D1)
(D1) edge  (E1)
(A2) edge  (B2)
(B2) edge  (C2)
(C2) edge  (D2)
(D2) edge  (E2)
(C1) edge (C2)
(D1) edge (D2);
\draw (B1) edge[=] (B2);
\end{tikzpicture}
\]
for each place $v$ of $k$;
\item a natural splitting
\[
i:\bG(F)\to \oG(\ba),
\]
which allows one to consider the space of automorphic forms on $\oG(\ba)$.
\end{itemize}

In this paper, we fix an embedding $\epsilon: \mu_n\to \bc^\times$. We say a representation $\pi$ of $\oG(\ba)$ is $\epsilon$-genuine if $\mu_n$ acts via $\epsilon$.

We briefly recall how the splitting $i$ is obtained. Let $X=\Spec(\sco_F)$. Let $S_1$ is a finite set of finite places of $F$. We assume that $S_1$ is large enough so that the conclusion of \cite{BD01} 10.6 holds. Write $S=S_1\cup \{\text{infinite places}\}$. This gives a central extension
\[
1 \to \mathrm{H}^0(X-S_1,\bK_2) \to E_1 \to G(X-S_1) \to 1.
\]
For $v$ a place of $F$, it maps to the local central extension
\[
1 \to \mu_n \to \oG_v \to G_v \to 1.
\]
For an unramified place $v$, the map factors through a central extension
\[
1\to \bK_2(\sco_v) \to \overline{\bG(\sco_v)} \to \bG(\sco_v) \to 1.
\]
If $p\nmid n$,  the exact sequence
\[
1\to \bK_2(\sco_v)\to \bK_2(F_v)\to F_v^\times \to 1
\]
shows that $\bK_2(\sco_v)$ maps to trivially to $\mu_n$. We obtain a trivialization of $\oG_v$ over $\bG(\sco_v)$.

We now have a commutative diagram
\[
\begin{tikzpicture}[scale=2]
\node (A1) at (2.5,1) {$1$};
\node (B1) at (4,1) {$\mathrm{H}^0(X-S_1,\bK_2)$};
\node (C1) at (6,1) {$E_1$};
\node (D1) at (8.5,1) {$G(X-S_1)$};
\node (E1) at (10.5,1) {$1$};
\node (A2) at (2.5,0) {$1$};
\node (B2) at (4,0) {$\prod\limits_{v\in S}\mu_v$};
\node (C2) at (6,0) {$\prod\limits_{v\notin S}\oG_v \times \prod\limits_{v\in X-S_1}\bG(\sco_v)$};
\node (D2) at (8.5,0) {$\prod\limits_{v\notin S}G_v \times \prod\limits_{v\in X-S_1}\bG(\sco_v)$};
\node (E2) at (10.5,0) {$1$};
\path[->,font=\scriptsize,>=angle 90]
(A1) edge  (B1)
(B1) edge  (C1)
(C1) edge  (D1)
(D1) edge  (E1)
(A2) edge  (B2)
(B2) edge  (C2)
(C2) edge  (D2)
(D2) edge  (E2)
(C1) edge (C2)
(D1) edge (D2);
\draw (B1) edge[=] (B2);
\end{tikzpicture}
\]
provided that for all $v$ in $X-S_1$, $p\nmid n$. This holds for $S_1$ is large enough. The first vertical map, composed with the reciprocity map $\prod \zeta_v$ with values in $\mu_n$, vanishes. We hence obtain
\[
\begin{tikzpicture}[scale=2]
\node (D1) at (8,1) {$G(X-S_1)$};
\node (A2) at (2,0) {$1$};
\node (B2) at (3,0) {$\mu_n$};
\node (C2) at (5,0) {$\overline{\prod\limits_{v\notin S}G_v} \times \prod\limits_{v\in X-S_1}\bG(\sco_v)$};
\node (D2) at (8,0) {$\prod\limits_{v\notin S}G_v \times \prod\limits_{v\in X-S_1}\bG(\sco_v)$};
\node (E2) at (10,0) {$1$};
\path[->,font=\scriptsize,>=angle 90]
(A2) edge  (B2)
(B2) edge  (C2)
(C2) edge  (D2)
(D2) edge  (E2)
(D1) edge (C2)
(D1) edge (D2);
\end{tikzpicture}
\]
Taking direct limit over $S$ gives the desired natural splitting:
\[
\begin{tikzpicture}[scale=1.5]
\node (D1) at (6.5,1) {$G(F)$};
\node (A2) at (2,0) {$1$};
\node (B2) at (3.5,0) {$\mu_n$};
\node (C2) at (5,0) {$\oG(\ba)$};
\node (D2) at (6.5,0) {$G(\ba)$};
\node (E2) at (8,0) {$1$};
\path[->,font=\scriptsize,>=angle 90]
(A2) edge  (B2)
(B2) edge  (C2)
(C2) edge  (D2)
(D2) edge  (E2)
(D1) edge (C2)
(D1) edge (D2);
\end{tikzpicture}
\]

\section{Pullback, pushout and Baer sum}\label{sec:pullback pushout baer sum}

We now discuss several constructions that give new exact sequences: pullback, pushout and the Baer sum. In this section we would like to describe these constructions in terms of the BD data.

\subsection{Pushout}

We now recall the definition of pushout.
\begin{Def}
For a central extension
\[
1\to A\xrightarrow{i} E\xrightarrow{p} G\to 1,
\]
and a homomorphism $f:A\to B$ of abelian groups, we define
\[
f_\ast(E):=(B\times E)/\la (f(a), i(a)^{-1}): a\in A \ra.
\]
The maps $B\to f_\ast(E), b\mapsto (b,1)$ and $f_{\ast}(E) \to G, (b,e)\mapsto p(e)$ define an exact sequence
\[
1\to B \to f_\ast(E) \to G \to 1.
\]
This exact sequence is called the pushout by $f$.
\end{Def}

\subsection{Baer sum}

Another method to construct new exact sequences is the Baer sum. In this paper, we only consider the Baer sum of $n$ copies of an exact sequence.

Given an exact sequence
\begin{equation}\label{eq:exact sequence in pushout}
1\to A \to E \to G \to 1
\end{equation}
with $A$ abelian. By taking the direct sum of $n$ copies of the exact sequence \eqref{eq:exact sequence in pushout}, we obtain
\[
1\to \oplus_{i=1}^n A \to \oplus_{i=1}^n E \to \oplus_{i=1}^n G \to 1.
\]
By pushing out the exact sequence via the product map
\[
\pr:\prod_{i=1}^n A\to A, \qquad (x_i)\mapsto \prod_{i=1}^n x_i,
\]
we obtain an exact sequence:
\[
1\to A \to \pr_\ast(\oplus_{i=1}^n E) \to \oplus_{i=1}^n G \to 1.
\]
Now we pull back this exact sequence via the diagonal map
\[
d:G \to \oplus_{i=1}^n G,\qquad x\mapsto (x,\cdots,x)
\]
to obtain
\[
1\to A \to d^\ast(p_{\ast} \oplus_{i=1}^n E) \to G \to 1.
\]
This exact sequence is the Baer sum of $n$ copies of \eqref{eq:exact sequence in pushout}.

We now claim that this is also the same as pushing out \eqref{eq:exact sequence in pushout} by the map $[n]: A\to A, x\mapsto x^n$.

\begin{Lem}
The following commutative diagram gives an isomorphism of exact sequences:
\[
\begin{tikzpicture}[scale=1.5]
\node (A1) at (0,1) {$1$};
\node (B1) at (1.5,1) {$A$};
\node (C1) at (3,1) {$[n]_\ast(E)$};
\node (D1) at (4.5,1) {$G$};
\node (E1) at (6,1) {$1$};
\node (A2) at (0,0) {$1$};
\node (B2) at (1.5,0) {$A$};
\node (C2) at (3,0) {$d^\ast(p_{\ast} \oplus_{i=1}^n E)$};
\node (D2) at (4.5,0) {$G$};
\node (E2) at (6,0) {$1$};
\path[->,font=\scriptsize,>=angle 90]
(A1) edge  (B1)
(B1) edge  (C1)
(C1) edge  (D1)
(D1) edge  (E1)
(A2) edge  (B2)
(B2) edge  (C2)
(C2) edge  (D2)
(D2) edge  (E2)
(C1) edge (C2)
(D1) edge (D2)
(B1) edge (B2);
\end{tikzpicture}
\]
\end{Lem}

\begin{proof}
Recall that $[n]_\ast(E)=(A\times E)/\la x^n,i^{-1}(x)\mid x\in A \ra$ and \[
p_{\ast} (\oplus_{i=1}^n E )= A \times (\oplus_{i=1}^n E)/\la \prod_{i=1}^n x_i,(i^{-1}(x_1),\cdots,i^{-1}(x_n)) \mid  (x_1,\cdots,x_n)\in \oplus_{i=1}^n A \ra
\]
We now define $[n]_\ast(E) \to d^\ast(p_{\ast} \oplus_{i=1}^n E)$ by
\[
(a,e)\mapsto (a,(e,\cdots,e)).
\]
It is straightforward to check that this is well-defined and is an isomorphism of exact sequences.
\end{proof}

\subsection{Functoriality of pullback}

Let $f:\bG\to \bH$ be a morphism of connected reductive groups. Let
\[
1\to \bK_2\to \obH\to \bH\to 1
\]
be a multiplicative $\bK_2$-torsor on $\bH$. By pulling back via $f$, we obtain multiplicative $\bK_2$-torsor on $\bG$:
\[
1\to \bK_2\to f^\ast(\obH)\to \bG\to 1.
\]
For ease of notations, let us write $\obG=f^{\ast}(\obH)$. Thus, this fits into a commutative diagram
\[
\begin{tikzpicture}[scale=1.5]
\node (A1) at (0,1) {$1$};
\node (B1) at (1.5,1) {$\bK_2$};
\node (C1) at (3,1) {$\obG$};
\node (D1) at (4.5,1) {$\bG$};
\node (E1) at (6,1) {$1$};
\node (A2) at (0,0) {$1$};
\node (B2) at (1.5,0) {$\bK_2$};
\node (C2) at (3,0) {$\obH$};
\node (D2) at (4.5,0) {$\bH$};
\node (E2) at (6,0) {$1$};
\path[->,font=\scriptsize,>=angle 90]
(A1) edge  (B1)
(B1) edge  (C1)
(C1) edge  (D1)
(D1) edge  (E1)
(A2) edge  (B2)
(B2) edge  (C2)
(C2) edge  (D2)
(D2) edge  (E2)
(C1) edge (C2)
(D1) edge (D2);
\draw (B1) edge[=] (B2);
\end{tikzpicture}
\]
and gives a functor
\begin{equation}\label{eq:functor of pullback}
\CExt(\bH,\bK_2)\to \CExt(\bG,\bK_2).
\end{equation}
At every local place $v$, the pullback determines the following data:
\begin{itemize}
\item At every local place $v$, we have a commutative diagram
\[
\begin{tikzpicture}[scale=1.5]
\node (A1) at (0,1) {$1$};
\node (B1) at (1.5,1) {$\mu_n$};
\node (C1) at (3,1) {$\oG_v$};
\node (D1) at (4.5,1) {$G_v$};
\node (E1) at (6,1) {$1$};
\node (A2) at (0,0) {$1$};
\node (B2) at (1.5,0) {$\mu_n$};
\node (C2) at (3,0) {$\oH_v$};
\node (D2) at (4.5,0) {$H_v$};
\node (E2) at (6,0) {$1$};
\path[->,font=\scriptsize,>=angle 90]
(A1) edge  (B1)
(B1) edge  (C1)
(C1) edge  (D1)
(D1) edge  (E1)
(A2) edge  (B2)
(B2) edge  (C2)
(C2) edge  (D2)
(D2) edge  (E2)
(C1) edge (C2)
(D1) edge (D2);
\draw (B1) edge[=] (B2);
\end{tikzpicture}
\]
\item The commutative diagram is compatible with the lift of unipotent elements. In other words, the following diagram commutes:
\[
\begin{tikzpicture}[scale=1.5]
\node (A1) at (0,1) {$\mathcal{N}_{G_v}$};
\node (B1) at (1.5,1) {$\oG_v$};
\node (A2) at (0,0) {$\mathcal{N}_{H_v}$};
\node (B2) at (1.5,0) {$\oH_v$};
\path[->,font=\scriptsize,>=angle 90]
(A1) edge  (B1)
(A2) edge  (B2)
(A1) edge (A2)
(B1) edge (B2);
\end{tikzpicture}
\]
\item In the tame case, the commutative diagram is compatible with the natural lift of maximal compact subgroups:
\[
\begin{tikzpicture}[scale=1.5]
\node (A1) at (0,1) {$\bG(\sco_v)$};
\node (B1) at (1.5,1) {$\oG_v$};
\node (A2) at (0,0) {$\bH(\sco_v)$};
\node (B2) at (1.5,0) {$\oH_v$};
\path[->,font=\scriptsize,>=angle 90]
(A1) edge  (B1)
(A2) edge  (B2)
(A1) edge (A2)
(B1) edge (B2);
\end{tikzpicture}
\]
\end{itemize}

We now move to the global setup. So from now on, $F$ is a global field. The local homomorphisms glue to
\[
\prod_v \oG_v \to \prod_v \oH_v.
\]
As $f_v(s_v(\bG(\sco_v)))\subset s_v(\bH(\sco_v))$, we obtain a homomorphism
\[
\prod_v{}' \oG_v\to \prod_v{}' \oH_v\to \oH(\ba).
\]
This map factors through
\[
f_\ba:\oG(\ba) \to \oH(\ba).
\]
From the construction of the natural splitting $H(F) \to \oH(\ba)$, it is not hard to check that this is compatible with the splitting over rational points. In other words, the diagram
\[
\begin{tikzpicture}[scale=1.5]
\node (A1) at (0,1) {$\bG(F)$};
\node (B1) at (1.5,1) {$\oG(\ba)$};
\node (A2) at (0,0) {$\bH(F)$};
\node (B2) at (1.5,0) {$\oH(\ba)$};
\path[->,font=\scriptsize,>=angle 90]
(A1) edge  (B1)
(A2) edge  (B2)
(A1) edge  (A2)
(B1) edge  (B2);
\end{tikzpicture}
\]
commutes.

We now describe the functor \eqref{eq:functor of pullback} in terms of the BD data. Let $\bT_G$ and $\bT_H$ be maximal $F$-tori of $\bG$ and $\bH$, respectively. We assume that $f(\bT_G) \subset \bT_H$. The map $\bT_G \to \bT_H$ induces a map $Y_G \to Y_H$, which gives $\sce_G$ as the pullback of $\sce_H$ via $Y_G\to Y_H$:
\begin{equation}\label{eq:pullback of second invariant}
\begin{tikzpicture}[scale=1.5]
\node (A1) at (0,1) {$1$};
\node (B1) at (1.5,1) {$F_s^\times$};
\node (C1) at (3,1) {$\sce_G$};
\node (D1) at (4.5,1) {$Y_G$};
\node (E1) at (6,1) {$1$};
\node (A2) at (0,0) {$1$};
\node (B2) at (1.5,0) {$F_s^\times$};
\node (C2) at (3,0) {$\sce_H$};
\node (D2) at (4.5,0) {$Y_H$};
\node (E2) at (6,0) {$1$};
\path[->,font=\scriptsize,>=angle 90]
(A1) edge  (B1)
(B1) edge  (C1)
(C1) edge  (D1)
(D1) edge  (E1)
(A2) edge  (B2)
(B2) edge  (C2)
(C2) edge  (D2)
(D2) edge  (E2)
(C1) edge (C2)
(D1) edge (D2);
\draw (B1) edge[=] (B2);
\end{tikzpicture}.
\end{equation}
The homomorphism $\bG \to \bH$ also determines a map $\bG^{sc} \to \bH^{sc}$. This gives a map $\sce_{\bG^{sc}} \to \sce_{\bH^{sc}}$. It is easy to check that the image of $\sce_{\bG^{sc}} \to \sce_{\bH^{sc}} \to \sce_H$ agrees with the image of $\sce_G \to \sce_H$, which gives a commutative diagram
\begin{equation}\label{eq:pullback of third invariant}
\begin{tikzpicture}[scale=1.5]
\node (A1) at (0,1) {$1$};
\node (B1) at (1.5,1) {$F_s^\times$};
\node (C1) at (3,1) {$\sce_{\bG^{sc}}$};
\node (D1) at (4.5,1) {$Y_G^{sc}$};
\node (E1) at (6,1) {$1$};
\node (A2) at (0,0) {$1$};
\node (B2) at (1.5,0) {$F_s^\times$};
\node (C2) at (3,0) {$\sce_G$};
\node (D2) at (4.5,0) {$Y_G$};
\node (E2) at (6,0) {$1$};
\path[->,font=\scriptsize,>=angle 90]
(A1) edge  (B1)
(B1) edge  (C1)
(C1) edge  (D1)
(D1) edge  (E1)
(A2) edge  (B2)
(B2) edge  (C2)
(C2) edge  (D2)
(D2) edge  (E2)
(C1) edge (C2)
(D1) edge (D2);
\draw (B1) edge[=] (B2);
\end{tikzpicture}.
\end{equation}

\begin{Prop}
With notations as above, the functor \eqref{eq:functor of pullback} can be described in terms of BD data:
\[
\BD(\bH,\bT_H) \to \BD(\bG,\bT_G),\qquad (Q_H,\sce_H,f_H)\mapsto (Q_G,\sce_G,f_G),
\]
where
\begin{itemize}
\item $Q_G=Q_H|_{Y_G}$;
\item $\sce_G$ is given by the top row of \eqref{eq:pullback of second invariant};
\item $f_G$ is given by the commutative diagram in \eqref{eq:pullback of third invariant}.
\end{itemize}
\end{Prop}

\begin{proof}
The quadratic form is determined by the commutator map on $\obT_G$. The other two invariants follow from their construction from $\obG$ directly.
\end{proof}

\subsection{Functoriality of pushout}

The pushout action is functorial so it can be glued to a construction of sheaves. Let $f\in \End(\bK_2)$. Then for a multiplicative $\bK_2$-torsor $\obG$, one can push it out via $f$ to obtain a new multiplicative $\bK_2$-torsor. In this paper, we consider the following map
\[
[m]:\bK_2\to \bK_2,\qquad x\mapsto x^m
\]
for an integer $m$.
We have a natural map $\obG\to \obG^\natural$ which fits into the commutative diagram
\[
\begin{tikzpicture}[scale=1.5]
\node (A1) at (0,1) {$1$};
\node (B1) at (1.5,1) {$\bK_2$};
\node (C1) at (3,1) {$\obG$};
\node (D1) at (4.5,1) {$\bG$};
\node (E1) at (6,1) {$1$};
\node (A2) at (0,0) {$1$};
\node (B2) at (1.5,0) {$\bK_2$};
\node (C2) at (3,0) {$\obG^\natural$};
\node (D2) at (4.5,0) {$\bG$};
\node (E2) at (6,0) {$1$};
\path[->,font=\scriptsize,>=angle 90]
(A1) edge  (B1)
(B1) edge  (C1)
(C1) edge  (D1)
(D1) edge  (E1)
(A2) edge  (B2)
(B2) edge  (C2)
(C2) edge  (D2)
(D2) edge  (E2)
(C1) edge (C2)
(D1) edge (D2)
(B1) edge (B2);
\end{tikzpicture}
\]
This defines a functor
\[
\CExt(\bG,\bK_2)\to \CExt(\bG,\bK_2).
\]
At every local place $v$, we obtain the following:
\begin{itemize}
\item We have the following commutative diagram:
\[
\begin{tikzpicture}[scale=1.5]
\node (A1) at (0,1) {$1$};
\node (B1) at (1.5,1) {$\mu_n$};
\node (C1) at (3,1) {$\oG_v$};
\node (D1) at (4.5,1) {$G_v$};
\node (E1) at (6,1) {$1$};
\node (A2) at (0,0) {$1$};
\node (B2) at (1.5,0) {$\mu_n$};
\node (C2) at (3,0) {$\oG^\natural_v$};
\node (D2) at (4.5,0) {$G_v$};
\node (E2) at (6,0) {$1$};
\path[->,font=\scriptsize,>=angle 90]
(A1) edge  (B1)
(B1) edge  (C1)
(C1) edge  (D1)
(D1) edge  (E1)
(A2) edge  (B2)
(B2) edge  (C2)
(C2) edge  (D2)
(D2) edge  (E2)
(C1) edge (C2)
(D1) edge (D2)
(B1) edge (B2);
\end{tikzpicture}.
\]
where the first vertical map is $x\mapsto x^m$.
\item the map $\mathcal{N}_{G_v} \to \oG_v \to \oG_v^{\natural}$ is the canonical unipotent section for $\oG_v^{\natural}$;
\item in the tame case, the map $\bG(\sco_v)\to \oG_v \to \oG_v^{\natural}$ is the natural splitting of maximal compact subgroups.
\end{itemize}

Globally, we can glue the local maps to obtain a global map. The map $G(F) \to \oG(\ba) \to \oG^{\natural}(\ba)$ is the natural splitting for the multiplicative $\bK_2$-torsor $\obG^\natural$.

We now describe the functor $\CExt(\bG,\bK_2)\to \CExt(\bG,\bK_2)$ in terms of BD data.

\begin{Prop}
The functor
\[
\CExt(\bG,\bK_2)\to \CExt(\bG,\bK_2), \qquad \obG \mapsto \obG^{\natural}
\]
in terms of the BD data is given by
\[
\BD(\bG,\bT)\mapsto \BD(\bG,\bT),\qquad (Q,\sce,f)\mapsto (Q^{\natural},\sce^{\natural},f^{\natural}),
\]
where
\begin{itemize}
\item $Q^{\natural}=mQ$;
\item $\sce^{\natural}$ is obtained from $\sce$ by pushing out via the map $[m]: F_s^\times \to F_s^\times, x\mapsto x^m$;
\item $f^{\natural}$ is obtained by pushing out the commutative diagram in \eqref{eq:third invariant} via the map $[m]$.
\end{itemize}
\end{Prop}

\begin{proof}
The functoriality of Baer-multiples $\obG \mapsto \obG^{\natural}$ can be found in \cite{Weissman16a} Theorem 2.2. Indeed, the quadratic form $Q^{\natural}$ is determined by the multiplication map on $\obT^{\natural}$. The other two invariants again follows from the construction directly. 
\end{proof}

Observe that if $m\equiv -1 \mod n$, then
\[
[m]:\bK_2 \to \bK_2, \qquad \zeta \mapsto \zeta^{m}.
\]
becomes $\mu_n\to \mu_n, \zeta \mapsto \zeta^{-1}$ after taking the Hilbert symbol. Thus,
\[
\oG_v^\natural =(\mu_n\times \oG_v)/\la (\zeta,\zeta):\zeta\in \mu_n\ra.
\]
We have a commutative diagram
\[
\begin{tikzpicture}[scale=1.5]
\node (A1) at (0,1) {$1$};
\node (B1) at (1.5,1) {$\mu_n$};
\node (C1) at (3,1) {$\oG_v$};
\node (D1) at (4.5,1) {$G_v$};
\node (E1) at (6,1) {$1$};
\node (A2) at (0,0) {$1$};
\node (B2) at (1.5,0) {$\mu_n$};
\node (C2) at (3,0) {$\oG_v^\natural$};
\node (D2) at (4.5,0) {$G_v$};
\node (E2) at (6,0) {$1$};
\path[->,font=\scriptsize,>=angle 90]
(A1) edge  (B1)
(B1) edge  (C1)
(C1) edge  (D1)
(D1) edge  (E1)
(A2) edge  (B2)
(B2) edge  (C2)
(C2) edge  (D2)
(D2) edge  (E2)
(C1) edge (C2)
(D1) edge (D2)
(B1) edge (B2);
\end{tikzpicture},
\]
where the first vertical map is the inverse map and the second vertical map $\oG_v\to \oG_v^\natural$ is a group isomorphism. This works globally as well.

\section{Doubling variables}\label{sec:doubling variables}

We now review some definitions related to classical groups that are used in the twisted doubling integrals. Here we use the conceptual description in \cite{Cai21}. 

\subsection{Classical groups}\label{sec:classical groups}

For the definition of classical groups, we follow the setup in \cite{Yamana14}. 

By an involution of an algebra $D$ whose center $E$ contains $F$, we mean an arbitrary anti-automorphism $\rho$ of $D$ of order two under which $F$ is the fixed subfield of $E$. We denote the restriction of $\rho$ to $E$ also by $\rho$. We take a couple of $(D,\rho)$ belonging to the following five types:

\begin{enumerate}[(a)]
  \item $D=E=F$ and $\rho$ is the identity map;
  \item $D$ is a division quaternion algebra over $E=F$ and $\rho$ is the main involution of $D$;
  \item $D$ is a division algebra central over a quadratic extension $E$ of $F$ and $\rho$ generates $\mathrm{Gal}(E/F)$;
  \item $D=\M_2(E)$, $E=F$ and $\begin{pmatrix} a & b \\ c & d \end{pmatrix}^\rho = \begin{pmatrix} d & -b \\ -c & a \end{pmatrix}$;
  \item $D=\bD\oplus \bD^{\op}$, $E=F\oplus F$ and $(x,y)^\rho=(y,x)$, where $\bD$ is a division algebra central over $F$ and $\bD^{\op}$ is its opposite algebra.
\end{enumerate}

If $F$ is local, we fix a nontrivial additive character $\psi_F$ of $F$; if $F$ is global, we fix a nontrivial additive character $\psi_F$ of $F\bs \ba$.
If $E=F$, we set $\psi=\psi_F$; if $E/F$ is an \'etale quadratic algebra, we set $\psi=\psi_F\circ \mathrm{tr}_{E/F}$. The global version is defined similarly. If $x$ is a square matrix with coordinates in $D$, then $\nu(x)\in E$ and $\tau(x)\in E$ stand for its reduced norm and reduced trace to the center $E$ of $D$.

The rank of $D$ as a module over $E$ is a square of a natural number which will be denoted by $d$. We assume $D$ to be division if $F$ is a number field, so that $D$ is of type (d) (resp. (e)) will appear in our later discussion as a localization of a global $D$ of type (b) (resp. (c)).

Let $\epsilon$ be either $1$ or $-1$. We fix once and for all the triple $(D,\rho,\epsilon)$.

Let $W$ be a free left $D$-module of rank $m$. By an $\epsilon$-skew hermitian space we mean a structure $\scw=(W,\la\ ,\ \ra)$, where $\la \ , \ \ra$ is a $\epsilon$-skew hermitian form on $W$, that is, an $F$-bilinear map $\la \ , \ \ra: W\times W\to D$ such that
\[
\la x, y\ra^\rho = -\epsilon\la y , x \ra,\qquad \la ax,by\ra=a\la x,y\ra b^\rho,  \ (a,b\in D; \ x,y \in W).
\]
Such a form is called non-degenerate if $\la x ,W\ra=0$ implies that $x=0$. We assume that $\la \ , \ \ra$ is non-degenerate.

We denote the ring of all $D$-linear endomorphisms of $W$ by $\End_D(W)$ and set $\bGL_D(W)=\End_D(W)^\times$. Note that $\bGL_D(W)$ acts on $W$ on the right. We sometimes write $\bGL_{W;D}$ for $\bGL_D(W)$ for ease of notations. 
Let
\[
\bG=
\{
g\in \bGL_D(W): \la xg,yg\ra = \la x,y\ra \text{ for all }x,y\in W
\}
\]
be the unitary group of $(W,\la \ , \ \ra)$, which is 
a reductive algebraic group defined over $F$. It is important to realize that $\bG$ always comes together with a space $W$ and a form $\la \ , \ \ra$. We usually just speak of $\bG$ and the data $\scw=(W,\la\ ,\ \ra)$ will be implicitly understood. We write $\bG=\bG(\scw)$ when the dependence of $\bG$ on $\scw$ needs to be stressed. 

\subsection{Doubling homomorphism}

Let $\scw=(W,\la \ , \ \ra)$ be one of the $\epsilon$-skew hermitian forms described above. Let $k$ be a fixed positive integer.  We would like to define the following in this section:
\begin{equation}\label{eq:doubling data}
(\bG,\bG^{\square,k},\iota,\bP, \bN_{\scw,k}^{\bullet}, \psi_{\scw,k}^{\bullet}).
\end{equation}

Put $W^{\square,k}=W^{\oplus 2k}$. We usually write
\[
W^{\square,k} = W_{1,+}\oplus W_{2,+}\oplus \cdots \oplus  W_{k,+}\oplus W_{k,-}\oplus \cdots \oplus W_{2,-} \oplus W_{1,-}
\]
to distinguish the copies of $W$ in $W^{\square,k}$. We write an element in $W^{\square,k}$ as
\[
(\bx;\by)=(x_1,\cdots, x_k; y_{k},\cdots, y_{1}),\qquad x_i\in W_{i,+}, \ y_i\in W_{i,-}.
\]
Define an $\epsilon$-skew hermitian form $\la \ , \ \ra^{\square,k} $ on $W^{\square,k}$ by
\[
\la (\bx;\by), (\bx';\by') \ra^{\square,k} = \sum_{i=1}^k(\la x_i,x'_i\ra-\la y_{i},y'_{i}\ra) \qquad (x_i,x'_i\in W_{i,+}; y_{i},y'_{i}\in W_{i,-}).
\]
Let $\bG^{\square,k}$ denote the unitary group of $(W^{\square,k}, \la \ , \ \ra^{\square,k})$.

For $W^\square =W_+\oplus W_-$, let
\[
W^\nabla=\{(x,-x)\in W_{+}\oplus W_{-}: x\in W\}
\]
be the graph of minus the identity map from $W$ to $W$, and
\[
W^{\Delta}=\{(x,x)\in W_{+}\oplus W_{-}: x\in W\}.
\]
be the graph of the identity map. Given $x\in W$, we write
\[
x^\Delta=(x ,x)\in W^\Delta \text{ and } x^\nabla=(x,-x)\in W^\nabla.
\]
We have the following observations:
\begin{enumerate}
\item For each $i$, $W_i^\square=W_{i,+}\oplus W_{i,-}=W^\Delta_{i}+ W_i^\nabla.$
Both $W_i^\Delta$ and $W_i^\nabla$ is totally isotropic in $W^{\square,k}$.
\item The space $W^\Delta$ is  isomorphic to $W$ as vector spaces via
\[
W^\Delta \simeq W, \qquad (x,x)\mapsto x.
\]
The space $W^\nabla$ is identified with $W$ via $(x,-x)\mapsto 2x$. Thus, we can view $\bG(\scw)$ as a subgroup of $\bGL_D(W^\Delta)$ or $\bGL_D(W^\nabla)$, and identify $\Hom_D(W_i^\nabla,W_j^\nabla)$ with $\End_D(W)$.
\end{enumerate}

Define
\[
W^{\Delta,k}=W_1^\Delta \oplus \cdots \oplus W_k^\Delta,\qquad W^{\nabla,k}=W_1^\nabla \oplus \cdots \oplus W_k^\nabla.
\]
Both spaces are totally isotropic in $W^{\square,k}$ and $W^{\square,k} =W^{\Delta,k}+W^{\nabla,k}$. This is a complete polarization of $W^{\square,k}$. Unless otherwise specified, we write $\bP=\bP(W^{\Delta,k})$.

We first construct a Fourier coefficient for the group $\bG^{\square,k}$.  
We choose the following flag of totally isotropic subspaces in $W^{\square,k}$
\begin{equation}\label{eq:flag in twisted doubling}
0\subset W_k^\nabla \subset W_{k-1}^\nabla \oplus W_{k}^\nabla \subset \cdots \subset W_2^\nabla \oplus \cdots \oplus W_{k}^\nabla.
\end{equation}
Let $\bP_{\scw,k}^{\bullet}=\bM_{\scw,k}^{\bullet}\cdot \bN_{\scw,k}^{\bullet}$ be the corresponding parabolic subgroup. Then
\[
\bM_{\scw,k}^{\bullet}\simeq \bGL_D(W_{k}^\nabla)\times \cdots \times \bGL_D(W_{2}^\nabla) \times \bG(\mathcal{W}_1^\square).
\]
The character is defined on the group $\bN_{\scw,k}^{\bullet}$.

We reindex the flag in \eqref{eq:flag in twisted doubling} as
\[
0\subset Y_1\subset \cdots \subset Y_{k-1}
\]
and extend it to
\[
0\subset  Y_1\subset \cdots \subset Y_{k-1} \subset Y_{k-1}^{\perp}\subset \cdots \subset Y_1^\perp\subset W^{\square,k}.
\]
Note that except $Y_{k-1}^\perp/Y_{k-1}=W_1^\square$, the quotient between two successive terms is isomorphic to either $W^\Delta$ or $W^\nabla$. For convenience, we write $Y_0=0$ and $Y_k:=Y_{k-1}^\perp$.

To describe the character of $\bN_{\scw,k}^{\bullet}(F)\bs \bN_{\scw,k}^{\bullet}(\ba)$, we have to specify elements
\[
A_i\in \Hom_D(Y_i/Y_{i-1},Y_{i+1}/Y_i)\simeq \End_D(W^\nabla),\qquad i=1,\cdots, k-2,
\]
and
\[
A_{k-1}\in \Hom_D(Y_{k-1}/Y_{k-2},Y_{k-1}^{\perp}/Y_{k-1})\simeq \Hom_D(W^\nabla,W^\square).
\]
We choose $A_1,\cdots,A_{k-2}$ to be the identity map in $\End_D(W^{\nabla})$. The map
\[
Y_{k-1}/Y_{k-2}\xrightarrow{A_{k-1}} Y_{k-1}^{\perp}/Y_{k-1}\xrightarrow{A_k} Y_{k-2}^{\perp}/Y_{k-1}^{\perp}
\]
is translated from
\[
W^\nabla\to W_+\oplus W_- \to W^\Delta,\qquad x^{\nabla}\mapsto (2x,0)\mapsto 2x^{\Delta}.
\]
Note that $A_{k}\circ A_{k-1}$ is an isomorphism.

An element $u\in \bN_{\scw,k}^{\bullet}(F)\bs \bN_{\scw,k}^{\bullet}(\ba)$ induces
\[
u_i:(Y_{i+1}/Y_i)\otimes (F\bs \ba) \to (Y_i/Y_{i-1})\otimes (F\bs \ba).
\]
Then we define
\[
\psi_{\scw,k}^{\bullet}:\bN_{\scw,k}^{\bullet}(F)\bs \bN_{\scw,k}^{\bullet}(\ba) \to \bc,\qquad u\mapsto \psi\left(\sum_{i=1}^{k-1}\tau(u_i\circ A_i)\right).
\]

Given $(g_1,g_2)\in \bG\times \bG$, we define its action on $W_{1,+}\oplus W_{2,+}\oplus \cdots \oplus  W_{k,+}\oplus W_{k,-}\oplus \cdots \oplus W_{2,-} \oplus W_{1,-}$ via
\[
(x_1,\cdots, x_k; y_{k},\cdots, y_2,y_{1})(g_1,g_2)=(x_1 g_1,\cdots, x_kg_1; y_{k}g_1,\cdots, y_2g_1,y_{1}g_2).
\]
This extends to an action of $\bG\times \bG$ on $W^{\square,k}$ and gives a map
\[
\iota=\iota_k: \bG\times \bG \to \bG^{\square,k}.
\]
It is in fact a homomorphism and in particular, the images of these two copies of $\bG$ commute in $\bG^{\square,k}$.
It is straightforward to check that $\iota(G\times G)$ lies in the stabilizer of $\psi_{\scw,k}^{\bullet}$ in $\bG^{\square,k}$.

Recall that for a subgroup $J$ of $G$, we define $J^{\diamondsuit}=\{(g,g)\in G\times G \mid g\in J\}$. We have the following results from \cite{Cai21}.

\begin{Lem}[\cite{Cai21} Lemma 5.1]
We have $\iota(\bG\times \bG)\cap \bP=\iota(\bG^\diamondsuit)$.
\end{Lem}

\begin{Lem}[\cite{Cai21} Lemma 5.2]\label{lem:modular quasicharacter}
The modular quasicharacter $\delta_{\iota(\bG\times \bG);\bN_{\scw,k}^{\bullet}}(\iota(g_1,g_2))=1$ for any $g_1,g_2\in \bG$.
\end{Lem}

\subsection{The case of special orthogonal groups}

We now discuss the case of special orthogonal groups. Since the group $\mathbf{O}(W)$ is disconnected, to consider multiplicative $\bK_2$-torsors, it would be better to consider its connected component $\bSO(W)$. We now explain the modifications in order to develop the twisted doubling integrals. For the group $\mathbf{O}(W)$ and a fixed positive integer $k$, we have defined a list of input $(\bG,\bG^{\square},\iota, \bP,\bN_{\scw,k}^{\bullet},\psi_{\scw,k}^{\bullet})$. We now explain how to define it for the group $\bSO(W)$.

The doubling homomorphism $\iota:\mathbf{O}(W) \times \mathbf{O}(W) \to \mathbf{O}(W^{\square,k})$ restricts to
\[
\iota: \bSO(W) \times \bSO(W) \to \bSO(W^{\square,k}).
\]
Note that $\bP, \bN_{\scw,k}^{\bullet} \subset \bSO(W^{\square,k})$. Thus, we can still use $\psi_{\scw,k}^{\bullet}$ and $\bN_{\scw,k}^{\bullet}$ to define a Fourier coefficient of an automorphic form on $\bSO(W^{\square,k})(\ba)$. Thus we take
\[
(\bSO(W), \bSO(W^{\square,k}),\iota, \bP,\bN_{\scw,k}^{\bullet}, \psi_{\scw,k}^{\bullet})
\]
to be the input in the case of special orthogonal groups. To unify our discussion, if $\bG=\bSO(W)$, then we take $\bG^{\square,k}=\bSO(W^{\square,k})$.

The case of inner forms of orthogonal groups case can be treated similarly.

\section{Degenerate representations}\label{sec:degenerate reps}

The purpose of this section is to discuss a family of representations that are used in the global zeta integrals. These can be viewed as the analog of the generalized Speh representations in the covering group setup. As indicated in \cite{Cai21}, in order to prove the global identity, one only has to use information on Fourier coefficients of these representations. We will discuss the conjectural construction of such representations in Section \ref{sec:construction of Speh}.

\subsection{Degenerate Whittaker models}

We first recall the definition of degenerate Whittaker models.  One can attach a degenerate Whittaker models for a Whittaker pair (see \cite{GGS17,MW87}). Given an admissible representation $\pi$, an important question is to find the largest nilpotent orbits that support degenerate Whittaker models for $\pi$. Locally, this determination is related to other nilpotent invariants such as the wave-front set. We refer the reader to the introduction of \cite{GGS17} for a comprehensive account of discussion.

In this paper, we consider only a subclass of degenerate Whittaker models for $\GL_{n,D}$. As explained in \cite{Cai21} Section 2.4, this is sufficient for determining nilpotent invariants and fits into our examples later. As unipotent subgroups split canonically over covering groups, these notions transfers from the linear case to the covering group case automatically. In the following, we only define these in the linear case.

Let $R=F$ if $F$ is a local field, and $R=\ba/F$ if $F$ is a number field. Fix a nontrivial additive character $\psi_F:R \to \bc^\times$. Let $D$ be a central division algebra as in Section \ref{sec:classical groups}. Let $W$ be a free left $D$-module of rank $m$ and consider the group $\bGL_D(W)$. Recall that we sometimes write $\bGL_{W;D}$ for $\bGL_D(W)$ for ease of notations.  Let
\[
\scy:0\subset Y_1\subset Y_2\subset \cdots \subset Y_k\subset W
\]
be a flag of distinct subspaces of $W$. We sometimes write $Y_0=\{0\}$ and $Y_{k+1}=W$ for convenience. The stabilizer of $\scy$ is a parabolic subgroup $\bP(\scy)=\bM(\scy)\cdot \bN(\scy)$ with Levi component $\bM(\scy)$. Then as algebraic groups,
\[
\bN(\scy)^{ab}\cong \prod_{i=1}^k \Hom_D(Y_{i+1}/Y_i,Y_i/Y_{i-1}),\qquad u\mapsto (u_i)_{i=1}^{k}.
\]
To give a character of $\bN(\scy)(R)$, we specify an element in
\[
\sca=(A_1,\cdots, A_k)\in \prod_{i=1}^k \Hom_D(Y_{i}/Y_{i-1},Y_{i+1}/Y_i).
\]
More concretely, given such an $\sca$, we define a character $\psi_{\sca}$ of $N(\scy)(R)$ by
\[
\psi_{\sca}(u)=\psi\left(\sum_{i=1}^k \tau(u_i\circ A_i)\right).
\]
Here, $\psi=\psi_F\circ \mathrm{tr}_{E/F}$. 

Assume now we have a pair $(\bN(\scy),\psi_\sca)$. Globally, for an irreducible automorphic representation $\pi$ of $\oGL_{W;D}(\ba)$, we define the $(N(\scy),\psi_\sca)$-Fourier coefficient of $\phi\in \pi$ as
\[
\phi^{N(\scy),\psi_\sca}(g)=\int\limits_{[N(\scy)]} \phi(ug)\psi_{\sca}(u) \ du.
\]
Locally, we consider the space $\Hom_{N(\scy)(F)}(\pi,\psi_{\sca})$ of $(N(\scy),\psi_{\sca})$-functional for an admissible representation $\pi$ of $\oGL_{W;D}(F)$. 

\subsection{Representations of type $(k,m)_D$}

The purpose of this section is to introduce the notion of representations of type $(k,m)_D$, both locally and globally. These representations are supported on a suitable nilpotent orbit and admits unique models of degenerate type. In the linear case and when $D$ is a field, the generalized Speh representations are examples of such representations.

In this section, we assume that $\dim_D W=km$.

\begin{Def}\label{def:orbit kn}
We say a pair $(N(\scy),\psi_{\sca})$ is in the orbit $(k^m)_D$ if $\scy$ is of the form 
\[
0\subset Y_1 \subset \cdots \subset Y_{k-1} \subset W
\] 
and for $i=1,\cdots, k-1$, $\dim_D Y_i=mi$ and $A_i$ is an isomorphism.
\end{Def}

The stabilizer of a coefficient in the orbit $(k^m)_D$ is isomorphic to $\GL_{m,D}$.

\begin{Def}
We say a pair  $(N(\scy),\psi_{\sca})$ lies in an orbit higher than $(k^m)_D$ if
\[
A_{i+k-1}\circ \cdots \circ A_{i}\neq 0
\]
for some $i$.
\end{Def}
Note that this implies that there are at least $k$ proper subspaces in the flag $\scy$.

\begin{Def}\label{def:local (k,n)}
We say an representation $\theta$ of a local group $\oGL_{W;D}$ is of type $(k,m)_D$ if the following two conditions holds:
  \begin{enumerate}
    \item For a pair $(N(\scy),\psi_{\sca})$ that lies in the orbit $(k^m)_D$,
        \[
        \dim\Hom_{N(\scy)}(\theta,\psi_{\sca})=1.
        \]
    \item For any pair $(N(\scy),\psi_{\sca})$ that lies in an orbit higher than $(k^m)_D$,
        \[
        \dim \Hom_{N(\scy)}(\theta,\psi_{\sca})=0.
        \]
  \end{enumerate}
\end{Def}

\begin{Rem}
By Frobenius reciprocity, $\Hom_{N(\scy)}(\theta,\psi_{\sca})\simeq \Hom_{\oGL_{W;D}}(\theta, \Ind_{N(\scy)}^{\oGL_{W;D}}(\psi_{\sca}))$. An element in the latter space is called a $(N(\scy),\psi_{\sca})$-model for $\theta$. For a representation $\theta$ of type $(k,m)_D$, we write $\Wh_{N(\scy),\psi_{\sca}}(\theta)$ for the image of a nonzero map in $\Hom_{\oGL_{W;D}} (\theta,\Ind_{N(\scy)}^{\oGL_{W;D}}(\psi_{\sca}))$.
\end{Rem}

\begin{Def}
  We say an irreducible automorphic representation $\theta$ of $\oGL_{W;D}(\ba)$ is of type $(k,m)_D$ if the following conditions holds:
  \begin{enumerate}
    \item The representation supports a nonzero $(N(\scy),\psi_{\sca})$-Fourier coefficient such that the pair is in the orbit $(k^m)_D$.
    \item For any pair $(N(\scy),\psi_{\sca})$ lies in an orbit higher than $(k^m)_D$, the $(N(\scy),\psi_A)$-Fourier coefficient vanishes identically.
    \item The local component $\theta_v$ is a representation of type $(k,m)_D$ for every place $v$.
  \end{enumerate}
We also say that the nilpotent orbit attached to $\theta$ is $(k^m)_D$ if only parts (1) and (2) hold.
\end{Def}

\subsection{Invariance under stabilizer}\label{sec:invariance under stab}

This section is a straightforward adaption of \cite{Cai21} Section 2.4.2 to the case of covering groups. We collect necessary results but omit the proofs.

We now assume that $\dim_D W=km$ and the representation $\theta$ of $\oGL_{W;D}$ is of type $(k,m)_D$. We consider the following situation:
\[
0\subset Y_1 \subset \cdots \subset Y_{k-1}\subset W
\]
such that $\dim_D Y_i=mi$, $A_2,\cdots, A_{k-1}$ are isomorphisms, the rank of $A_{1}$ is $a>0$ (which might not be of full rank).

The Fourier coefficients defined by such a pair enjoy an extra invariance property. We start with the case of $a=m$. Recall that the stabilizer $\St_{\sca}$ of a pair $(N(\scy),\psi_{\sca})$ that lies in the orbit $(k^m)_D$ is isomorphic to $\GL_{m,D}$. We start with the local version.

\begin{Lem}[\cite{Cai21} Lemma 2.14]\label{lem:invariance I local}
Let $\theta$ be an irreducible $\epsilon$-genuine admissible representation of $\oGL_{W;D}(F)$ that is of type $(k,m)_D$.
\begin{enumerate}
\item The stabilizer $\overline{\St}_{\sca}$ acts on $\Hom_{N(\scy)}(\theta,\psi_{\sca})$ via an $\epsilon$-genuine character $\overline{\chi}_\theta:\overline{\St}_{\sca}(F) \to \bc^\times$.
\item For $f\in \Wh_{N(\scy),\psi_{\sca}}(\theta)$,
\[
f(gh)=\overline{\chi}_{\theta}(g)f(h)
\]
for $g\in \overline{\St}_{\sca}(F)$ and $h\in \oGL_{W;D}(F)$.
\end{enumerate}
\end{Lem}

Here is the global version.

\begin{Lem}[\cite{Cai21} Lemma 2.15]\label{lem:invariance I global}
Let $\theta$ be an irreducible unitary $\epsilon$-genuine automorphic representation of $\GL_{W;D}(\ba)$. Then there is a character $\overline{\chi}_\theta: \overline{\St}_{\sca}(F\bs \ba) \to \bc^\times$ such that, for any $\phi\in \theta$,
\[
\phi^{N(\scy),\psi_{\sca}}(gh) =\overline{\chi}_{\theta}(g)\phi^{N(\scy),\psi_{\sca}}(h)
\]
for any $g\in \overline{\St}_{\sca}(F\bs \ba)$ and $h\in \oGL_{W;D}(\ba)$.
\end{Lem}

We now consider the case $a<m$. Define $S_{\sca}$ to be the subgroup of $M(\scy)$:
\[
N(\mathrm{Ker}(A_1))\times \{1\} \times\cdots \times \{1\} \subset \GL(Y_1)\times \GL(Y_2/Y_1) \times \cdots \times \GL(W/Y_{k-1}).
\]
Here, $N(\mathrm{Ker}(A_1))$ is the unipotent radical of the parabolic subgroup of $\GL(Y_1)$ stabilizing $\mathrm{Ker}(A_1)$. 
Then the unipotent group $S_{\sca}$ is in the stabilizer of the pair $(N(\scy),\psi_{\sca})$. (Note that $S_{\sca}$ is not the full stabilizer.) The Fourier coefficient $\phi^{N(\scy),\psi_{\sca}}(g)$ is left-invariant under $[S_{\sca}]$.

\begin{Prop}[\cite{Cai21} Proposition 2.17]\label{prop:invariance II}
For $\phi\in \theta$,
\[
\phi^{N(\scy),\psi_{\sca}}(gh)=\phi^{N(\scy),\psi_{\sca}}(h)
\]
for any $g\in S_{\sca}(F\bs \ba)$ and $h\in \oGL_{W;D}(\ba)$.
\end{Prop}

\begin{Rem}
To further develop the local and global theory of the twisted doubling integrals, we need finer properties of these representations. For instance, multiplicativity of $\gamma$-factors corresponds to the properties of representations of type $(k,m)_D$ with respect to parabolic induction. As we do not require these properties in the present paper, we leave them to a future article.
\end{Rem}

\section{Basic setup in the linear case}\label{sec:Basic setup in the linear case}

We now review the basic setup of the twisted doubling integrals from \cite{Cai21} Section 6. Write $\bP=\bM \cdot \bN$. Then $P(F)\bs G^{\square,k}(F)$ can be identified with the flag variety $\Omega(W^{\square,k})$ of maximal totally isotropic subspaces of $W^{\square,k}$. (In the case of special orthogonal groups, this corresponds to a subset of $\Omega(W^{\square,k})$.) The identification is given by $\gamma\mapsto W^{\Delta,k}\gamma$. We write $L=W^{\Delta,k}\gamma$.

We define the following subset of $G^{\square,k}(F)$:
\[
\Omega_1=\{\gamma \in G^{\square,k}(F): \psi_{\scw,k}^{\bullet}|_{[\gamma^{-1} N\gamma \cap N_{\scw,k}^{\bullet}]}\neq 1\}.
\]
If $\gamma \in G^{\square,k}(F) - \Omega_1$, then $\psi_{\scw,k}^{\bullet}|_{[\gamma^{-1} N\gamma \cap N_{\scw,k}^{\bullet}]}= 1$. The character $\psi_{\scw,k}^{\bullet}$ induces a character on a unipotent subgroup of $[\gamma^{-1}N\gamma\bs \gamma^{-1}P\gamma]$, which is isomorphic to a general linear group over $D$. It is defined by the following pair
\begin{equation}\label{eq:character induced on levi}
([\gamma^{-1} N\gamma \cap N_{\scw,k}^{\bullet}\bs \gamma^{-1} P\gamma \cap N_{\scw,k}^{\bullet}],\psi_{\scw,k}^{\bullet}).
\end{equation}
We define
\[
\Omega_2=\{\gamma \in G^{\square,k}(F)-\Omega_1: \eqref{eq:character induced on levi} \text{ is given by a pair that lies in an orbit higher than }(k^m)_D\}.
\]
Both $\Omega_1$ and $\Omega_2$ are double cosets in $P(F)\bs G^{\square,k}(F)/N_{\scw,k}^{\bullet}(F)$ and we have a nice geometric interpretation of $\Omega_1\cup \Omega_2$.

\begin{Lem}[\cite{Cai21} Section 6.1]
With notations being as above, then
\[
\Omega_1\cup \Omega_2=\{\gamma \in G^{\square,k}(F): L\cap Y_{k-1}=\{0\}\}.
\]
\end{Lem}

\begin{proof}
The only new case is the case of special orthogonal groups, which follows from the case of orthogonal cases.
\end{proof}

Let $\tilde G^{\square,k}(F)=G^{\square,k}(F) - (\Omega_1\cup \Omega_2)$ and consider $P(F)\bs \tilde G^{\square,k}(F) /N_{\scw,k}^{\bullet}(F)$. It is stable under the right action of $\iota(G\times G)(F)$.

The results in \cite{Cai21} Section 6.2 can be summarized as follows:

\begin{Prop}[\cite{Cai21} Section 6.2]\label{prop:results about orbits}
We have the following.
\begin{enumerate}
\item The double coset $P(F)\bs \tilde G^{\square,k}(F) /\iota(G\times G)N_{\scw,k}^{\bullet}(F)$ is finite.
\item For an $\iota(G\times G)(F)$-orbit in $P(F)\bs \tilde G^{\square,k}(F) /\iota(G\times G)N_{\scw,k}^{\bullet}(F)$, one can choose a representative $\gamma$ such that
\begin{itemize}
\item The Fourier coefficient in \eqref{eq:character induced on levi} is of the form studied in Section \ref{sec:invariance under stab}. The value $a$ can be determined explicitly by $\gamma$. If $a<m$, let $S_\lambda$ denote the subgroup $S_\mathcal{A}$ in Section \ref{sec:invariance under stab}.
\end{itemize}
\item The stabilizer $R_-$ in $\{1\}\times G$ of each representative in (2) contains the unipotent radical $N_-$ of a parabolic subgroup of $\{1\}\times G$ as a normal subgroup.
\item If $a<m$, then $N^-$ is nontrivial and the projection of $\{1\} \times N^-$ to $M$ is a subgroup of $S_\lam$. In this case we say that this orbit is \textit{negligible}.
\item If $N^-$ is trivial, we call this orbit the \textit{main orbit}. The representative of the main orbit can be chosen to be the identity element. The stabilizer of the identity element in $\iota(G\times G)(F)$ is $P(F)\cap \iota(G\times G)(F)=\iota(G^{\diamondsuit})(F)$.
\end{enumerate}
\end{Prop}

\begin{proof}
The only case to check is the case of special orthogonal groups or its inner forms. It can be deduced from the orthogonal group case by noting that there is a bijection between \[
P(F)\bs \tilde{\mathrm{O}}(W^{\square,k})(F) /\iota(\mathrm{O}(W)\times \mathrm{O}(W))N_{\scw,k}^{\bullet}(F)
\]
and
\[
P(F)\bs \tilde{\SO}(W^{\square,k})(F) /\iota(\SO(W)\times \SO(W))N_{\scw,k}^{\bullet}(F).
\]
\end{proof}

\section{Assumptions in the covering group case}\label{sec:covering assumptions}

We now discuss the necessary modifications in the case of covering groups. In order to keep the length of this section reasonable, we defer some of the proofs to Section \ref{sec:BD data explicit}.

From now on, we consider $\bG$ to be one of the following groups:
\begin{enumerate}
\item $\bSp(W)$ or its inner forms;
\item $\bSO(W)$ with $\dim W$ even or its inner forms ;
\item $\bSO(W)$ with $\dim W$ odd and $\dim W\geq 3$;
\item $\bU(W)$.
\end{enumerate}
For each group in the list and a fixed integer $k$, we have a list of input $(\bG,\bG^{\square,k}, \iota, \bP,\bN_{\scw,k}^{\bullet},\psi_{\scw,k}^{\bullet})$ from the previous section.

\begin{Rem}
The group $\bSO_1$ is trivial, so all the results in this section are trivial in this case. We will exclude this case. (We still need to consider this case locally in order to discuss multiplicativity, for example.)
\end{Rem}

Let $n$ be a fixed positive integer. Let $\obG\in \CExt(\bG,\bK_2)$ which is classified by the BD data $(Q,\sce,f)$ given a choice of a maximal $F$-torus $\bT$. In the case of unitary groups, $\bU(W)_{F_s}\simeq \bGL_{dm,F_s}$. We also assume that quadratic form $Q$ in the BD data is decomposable. (This means that if we write $Y=Y_1\oplus Y_2$ according to $\bT=\bT_1\times \bT_2$, then $Q$ can be written as the direct sum of $Q|_{Y_1}$ and $Q|_{Y_2}$.) This assumption will greatly simplify the situation. We will discuss more on this in Remark \ref{rem:commutative assumption}.

\begin{Lem}
Assume that $D$ is a field and $\dim_D W> 1$. Then $2 \mid Q(\al^\vee)$ where $\al^\vee$ is a coroot in the Siegel parabolic subgroup of $\bG$.
\end{Lem}

This result will be proved case-by-case in Section \ref{sec:BD data explicit}. In the unitary case, this result is not true without the decomposable assumption.

Let $n_Q=n/\gcd(n,Q(\al^\vee))$ where $\al^\vee$ is a coroot in the Siegel parabolic subgroup of $\bG_{F_s}$. In the case $\bG=\bU_1$, the quadratic form is $Q(e_1^\vee)=a$. We define $n_Q=n/\gcd(n,2a)$. Without loss of generality, we assume that $n_Q$ is either $n$ or $n/2$. 

Let $\overline{\bG^{\square,kn_Q}} \in \CExt(\bG^{\square,kn_Q},\bK_2)$. The pullback of $\overline{\bG^{\square,kn_Q}} $ via the doubling homomorphism
\[
\iota:\bG\times \bG\to \bG^{\square,kn_Q},
\]
gives two multiplicative $\bK_2$-torsors on $\bG$:
\[
1\to \bK_2\to \obG_- \to \iota(1\times \bG) \to 1
\]
and
\[
1\to \bK_2\to \obG_+ \to \iota(\bG\times 1) \to 1.
\]
A priori, these two multiplicative $\bK_2$-torsors may not be isomorphic. In fact, a simple calculation on the quadratic form suggests that they are not isomorphic unless $kn_Q=1$. In any case, what we need is the following result.

\begin{Thm}\label{thm:existence of cover}
For $\obG \in \CExt(\bG,\bK_2)$, there exists $\overline{\bG^{\square,kn_Q}} \in \CExt(\bG^{\square,kn_Q},\bK_2)$ such that $\obG_-\simeq \obG$ and $\obG_+$ is the pushout of $\obG$ by the endomorphism $[2kn_Q-1]:\bK_2\to \bK_2$.
\end{Thm}

This theorem is proved as a special case of Corollary \ref{cor:surj III}.

Let $\overline{\bG^{\square,kn_Q}} $ be a multiplicative $\bK_2$-torsor given by the theorem. We simply write $\obG=\obG_-$ and $\obG^\natural = \obG_+$. Then these two extensions fit into the following commutative diagram:
\[
\begin{tikzpicture}[scale=1.5]
\node (A1) at (0,1) {$1$};
\node (B1) at (1.5,1) {$\bK_2$};
\node (C1) at (3,1) {$\obG$};
\node (D1) at (4.5,1) {$\iota(1\times \bG)$};
\node (E1) at (6,1) {$1$};
\node (A2) at (0,0) {$1$};
\node (B2) at (1.5,0) {$\bK_2$};
\node (C2) at (3,0) {$\obG^\natural$};
\node (D2) at (4.5,0) {$\iota(\bG\times 1)$};
\node (E2) at (6,0) {$1$};
\path[->,font=\scriptsize,>=angle 90]
(A1) edge  (B1)
(B1) edge  (C1)
(C1) edge  (D1)
(D1) edge  (E1)
(A2) edge  (B2)
(B2) edge  (C2)
(C2) edge  (D2)
(D2) edge  (E2)
(C1) edge (C2)
(D1) edge (D2)
(B1) edge (B2);
\end{tikzpicture}
\]
Here, the leftmost vertical map is $[2kn_Q-1]: \bK_2 \to \bK_2$. We fix a morphism $\obG \to \obG^\natural$ once and for all.

The decomposable assumption in the unitary case significantly simplifies our argument since the following result is true. This will be proved in Proposition \ref{prop:commutative}.

\begin{Lem}\label{lem:commute}
The multiplicative $\bK_2$-torsors $\obG^\natural$ and $\obG$ commute in $\overline{\bG^{\square,kn_Q}}$.
\end{Lem}

A consequence of this lemma is that we have a doubling homomorphism of multiplicative $\bK_2$-torsors (instead of a map of sets):
\[
\obG^\natural \times \obG \to \overline{\bG^{\square,kn_Q}}.
\]
By composing it with the fixed morphism $\obG\to \obG^{\natural}$, we obtain a homomorphism
\[
\iota:\obG\times \obG\to \overline{\bG^{\square,kn_Q}},
\]
which fits into the following commutative diagram:
\[
\begin{tikzpicture}[scale=1.5]
\node (A1) at (0,1) {$1$};
\node (B1) at (1.5,1) {$\bK_2\times \bK_2$};
\node (C1) at (3,1) {$\obG \times \obG$};
\node (D1) at (4.5,1) {$\bG\times \bG$};
\node (E1) at (6,1) {$1$};
\node (A2) at (0,0) {$1$};
\node (B2) at (1.5,0) {$\bK_2$};
\node (C2) at (3,0) {$\overline{\bG^{\square,kn_Q}}$};
\node (D2) at (4.5,0) {$\bG^{\square,kn_Q}$};
\node (E2) at (6,0) {$1$};
\path[->,font=\scriptsize,>=angle 90]
(A1) edge  (B1)
(B1) edge  (C1)
(C1) edge  (D1)
(D1) edge  (E1)
(A2) edge  (B2)
(B2) edge  (C2)
(C2) edge  (D2)
(D2) edge  (E2)
(C1) edge (C2)
(D1) edge (D2)
(B1) edge (B2);
\end{tikzpicture}.
\]
Here the first vertical map is given by
\[
\bK_2\times \bK_2 \to \bK_2,\qquad (x,y)\mapsto x^{2kn_Q-1}y.
\]

\subsection{The local doubling homomorphism}

We now discuss the local and global consequences using results from Section \ref{sec:pullback pushout baer sum}. The homomorphism
\[
\iota:\oG_v^\natural\times \oG_v\to \overline{G^{\square,kn_Q}_v}
\]
is a lift of the local doubling homomorphism $\iota:G_v\times G_v\to G^{\square,kn_Q}_v$. Here are the consequences:
\begin{itemize}
\item The images of $\oG_v$ and $\oG_v^\natural$ commutes in $\overline{G^{\square,kn_Q}_v}$.
\item We have a commutative diagram
\[
\begin{tikzpicture}[scale=1.5]
\node (A1) at (0,1) {$1$};
\node (B1) at (1.5,1) {$\mu_n\times \mu_n$};
\node (C1) at (3,1) {$\oG_v \times \oG_v$};
\node (D1) at (4.5,1) {$G_v\times G_v$};
\node (E1) at (6,1) {$1$};
\node (A2) at (0,0) {$1$};
\node (B2) at (1.5,0) {$\mu_n$};
\node (C2) at (3,0) {$\overline{G^{\square,kn_Q}_v}$};
\node (D2) at (4.5,0) {$G^{\square,kn_Q}_v$};
\node (E2) at (6,0) {$1$};
\path[->,font=\scriptsize,>=angle 90]
(A1) edge  (B1)
(B1) edge  (C1)
(C1) edge  (D1)
(D1) edge  (E1)
(A2) edge  (B2)
(B2) edge  (C2)
(C2) edge  (D2)
(D2) edge  (E2)
(C1) edge (C2)
(D1) edge (D2)
(B1) edge (B2);
\end{tikzpicture}
\]
where the first vertical map is given by
\[
\mu_n\times \mu_n \to \mu_n,\qquad (\zeta_1,\zeta_2)\mapsto \zeta_1^{-1}\zeta_2.
\]
\end{itemize}

\begin{Lem}
We have the following.
\begin{enumerate}
\item For a unipotent subgroup $\bU$ of $\bG$,
\[
i_u(\iota(g_1,g_2))=\iota(i_u(g_1),i_u(g_2))
\]
for $g_1,g_2\in \bU(F_v)$.
\item Let $F$ be a non-Archimedean field. Assume that $p\nmid n$. We have
\[
s_v(\iota(g_1,g_2))=\iota(s_v(g_1),s_v(g_2)).
\]
for $g_1,g_2\in \bG(\sco_v)$. 
\end{enumerate}
\end{Lem}

\begin{proof}
These are simply consequences of the results in Section \ref{sec:pullback pushout baer sum}.
\end{proof}

\subsection{Splitting over the diagonal copy}

We now discuss an important consequence of Lemma \ref{lem:commute}. We write $\overline{(G\times G)}_v$ to be the pullback of $\overline{G^{\square,kn_Q}_v}$ via $\iota:G_v\times G_v \to G^{\square,kn_Q}_v$.

\begin{Lem}\label{lem:local diagonal splitting}
With the above assumptions, there is a natural splitting $G^{\diamondsuit}_v \to \overline{(G\times G)}_v$.
\end{Lem}

\begin{proof}
We know that the homomorphism
\[
\oG_v \times \oG_v \to \overline{(G\times G)}_v \to \overline{G^{\square,kn_Q}_v}
\]
restricts to $\mu_n\times \mu_n\to \mu_n,(\zeta_1,\zeta_2)\mapsto \zeta_1^{-1}\zeta_2$. Thus we obtain
\[
\overline{(G\times G)}_v=(\oG_v\times \oG_v)/\mu_n^\diamondsuit.
\]
This implies that the image of $\oG_v^\diamondsuit\subset \oG_v\times \oG_v$ in $\overline{(G\times G)}_v$ is $\oG_v^\diamondsuit/\mu_n^\diamondsuit\simeq G_v^{\diamondsuit}$. Thus we have a natural splitting $G_v^{\diamondsuit}\to \overline{(G\times G)}_v$.
\end{proof}

\subsection{The global doubling homomorphism}
The local doubling homomorphism glues to
\[
\iota_{\ba}:\prod_v \oG_v \times \prod_v \oG_v \to \prod_v \overline{G^{\square,kn_Q}_v}.
\]
As $\iota_v(s_v(\bG(\sco_v))\times s_v(\bG(\sco_v))) \subset s_v(\bG^{\square,kn_Q}(\sco_v))$ for almost all $v$, we obtain a homomorphism
\[
\iota_{\ba}:\prod_v{}' \ \oG_v \times \prod_v {}' \ \oG_v \to \prod_v{}'\ \overline{G^{\square,kn_Q}_v}\to \overline{G^{\square,kn_Q}}(\ba).
\]
This map factors through
\[
\iota_{\ba}:\oG(\ba)\times \oG(\ba) \to \overline{G^{\square,kn_Q}}(\ba).
\]
To summarize, we have obtain a global doubling homomorphism so that
\begin{itemize}
\item the restriction to $\mu_n\times \mu_n$ is given by
\[
\iota_{\ba}:\mu_n\times \mu_n\to \mu_n, \qquad (\zeta_1,\zeta_2)\mapsto \zeta_1^{-1}\zeta_2.
\]
\item this map is a lift of the linear doubling homomorphism $G(\ba)\times G(\ba) \to G^{\square,kn_Q}(\ba)$.
\item this map is also compatible with the local doubling homomorphism.
\end{itemize}

The section over the rational points is also compatible with the doubling homomorphism.

\begin{Prop}\label{prop:diag splitting global}
We have a commutative diagram
\[
\begin{tikzpicture}[scale=1.5]
\node (A1) at (0,1) {$G(F)\times G(F)$};
\node (B1) at (3,1) {$\oG(\ba) \times \oG(\ba)$};
\node (A2) at (0,0) {$G^{\square,kn_Q}(F)$};
\node (B2) at (3,0) {$\overline{G^{\square,kn_Q}}(\ba)$};
\path[->,font=\scriptsize,>=angle 90]
(A1) edge  (B1)
(A2) edge  (B2)
(A1) edge  (A2)
(B1) edge  (B2);
\end{tikzpicture}.
\]
\end{Prop}

\begin{proof}
This is again a consequence of the results in Section \ref{sec:pullback pushout baer sum}.
\end{proof}

We can also check that the cover splits over $\iota_{\ba}(G(\ba)^{\diamondsuit})$.

\begin{Lem}\label{lem:diagonal splitting global}
There is a natural map $G(\ba)\to \overline{(G\times G)}(\ba) \to \overline{G^{\square,kn_Q}}(\ba)$.
\end{Lem}

\begin{proof}
This is trivial from the local result Lemma \ref{lem:diagonal splitting global}.
\end{proof}

\begin{Rem}\label{rem:commutative assumption}
An important consequence of Lemma \ref{lem:commute} is that we have $\oG\times \oG/\mu_n^\diamondsuit \simeq \overline{G\times G}$. Given a $\epsilon$-genuine representation $\pi$ and a $\epsilon^{-1}$-genuine representation $\pi'$, their tensor product descents to a representation of $\overline{G\times G}$.

We briefly explain what would happen when Lemma \ref{lem:commute} does not hold. We no longer have a homomorphism $\obG\times \obG\to \overline{\bG\times \bG}$. Thus, given two automorphic representations of $\oG(\ba)$, one cannot directly construct an automorphic representation of $\overline{G\times G}(\ba)$. Instead, one has to use a version of metaplectic tensor product, as in \cite{Mezo04,Takeda16,Takeda17}. In this paper, we do not plan to treat such cases. Moreover, the diagonal copy might not split. 
\end{Rem}

\section{The twisted doubling integrals}\label{sec:global integrals}

We can now present the global twisted doubling integral.

Notation: we use notation $[\oG]$ for $G(F)\bs \oG(\ba)$ for a multiplicative $\bK_2$-torsor of $\bG$. For a unipotent subgroup $\bN\subset \bG$, let $[N]=\bN(F)\bs \bN(\ba)$. Recall that we always identify $[N]$ with the subgroup $i_u([N])$ of $[\oG]$.

\subsection{Petersson inner product}

Let $\pi$ be an irreducible $\epsilon$-genuine cuspidal automorphic representation of $\oG(\ba)$ realized on a space $V_\pi\subset L_{\epsilon}^2(G(F)\bs \oG(\ba))$, where we fix an embedding $\pi\hookrightarrow V_\pi\subset \sca(\oG(\ba))$. The contragredient representation $\pi^\vee$ is $\epsilon^{-1}$-genuine and is realized on the complex conjugate $\overline{V_\pi}$ of $V_\pi$. The Petersson pairing $\scp=\scp_{\pi}: V_\pi\boxtimes \overline{V_\pi}\to \bc$ is defined by
\[
\scp_{\pi}(\xi_1\boxtimes \xi_2)=\dfrac{1}{n}\int\limits_{G(F)\bs \oG(\ba)}\xi_1(g)\xi_2(g) \ dg.
\]
The integrand is trivial on both $\mu_n$ and $G(F)$ and the pairing is $\oG(\ba)$-invariant.

The Petersson inner product admits a slightly different formula. The function $\xi_1\boxtimes \xi_2\in \pi\boxtimes \pi^\vee$ is a function on $\oG(\ba)\times \oG(\ba)$ which is trivial on $\mu_n^\diamondsuit$. Thus this descents to a function on $\overline{G\times G}(\ba)$, which will be denoted as $\overline{\xi_1\boxtimes \xi_2}$. By Lemma \ref{lem:diagonal splitting global}, the image of $\oG(\ba)$ under $\oG(\ba)\times \oG(\ba)\to \overline{G\times G}(\ba)$ is $G(\ba)$. It is easy to check that
\[
\scp_{\pi}(\xi_1\boxtimes \xi_2)=\int\limits_{G(F)\bs G(\ba)} \overline{\xi_1\boxtimes \xi_2}((g,g)) \ dg.
\]

\subsection{Metaplectic restricted tensor product}

We now recall the notion of meteplectic restricted tensor product. Notation: $\pi=\tilde\otimes_v'\pi_v$. We can view $\pi$ as a representation of $\prod_v' \oG_v$ via the projection map $\prod_v' \oG_v\to \oG(\ba)$. The space of the representation stays the same. Thus we can write $\pi$ as a restricted tensor product $\pi=\otimes'_v \pi_v$, where $\pi_v$ is an admissible representation of $\oG_v$.

We fix isomorphisms $V_\pi\simeq \tilde\otimes'_v \pi_v$ and $\overline{V_\pi}\simeq \tilde\otimes'_v \pi_v^\vee$.

\begin{Lem}
We can choose standard local pairing $\scp_{\pi_v}:\pi_v\boxtimes \pi_v^\vee \to \bc$ at every local place $v$ in order that
\[
\scp_{\pi}(\xi)=\prod_v \scp_{\pi_v}(\xi_v)
\]
for all $\xi=\otimes_v \xi_v\in V_\pi\boxtimes \overline{V_\pi}$, where $\scp_{\pi_v}(\xi_{0,v})=1$ for almost all the $s_v(K_v)\times s_v(K_v)$-invariant vectors $\xi_{0,v}\in \pi_v\boxtimes \pi_v^\vee$ used to defined the restricted tensor products.
\end{Lem}

\begin{proof}
Let $\pi_v$ be an $\epsilon$-genuine representation of $\oG_v$. The action on $\pi^\vee$ is given via
\[
\la \xi_v,\pi_v^\vee(g)\xi_v^\vee \ra=\la \pi_v(g^{-1})\xi_v,\xi_v^\vee \ra.
\]
Thus $\pi_v^\vee$ is an $\epsilon^{-1}$-genuine representation of $\oG_v$. Recall that $\Hom_{\oG_v}(\pi_v,\pi_v)=\bc$. This implies that $\Hom_{\oG_v}(\pi_v\boxtimes \pi_v^\vee,\bc)=\bc$.

The Petersson inner product defines a $\oG(\ba)$-equivariant pairing on global representations
\[
\pi\boxtimes \pi^\vee \to \bc.
\]
This (abstract) pairing is also $\prod_v'\oG_v$-equivariant. One can argue as in the linear case to prove the result (see \cite{Bump97} Section 3.5 for an analogous argument).
\end{proof}

\subsection{Eisenstein series}

We now describe the Eisenstein series that appear in the global construction. Let $\theta$ be an irreducible unitary automorphic representation of $\oGL_{kmn_Q,D}(\ba)$ of type $(kn_Q,m)_D$. We define the normalized global induced representation $I(s,\theta)=\Ind_{\oP(\ba)}^{\overline{G^{\square,kmn_Q}}(\ba)}(\theta\cdot \nu^s)$.
Here $\nu$ is defined as $\oP(\ba)\to P(\ba)\to M(\ba) \to \bc^\times$.

For any holomorphic section $\tilde\phi^{(s)}$ of $I(s,\theta)$, we write $\phi^{(s)}(g)=\tilde\phi^{(s)}(g;1)$ to be the value at the identity. We form the associated Eisenstein series $E(\phi^{(s)})$ on $G^{\square,kmn_Q}(F)\bs \overline{G^{\square,kmn_Q}}(\ba)$ by
\[
E(\phi^{(s)})(g)=\sum_{\gamma\in P(F)\bs G^{\square,kmn_Q}(F)} \phi^{(s)}(\gamma g).
\]
The Eisenstein series converges for $\Re s\gg 0$. By the theory of Eisenstein series, it can be continued to a meromorphic function in $s$ on all of $\bc$ satisfying a functional equation.

We now discuss the extra invariance property in this setup.  Notation: we write $N_{\scw}^{\bullet} = N_{\scw,kn_Q}^{\bullet}\subset G^{\square,kmn_Q}$ and $\psi_{\scw}^{\bullet}=\psi_{\scw,kn_Q}^{\bullet}$ for ease of notations.

Let
\[
f^{(s)}(g)=f_{\phi}^{(s)}(g):=\int\limits_{[N_{\scw}^{\bullet} \cap P]}\phi^{(s)}(ug) \psi_{\scw}^{\bullet}(u) \ du=\int\limits_{[N_{\scw}^{\bullet}\cap P]}\tilde\phi^{(s)}(g;u) \psi_{\scw}^{\bullet}(u) \ du.
\]
The pair $(\bN_{\scw}^{\bullet}\cap \bP,\psi_{\scw}^{\bullet}|_{[N_{\scw}^{\bullet}\cap P]})$ gives a Fourier coefficient in the orbit $((kn_Q)^m)_D$ and $\iota(G^{\diamondsuit})$ lies in the stabilizer. For fixed $g$, the function $u\mapsto \phi^{(s)}(ug)$ on $(N_{\scw}^{\bullet}\cap P)(\ba)$ is an element of $\theta\cdot \nu^s$. Thus one can view $f^{(s)}(g)$ as a $(kn_Q,m)_D$-coefficient of an element in $\theta \cdot \nu^s$.

\begin{Lem}\label{lem:invariance in doubling}
There exists a character $\chi_\theta: [G] \to \bc^\times$ such that for $\iota(g,g)\in \iota(G(\ba)^\diamondsuit)$ and $h\in \overline{G^{\square,kmn_Q}}(\ba)$,
\[
f^{(s)}(\iota(g,g)h)=\chi_\theta(\nu(g)) f^{(s)}(h).
\]
\end{Lem}

\begin{proof}
Recall that the homomorphism
\[
\oG^\diamondsuit(\ba) \to \oG(\ba)\times \oG(\ba) \to \overline{G^{\square,kmn_Q}}(\ba)
\]
is trivial on $\mu_n$ and thus descent to
\[
G^{\diamondsuit}(\ba)\to \overline{G^{\square,kmn_Q}}(\ba).
\]
This indeed gives a homomorphism
\[
G^{\diamondsuit}(\ba) \to \oP(\ba) \to \oM(\ba).
\]
The image of $G^{\diamondsuit}(\ba)$ lies in the stabilizer of the Fourier coefficient. Thus the Fourier coefficient is left-equivariant under a character $\chi_\theta$; in other words,
\[
f^{(s)}(\iota(g,g)h)=\chi_\theta(\nu(g))f^{(s)}(h),
\]
for $\iota(g,g)\in \iota(G(\ba)^\diamondsuit)$ and $h\in \overline{G^{\square,kmn_Q}}(\ba)$.
\end{proof}

\subsection{The global integral}

We view $\chi_\theta$ as a character of $\oG(\ba)$ via $\oG(\ba)\to G(\ba) \xrightarrow{\chi_{\theta}} \bc^\times$, which we still denote as $\chi_\theta$. We define the global integral to be
\[
Z(\xi_1\boxtimes \xi_2,\phi^{(s)})=\dfrac{1}{n^2}\int\limits_{[\oG \times \oG]} \chi_\theta(\nu(g_2))^{-1}\xi_1(g_1) \xi_2(g_2) \int\limits_{[N_{\scw}^{\bullet}]}E(\phi^{(s)})(u\cdot \iota(g_1,g_2)) \psi_{\scw}^{\bullet}(u) \ du \ dg_1 \ dg_2.
\]
Since $\iota_{\ba}(G(F)\times G(F))\subset G^{\square,kmn_Q}(F)$, this integral is well-defined.

Since the two cusp forms are rapidly decreasing on $G(F)\bs \oG(\ba)$ and the Eisenstein series is only of moderate growth, we see that the integral converges absolutely for all $s$ away from the poles of the Eisenstein series and is hence meromorphic in $s$.

\begin{Rem}
One can easily check that the integrand as a function of $\oG(\ba)\times \oG(\ba)$, is trivial on $\mu_n\times \mu_n$ and thus can be viewed as a function on $G(\ba) \times G(\ba)$. The factor $1/n^2$ will be cancelled out if we write $Z(\xi_1\boxtimes \xi_2, \phi^{(s)})$ as an integral over $G(\ba) \times G(\ba)$.
\end{Rem}

\begin{Rem}
Observe that the function $\xi_1\boxtimes \xi_2$ is trivial on $\mu_n$ and thus descents to a function on $\overline{G\times G}$. Such functions generate an irreducible cuspidal representation of $\overline{G\times G}(\ba)$, which we denote by $\overline{\pi\boxtimes \pi^\vee}$. In the proof of the main global identity below, we only use the fact that $\overline{\pi\boxtimes \pi^\vee}$ is cuspidal.
\end{Rem}

\subsection{Main global identity}

We now state the main global identity.

\begin{Thm}
  When $\Re s\gg 0$, $Z(\xi_1\boxtimes \xi_2,\phi^{(s)})$ equals
\[
\dfrac{1}{n^2}\int\limits_{G^\diamondsuit(F)\bs (\oG\times \oG)(\ba)}\chi_\theta(\nu(g_2))^{-1} \xi_1(g_1)\xi_2(g_2) \int\limits_{(N_{\scw}^{\bullet}\cap P)(F)\bs N_{\scw}^{\bullet}(\ba)} \phi^{(s)}(u \cdot \iota(g_1,g_2))\psi_{\scw}^{\bullet}(u) \ du \ dg_1 \ dg_2.
\]
\end{Thm}

\begin{proof}
The proof is similar to that of \cite{Cai21}. We will give a sketch here.

When $\Re s \gg 0$, the global integral becomes
\[
\begin{aligned}
&Z(\xi_1\boxtimes \xi_2,\phi^{(s)})\\
=&\dfrac{1}{n^2}\int\limits_{[\oG\times \oG]}
\chi_\theta(\nu(g_2))^{-1} \xi_1(g_1)\xi_2(g_2) \ \int\limits_{[N_{\scw}^{\bullet}]} \sum_{\gamma\in P(F)\bs G^{\square,kmn_Q}(F)} \phi^{(s)}(\gamma u \cdot \iota(g_1,g_2))\psi_{\scw}^{\bullet}(u) \ du \ dg_1 \ dg_2.\\
\end{aligned}
\]
We can rewrite the integral as a sum over $P(F)\bs G^{\square,kmn_Q}(F)/\iota(G\times G)N_{\scw}^{\bullet}(F)$. Our goal is to show that, only the double coset $P(F)\iota(G\times G)N_{\scw}^{\bullet}(F)$ supports nonzero contribution.

We first calculate the Fourier coefficient and deduce that
\begin{equation}\label{eq:FC in unfolding}
\begin{aligned}
& \int\limits_{[N_{\scw}^{\bullet}]} \ \sum_{\gamma\in P(F)\bs G^{\square,kmn_Q}(F)}  \phi^{(s)}(\gamma u \cdot \iota(g_1,g_2))\psi_{\scw}^{\bullet}(u) \ du\\
=&  \int\limits_{[N_{\scw}^{\bullet}]} \ \sum_{\gamma\in P(F)\bs G^{\square,kmn_Q}(F)/N_{\scw}^{\bullet}(F)} \ \sum_{\gamma'\in \gamma^{-1}P(F)\gamma\cap N(F)\bs N_{\scw}^{\bullet}(F)}  \phi^{(s)}(\gamma \gamma' u \cdot \iota(g_1,g_2))\psi_{\scw}^{\bullet}(u) \ du\\
=& \sum_{\gamma\in P(F)\bs G^{\square,kmn_Q}(F)/N_{\scw}^{\bullet}(F)} \ \int\limits_{(N_{\scw}^{\bullet}\cap \gamma^{-1}P\gamma)(F)\bs N_{\scw}^{\bullet}(\ba)} \phi^{(s)}(\gamma  u \cdot \iota(g_1,g_2))\psi_{\scw}^{\bullet}(u) \ du.
\end{aligned}
\end{equation}
For each $\gamma\in P(F)\bs G^{\square,kmn_Q}(F)/N_{\scw}^{\bullet}(F)$ and $h \in \overline{G^{\square,kmn_Q}}(\mathbb{A})$, we write
\[
\begin{aligned}
I_\gamma(h)=& \int\limits_{(N_{\scw}^{\bullet}\cap \gamma^{-1}P\gamma)(F)\bs N_{\scw}^{\bullet}(\ba)} \phi^{(s)}(\gamma  u  h)\psi_{\scw}^{\bullet}(u) \ du,\\
J_\gamma(h)=& \int\limits_{[N_{\scw}^{\bullet}\cap \gamma^{-1}P\gamma]} \phi^{(s)}(\gamma  u h)\psi_{\scw}^{\bullet}(u) \ du.
\end{aligned}
\]
Thus
\[
I_\gamma(h)= \int\limits_{(N_{\scw}^{\bullet}\cap \gamma^{-1}P\gamma\bs N_{\scw}^{\bullet})(\ba)} J_\gamma(uh)\psi_{\scw}^{\bullet}(u) \ du.
\]

Recall that in Section \ref{sec:Basic setup in the linear case}, we have defined two subsets $\Omega_1$ and $\Omega_2$ of $G^{\square,kmn_Q}$. 
If $\gamma\in \Omega_1$, then $J_{\gamma}(h)$ contains the following inner integral
\begin{equation}\label{eq:vanishing for omega1}
\int\limits_{[N_{\scw}^{\bullet}\cap \gamma^{-1}N\gamma]} \phi^{(s)}(\gamma  u h)\psi_{\scw}^{\bullet}(u) \ du.
\end{equation}
Note that $\phi^{(s)}$ is left invariant under $[N]$. Thus, as a function on $[N_{\scw}^{\bullet}\cap \gamma^{-1}N\gamma]$, $\phi^{(s)}(\gamma  u h)$ is a constant function. According to the definition of $\Omega_1$, the restriction of $\psi_{\scw}^{\bullet}(u)$ to $[N_{\scw}^{\bullet}\cap \gamma^{-1}N\gamma]$ is a nontrivial character. Therefore, \eqref{eq:vanishing for omega1} vanishes and so do $J_{\gamma}(h)$ and $I_{\gamma}(h)$. 

If $\gamma\in \Omega_2$, then \eqref{eq:vanishing for omega1} is constant as a function on $[N_{\scw}^{\bullet}\cap \gamma^{-1}N\gamma]$ and $J_{\gamma}(h)$ becomes a Fourier coefficient of $\theta$ which is given by a pair that lies in an orbit higher than $((kn_Q)^{m})_D$. We now have $J_{\gamma}(h)=0$ since $\theta$ is a representation of type $(kn_Q,m)_D$. Thus, $I_{\gamma}(h)=0$ as well. 

Therefore, we have shown that \eqref{eq:FC in unfolding} equals 
\[
\sum_{\gamma\in P(F)\bs \tilde G^{\square,kmn_Q}(F)/N_{\scw}^{\bullet}(F)} \ I_{\gamma}(\iota(g_1,g_2)),
\]
and therefore,
\[
Z(\xi_1\boxtimes \xi_2,\phi^{(s)})=\dfrac{1}{n^2}\int\limits_{[\oG\times \oG]}
\chi_\theta(\nu(g_2))^{-1} \xi_1(g_1)\xi_2(g_2) \ \sum_{\gamma\in P(F)\bs \tilde G^{\square,kmn_Q}(F)/N_{\scw}^{\bullet}(F)} \ I_{\gamma}(\iota(g_1,g_2)) \ dg_1 \ dg_2.
\]

We now break the sum, and exchange it with integration again. This shows that the above equation equals
\[
\dfrac{1}{n^2} \sum_{\gamma \in P(F)\bs \tilde G^{\square,kmn_Q}(F)/\iota(G\times G)N_{\scw}^{\bullet}(F)} K_{\gamma},
\]
where
\[
K_{\gamma}=\int\limits_{\iota^{-1}(\gamma P\gamma^{-1}\cap \iota(G\times G))(F)\bs (\oG\times \oG)(\ba)} \chi_\theta(\nu(g_1))^{-1}\xi_1(g_1)\xi_2(g_2) I_{\gamma}(\iota(g_1,g_2)) \ dg_1 \ dg_2.
\]
We now apply results in Proposition \ref{prop:results about orbits}.
The double coset $P(F)\bs \tilde G^{\square,kmn_Q}(F)/\iota(G\times G)N_{\scw}^{\bullet}(F)$ is finite, and the only open coset is $P(F)\iota(G\times G)N_{\scw}^{\bullet}(F)$. For negligible double cosets, by Proposition \ref{prop:results about orbits} and \ref{prop:invariance II}, we know
\[
J_{\gamma}(\iota(1,g_2)h)=J_{\gamma}(h)
\]
for $g_2 \in N_-(\ba)$ and $h\in \overline{G^{\square,kmn_Q}}(\ba)$. The proof as in \cite{Cai21} Proposition 6.7 shows that
\[
I_{\gamma}(\iota(1,g_2)h)=I_{\gamma}(h)
\]
for $g_2 \in N_-(\ba)$ and $h\in \overline{G^{\square,kmn_Q}}(\ba)$. Here we need to use Lemma \ref{lem:modular quasicharacter}.

We conclude that for a negligible double coset, the contribution $K_{\gamma}$ contains the inner integral
\[
\int\limits_{[N_-]}\xi_2(ug_2) \ du.
\]
This is zero since $\pi^\vee$ is cuspidal. Thus only the main orbit has nonzero contribution. We have arrived at
\[
Z(\xi_1\boxtimes \xi_2,\phi^{(s)})= \dfrac{1}{n^2}\int\limits_{G^\diamondsuit(F)\bs (\oG\times \oG)(\ba)}\chi_\theta(\nu(g_2))^{-1} \xi_1(g_1)\xi_2(g_2) I_1(\iota(g_1,g_2))\ d g_1 \ d g_2.
\]
This proves the result.
\end{proof}

\subsection{Euler product}

We can indeed rewrite $Z(\xi_1\boxtimes\xi_2,\phi^{(s)})$ as an Euler product using uniqueness of $(kn_Q,m)_D$-models for $\theta$.

Let $\bN_{\scw}^{\circ}=\bN_{\scw}^{\bullet}\cap \bN(W^{\nabla,kn_Q})$. Then for any $h\in \overline{G^{\square,kmn_Q}}(\ba)$,
\[
\begin{aligned}
&\int\limits_{(N_{\scw}^{\bullet}\cap P)(F)\bs N_{\scw}^{\bullet}(\ba)} \phi^{(s)}( u \cdot  \iota(g_1,g_2)h)\psi_{\scw}^{\bullet}(u) \ du\\
 =&\int\limits_{N_{\scw}^\circ(\ba)} \ \int\limits_{[N_{\scw}^{\bullet}\cap P]} \phi^{(s)}(  u u' \cdot  \iota(g_1,g_2)h)\psi_{\scw}^{\bullet}( uu') \ du \ du'\\
 =&\int\limits_{N_{\scw}^\circ(\ba)}  f^{(s)}( u' \cdot  \iota(g_1,g_2)h)\psi_{\scw}^{\bullet}( u') \ du'\\
 =&\int\limits_{N_{\scw}^\circ(\ba)} f^{(s)}(\iota(g_2,g_2)\cdot u' \cdot  \iota(g_2^{-1}g_1,1)h)\psi_{\scw}^{\bullet}( u')  \ du'\\
  =&\chi_{\theta}(\nu(g_2))\int\limits_{N_{\scw}^\circ(\ba)} f^{(s)}(u \cdot  \iota(g_2^{-1}g_1,1)h)\psi_{\scw}^{\bullet}(u)  \ du.\\
 \end{aligned}
\]
Observe that we use change of variable and Lemma \ref{lem:modular quasicharacter} in the third equality, and Lemma \ref{lem:invariance in doubling} in the last equality. We also use the fact that the canonical lift of $N_{\scw}^{\bullet}(\ba)$ is $G^{\square,kmn_Q}(\ba)$-equivariant. As a consequence, we can write $Z(\xi_1\boxtimes \xi_2,\phi^{(s)})$ as
\[
\begin{aligned}
&\dfrac{1}{n^2}\int\limits_{G^{\diamondsuit}(F)\bs (\oG\times \oG)(\ba)}\xi_1(g_1)\xi_2(g_2)\int\limits_{N_{\scw}^\circ(\ba)} f^{(s)}(u \cdot  \iota(g_2^{-1}g_1,1))\psi_{\scw}^{\bullet}( u)  \ du \ dg_2 \ dg_1\\
=&\dfrac{1}{n^2}\int\limits_{G^{\diamondsuit}(F)\bs (\oG\times \oG)(\ba)}\xi_1(g_2g_1)\xi_2(g_2)\int\limits_{N_{\scw}^\circ(\ba)} f^{(s)}(u \cdot  \iota(g_1,1))\psi_{\scw}^{\bullet}( u)  \ du\ dg_2 \ dg_1\\
=&\dfrac{1}{n^2}\int\limits_{\oG(\ba)} \ \int\limits_{[\oG]}\xi_1(g_2g_1)\xi_2(g_2)\int\limits_{N_{\scw}^\circ(\ba)} f^{(s)}(u \cdot  \iota(g_1,1))\psi_{\scw}^{\bullet}( u)  \ du\ dg_2 \ dg_1\\
=&\dfrac{1}{n}\int\limits_{\oG(\ba)}\scp(\pi(g)\xi_1\boxtimes \xi_2) \int\limits_{N_{\scw}^\circ(\ba)} f^{(s)}(u  \cdot  \iota(g,1))\psi_{\scw}^{\bullet}( u)  \ du  \ d g.\\
\end{aligned}
\]
For decomposable data, it follows from uniqueness of $(kn_Q,m)_D$-models for $\theta_v$ that
\[
f^{(s)}(g)=\prod_v f_v^{(s)}(g_v).
\]
If furthermore $\xi_i=\otimes_v \xi_{i,v}$, then
\[
Z(\xi_1\boxtimes\xi_2,\phi^{(s)})=\prod_v Z_v(\xi_{1,v}\boxtimes \xi_{2,v},f_v^{(s)}),
\]
where
\[
Z_v(\xi_{1,v}\boxtimes \xi_{2,v},f_v^{(s)}) = \dfrac{1}{n} \int\limits_{\oG_v}\scp_v(\pi(g)\xi_{1,v}\boxtimes \xi_{2,v})\int\limits_{N_{\scw}^{\circ}}f_v^{(s)}(u\cdot \iota(g,1))\psi_{\scw}^{\bullet}(u) \ du \ dg.
\]

\section{BD data of pullback}\label{sec:BD data explicit}

The goal in this section is to prove the unproven results in Section \ref{sec:covering assumptions}.

\subsection{Results}

Let $\bG$ be one of the following groups:
\begin{enumerate}
\item $\bSp(W)$ or its inner forms;
\item $\bSO(W)$ with $\dim W$ even or its inner forms;
\item $\bSO(W)$ with $\dim W$ odd and $\dim W\geq 3$;
\item $\bU(W)$.
\end{enumerate}
Observe that over $F_s$, we have to consider the following groups: $\bSp_{2m}, \bSO_{2m}, \bSO_{2m+1}$ and $\bGL_m$.

Let $\obG\in \CExt(\bG,\bK_2)$ with BD data $(Q,\sce,f)$. For simplicity, we write $\bG^{\square}:=\bG^{\square,1}$. Recall that in the unitary group case, we assume that the quadratic form $Q$ in the BD data is decomposable.

\begin{Def}
We define the subcategory
\[
\CExt(\bG, \bK_2)^{\Delta}\subset \CExt(\bG,\bK_2)\times \CExt(\bG,\bK_2)
\]
as follows: a pair $(\obG, \obG')$ is in the subcategory if and only if $\obG'$ is isomorphic to $\obG$.
\end{Def}

\begin{Def}
We define the subcategory
\[
\CExt(\bG, \bK_2)^{\natural}\subset \CExt(\bG,\bK_2)\times \CExt(\bG,\bK_2)
\]
as follows: a pair $(\obG',\obG)$ is in the subcategory if and only if $\obG'$ is isomorphic to the pushout of $\obG$ by the map $[2k-1]$.
\end{Def}

\begin{Def}
We define the subcategory
\[
\CExt(\bG^{\square},\bK_2)^\Delta
\]
of $\CExt(\bG^{\square},\bK_2) \times \cdots \times \CExt(\bG^{\square},\bK_2)$ as follows: an object $(\overline{\bG^{\square}_1},\cdots, \overline{\bG^{\square}_k})$ is in the subcategory if and only if the multiplicative $\bK_2$-torsors $\overline{\bG^{\square}_1},\cdots, \overline{\bG^{\square}_k}$ are isomorphic.
\end{Def}

\begin{Prop}\label{prop:surj I}
Let $\overline{\bG^{\square}} \in \CExt(\bG^{\square},\bK_2)$. The pullback of $\overline{\bG^{\square}} $ via the doubling homomorphism $\iota:\bG\times \bG\to \bG^{\square}$ gives two multiplicative $\bK_2$-torsors on $\bG$:
\[
1\to \bK_2\to \obG_1 \to \iota(1\times \bG) \to 1
\]
and
\[
1\to \bK_2\to \obG_2 \to \iota(\bG\times 1) \to 1.
\]
Then the resulting functor
\[
\CExt(\bG^\square,\bK_2) \to \CExt(\bG,\bK_2)\times \CExt(\bG,\bK_2)
\]
is essentially surjective on $\CExt(\bG,\bK_2)^\Delta$.
\end{Prop}

\begin{Prop}\label{prop:surj II}
The pullback via the homomorphism $\bG^{\square} \times \cdots \times \bG^{\square} \to \bG^{\square,k}$ gives a functor
\[
\CExt(\bG^{\square,k},\bK_2)\to \CExt(\bG^{\square},\bK_2) \times \cdots \times \CExt(\bG^{\square},\bK_2).
\]
Here, both $\bG^{\square}$ and $\CExt(\bG^{\square},\bK_2)$ appear $k$ times. This functor is essentially surjective on $\CExt(\bG^{\square},\bK_2)^\Delta$.
\end{Prop}

With the above two propositions, we deduce the following fact.
\begin{Cor}\label{cor:surj III}
The functor induced by the doubling isomorphism
\[
\CExt(\bG^{\square,k},\bK_2)\to \CExt(\bG,\bK_2)\times \CExt(\bG,\bK_2)
\]
is essentially surjective on $\CExt(\bG\times \bG,\bK_2)^{\natural}$.
\end{Cor}

\begin{proof}
We first consider
\[
\bG \times \bG \to \bG \times \cdots \times \bG,\qquad (g_1,g_2)\mapsto (g_1,g_2,g_1, g_1, \cdots, g_1, g_1).
\]
Here the target has $2k$ copies of $\bG$.
Then we can write $\iota:\bG\times \bG \to \bG^{\square,k}$ as the composition of
\begin{equation}\label{eq:decomposition of doubling}
\bG\times \bG \to \bG \times \cdots \times \bG \to \bG^{\square} \times \cdots \times \bG^{\square} \to \bG^{\square,k}
\end{equation}
By Propositions \ref{prop:surj I} and \ref{prop:surj II}, we can find $\obGsk$ so that its pullback to each $\bG$ under the map
\[
\bG\times \cdots \times \bG \to \bG^{\square,k}
\]
are all isomorphic to $\obG$. Thus, $\obGsk$ pulls back to $\obG$ for the second copy of $\bG$ in $\iota:\bG\times \bG \to \bG^{\square,k}$.

We now consider the pullback to the first copy. Then \eqref{eq:decomposition of doubling} restricts to
\[
\bG\times \{1\} \to \bG \times \{1\} \times \bG \times \bG \times \cdots \times \bG \to \bG^{\square} \times \bG^{\square} \times \cdots \times \bG^{\square} \to \bG^{\square,k}.
\]
We temporarily focus on the second map. By Lemma \ref{lem:commute for G}, the copies of $\obG$ commute in $\overline{\bG^{\square,k}}$, which gives a homomorphism
\[
\obG \times \{1\} \times \obG \times \obG \cdots \times \obG \to \obGsk.
\]
The multiplicative $\bK_2$-torsor $\overline{\bG \times \{1\} \times \bG \times \bG \times \cdots \times \bG}$ obtained by pulling pack along
\[
\bG \times \{1\} \times \bG \times \bG \times \cdots \times \bG \to \bG^{\square,k}
\]
is isomorphic to the pushout of $\obG \times \{1\} \times \obG \times \obG \times \cdots \times \obG $ via the product map $\bK_2 \times \cdots \times \bK_2 \to \bK_2$.

Finally, to obtain the pullback to the first copy of $\bG$ from $\overline{\bG^{\square,k}}$, we need to pullback $\overline{\bG \times \{1\} \times \bG \times \bG \times \cdots \times \bG}$ via the diagonal map
\[
\bG \to \bG \times \{1\} \times \bG \times \bG\times \cdots \times \bG,\qquad g_1\mapsto (g_1,1,g_1, g_1,\cdots, g_1).
\]
It follows from the definition that $\obG^{\natural}$ is the Baer sum of $2k-1$ copies of $\obG$.
\end{proof}

\begin{Prop}\label{prop:commutative}
In the case of unitary groups, we assume that the quadratic form $Q$ is decomposable. For the multiplicative $\bK_2$-torsor $\overline{\bG^{\square,k}}$ on $\bG^{\square,k}$ given by Corollary \ref{cor:surj III}, $\obG$ and $\obG^{\natural}$ commute in $\overline{\bG^{\square,k}}$.
\end{Prop}

\begin{proof}
This follows from Lemma \ref{lem:commute for G} and a simply calculation on the quadratic form. See also the details in all the cases.
\end{proof}

The proof of Proposition \ref{prop:surj II} is similar (and easier) than Proposition \ref{prop:surj I}. So we will only give the details in the latter case. The proof of Proposition \ref{prop:surj I} will be given case-by-case in the rest of this section.

\subsection{Strategy of the proofs}

In this section, we explain the strategy of the proof of Proposition \ref{prop:surj I} and set up some notations that are commonly used.

Let $\bT$ be a maximal $F$-torus of $\bG$. Then $\iota(\bT\times \bT)\subset \iota(\bG\times \bG)$ is an $F$-torus of $\bG^{\square}$. Let $\bT^{\square}\supset \iota(\bT\times \bT)$ be a maximal $F$-torus of $\bG^{\square}$. Note that $\bT^{\square}=\iota(\bT\times \bT)$ except in the case of odd $\bSO(W)$. Observe that this is not a maximal $F$-torus in $\bP(W^\Delta)$ but only up to a conjugation over $F_s$. We still use this non-split torus even if $\bG^{\square}$ might be a split group over $F$.

The torus $\bT$ splits over $F_s$. Let $Y$ be the cocharacter lattice of $\bT$ over $F_s$. Let $Y^{\square}$ be the cocharacter lattice of $\bT^{\square}$ over $F_s$. Then $Y^{\square}\supset Y\oplus Y$. We first fix a Chevalley system of pinning for $(\bG_{F_s}, \bT_{F_s})$, then choose one for $(\bG^{\square}_{F_s}, \bT^{\square}_{F_s})$ which is compatible with $\iota:\bG \times \bG \to \bG^{\square}$.

Notation: the BD data for $\obG$ is denoted $(Q,\sce,f)$; to distinguish the BD data for the two copies of $\obG$, we use $(Q_+,\sce_+,f_+)$ and $(Q_-, \sce_-, f_-)$ when needed. The BD data for $\obG^{\square}$ is denoted $(Q^{\square},\sce^{\square},f^{\square})$.

\begin{Def}
We define the subcategory
\[
\BD(\bG, \bT)^{\Delta}\subset \BD(\bG,\bT)\times \BD(\bG,\bT)
\]
as follows: the pair of triples $(Q,\sce,f)$ and $(Q',\sce',f')$ is in the subcategory if and only if $(Q,\sce,f)$ is isomorphic to $(Q',\sce',f')$.
\end{Def}

It is easy to see that Proposition \ref{prop:surj I} is equivalent to the following result.

\begin{Prop}\label{prop:surj III}
The functor induced by pulling back via $\bG\times \bG \to \bG^{\square}$:
\[
\BD(\bG^{\square},\bT^{\square}) \to \BD(\bG,\bT)  \times \BD(\bG,\bT)
\]
is essentially surjective on $\BD(\bG,\bT)^{\Delta}$.
\end{Prop}

We will prove this result case-by-case.

\subsection{A useful lemma}

Given an exact sequence
\[
1\to F_s^\times \to \sce \to Y \to 1,
\]
we can push out the direct sum of two copies via the product map $\pr:F_s^\times \times F_s^\times \to F_s^\times$ to obtain
\[
1\to F_s^\times \to \pr_\ast(\sce \oplus \sce) \to Y\oplus Y \to 1.
\]

We now give a useful criterion to compare elements in $\pr_{\ast}(\sce\oplus \sce)$. Recall that
\[
\pr_\ast(\sce \oplus \sce) = F_s^\times \times (\sce \oplus \sce)  /\la (x_1x_2, x_1^{-1},x_2^{-1}): (x_1, x_2) \in F_s^\times \times F_s^\times \ra.
\]
Define
\[
\mathrm{mul}:F_s^\times \times (\sce\oplus \sce) \to \sce, \qquad (x,e_1,e_2)\mapsto xe_1e_2.
\]
This gives a well-defined map $\mathrm{mul}:\pr_{\ast}(\sce\oplus \sce) \to \sce$.

\begin{Lem}\label{lem:very useful lemma}
Let $(x,e_1,e_2)$ and $(x',e_1',e_2')$ be two elements in $F_s^\times \times (\sce\oplus \sce)$. They have the same image in $\pr_{\ast}(\sce\oplus \sce)$ if and only if $(e_1,e_2)$ and $(e_1',e_2')$ have the same image in $Y\oplus Y$ and $\mathrm{mul}(x, e_1,e_2)=\mathrm{mul}(x', e_1', e_2')$.
\end{Lem}

\begin{proof}
The `only if' part is trivial. We now prove the `if' part. If $(e_1,e_2)$ and $(e_1',e_2')$ have the same image in $Y\oplus Y$, then $(e_1',e_2')=(e_1y_1,e_2 y_2)$ for some $y_1,y_2\in F_s^\times$. The condition $\mathrm{mul}(x, e_1, e_2)=\mathrm{mul}(x', e_1',e_2')$ implies that $x=x'y_1y_2$. This shows that $(x',e_1', e_2')=(xy_1^{-1}y_2^{-1}, e_1y_1,e_2y_2)$. The proof is complete.
\end{proof}

\subsection{Symplectic groups}

We now discuss the case of symplectic groups or their inner forms. This is probably the easiest case since $\bG$ is simply connected. A multiplicative $\bK_2$-torsor $\obG$ on $\bG$ is determined by a Galois invariant Weyl group invariant quadratic form $Q$ on $Y$. In \cite{Weissman11} Proposition 3.15, it is shown that for every integer $a$, there is a unique such quadratic form on $Y$ such that its value on a short coroot is $a$. In other words, we have an equivalence of categories
\[
\CExt(\bG,\bK_2)\to \BD(\bG,\bT)\to \bZ,
\]
where the last functor sends a quadratic form to its value on a short coroot.

The torus $\bT\times \bT$ is a maximal $F$-torus of $\bG^{\square}$, so $\bT^{\square}=\bT\times \bT$. The cocharacter lattice of $\bT^{\square}$ over $F_s$ is $Y^{\square}=Y\oplus Y$. The functor
\[
\CExt(\bG^{\square},\bK_2)\to \CExt(\bG,\bK_2)\times \CExt(\bG,\bK_2)
\]
can be described in terms of
\[
\BD(\bG^{\square},\bT^{\square}) \to \BD(\bG,\bT) \times \BD(\bG,\bT).
\]
In terms of $\bZ\to \bZ\times \bZ$, it is simply $a\mapsto (a,a)$. Proposition \ref{prop:surj III} follows trivially.

\subsection{Special even orthogonal groups and inner forms}

We now consider the case of special even orthogonal groups. We start with same basic results. 

\subsubsection{Preparation}

Let $n=(d\cdot \dim_D W)/2$. (The results in this section only involve multiplicative $\bK_2$-torsors. This $n$ has no relation with the degree of the cover. It shall not cause any confusion). We choose a standard basis of $Y=\bZ^n=\mathrm{Span}\{e_1^{\vee},\cdots, e_n^{\vee}\}$ so that the root lattice is given by
\[
Y^{sc}=\mathrm{Span}\{\al_1^\vee,\cdots,\al_n^\vee\} =\mathrm{Span}\{e_1^{\vee}-e_2^{\vee},\cdots, e_{n-1}^{\vee}-e_n^{\vee},e_{n-1}^{\vee}+e_n^{\vee}\}.
\]
Note that $Y^{sc}$ is a sublattice of $Y$ with index $2$. A $W$-invariant quadratic form $Q$ is determined by its value on a coroot. Let $Q(\al_1^\vee)=a$.

\begin{Lem}\label{lem:divisible by 2 SO even}
We have $2\mid a$.
\end{Lem}

\begin{proof}
We know that $Q(\al_{n-1}^\vee)=Q(\al_n^\vee)=a$. Let $Q(e_n^{\vee})=b\in \bZ$. As $Q(2e_{n}^{\vee})=Q(\al_{n-1}^\vee)+Q(\al_n^\vee)$, we have $4b=2a$. This implies $2\mid a$.
\end{proof}

The torus $\bT^{\square}=\bT \times \bT$ is a maximal $F$-torus of $\bG^{\square}$. Its cocharacter lattice over $F_s$ is $Y^{\square}=Y_+\oplus Y_-$. We write
\[
\begin{aligned}
Y_+^{sc}=&\mathrm{Span}\{e_1^{\vee}-e_2^{\vee}, \cdots, e_{n-1}^{\vee}-e_n^{\vee}, e_{n-1}^{\vee}+e_n^{\vee}\}\\ Y_-^{sc}=&\mathrm{Span}\{e_{n+1}^{\vee}-e_{n+2}^{\vee}, \cdots, e_{2n-1}^{\vee}-e_{2n}^{\vee}, e_{2n-1}^{\vee}+e_{2n}^{\vee}\}\\
Y^{\square,sc}=&\mathrm{Span}\{e_1^{\vee}-e_2^{\vee},\cdots, e_{2n-1}^{\vee}-e_{2n}^{\vee},e_{2n-1}^{\vee}+e_{2n}^{\vee}\}.
\end{aligned}
\]
Thus $Y_+^{sc}\oplus Y_-^{sc}$ is a subgroup of $Y^{\square,sc}$ of index $2$ and $e_1^{\vee}-e_{n+1}^{\vee}\notin Y_+^{sc}\oplus Y_-^{sc}$.

Let $Q^{\square}$ be a $W$-invariant quadratic form on $Y^{\square}$.

\begin{Lem}
The restriction of $Q^{\square}$ to $Y_+\oplus Y_-$ is a direct sum of two quadratic forms $Q_+\oplus Q_-$ with $Q_+=Q_-$. In particular, we have
\[
B_{Q^{\square}}((y_+,0),(0,y_-))=0
\]
for $y_+\in Y_+$ and $y_-\in Y_-$.

Conversely, for a $W$-invariant quadratic form $Q$ on $Y$, there is a unique $W$-invariant quadratic form $Q^{\square}$ which restricts to $Q_+\oplus Q_-$ on $Y_+\oplus Y_-$.
\end{Lem}

\begin{proof}
The proof is straightforward.
\end{proof}

\subsubsection{Construction of BD data}

Let $\obG$ be a multiplicative $\bK_2$-torsor on $\bG$ with BD data $(Q,\sce,f)$. We use it to construct a Galois equivariant triple $(Q^\square,\sce^\square,f^\square)$. 

Let $Q$ be a Galois invariant  $W$-invariant quadratic form on $Y$. Define $Q^{\square}=Q\oplus Q$ to be a quadratic form on $Y^{\square}=Y\oplus Y$. This is a $W$-invariant quadratic form and thus Galois invariant from the proof of \cite{Weissman11} Proposition 3.15.

Second, from the exact sequence for $\sce$, we form the exact sequence
\[
1\to F_s^\times \oplus F_s^\times \to \sce \oplus \sce \to Y\oplus Y=Y^{\square} \to 1.
\]
We push it out via the product map $\pr: F_s^\times \oplus F_s^\times \to F_s^\times$ to obtain
\[
\begin{tikzpicture}[scale=1.5]
\node (A1) at (0,1) {$1$};
\node (B1) at (1.5,1) {$F_s^\times \oplus  F_s^\times$};
\node (C1) at (3,1) {$\sce\oplus \sce$};
\node (D1) at (4.5,1) {$Y\oplus Y$};
\node (E1) at (6,1) {$1$};
\node (A2) at (0,0) {$1$};
\node (B2) at (1.5,0) {$F_s^\times$};
\node (C2) at (3,0) {$\pr_{\ast}(\sce \oplus \sce)$};
\node (D2) at (4.5,0) {$Y^{\square}$};
\node (E2) at (6,0) {$1$};
\path[->,font=\scriptsize,>=angle 90]
(A1) edge  (B1)
(B1) edge  (C1)
(C1) edge  (D1)
(D1) edge  (E1)
(A2) edge  (B2)
(B2) edge  (C2)
(C2) edge  (D2)
(D2) edge  (E2)
(C1) edge (C2)
(D1) edge (D2)
(B1) edge (B2);
\end{tikzpicture}
\]
The commutator map of the bottom exact sequence is given by $[y_1,y_2]=(-1)^{B_{Q^{\square}}(y_1,y_2)}$. Set $\sce^{\square}:=\pr_{\ast}(\sce \oplus \sce)$. The bottom exact sequence will be the second BD invariant for $\bG^{\square}$.

We also have a natural map $\pr_{\ast}(f\oplus f): \pr_\ast(\sce_{Q^{sc}} \oplus \sce_{Q^{sc}}) \to \pr_\ast(\sce \oplus \sce) $ which fits into
\[
\begin{tikzpicture}[scale=1.5]
\node (A1) at (0,1) {$1$};
\node (B1) at (1.5,1) {$F_s^\times$};
\node (C1) at (3.5,1) {$\pr_{\ast}(\sce_{Q^{sc}}\oplus \sce_{Q^{sc}})$};
\node (D1) at (5.5,1) {$Y^{sc}\oplus Y^{sc}$};
\node (E1) at (7,1) {$1$};
\node (A2) at (0,0) {$1$};
\node (B2) at (1.5,0) {$F_s^\times$};
\node (C2) at (3.5,0) {$\pr_{\ast}(\sce \oplus \sce)$};
\node (D2) at (5.5,0) {$Y^{\square}$};
\node (E2) at (7,0) {$1$};
\path[->,font=\scriptsize,>=angle 90]
(A1) edge  (B1)
(B1) edge  (C1)
(C1) edge  (D1)
(D1) edge  (E1)
(A2) edge  (B2)
(B2) edge  (C2)
(C2) edge  (D2)
(D2) edge  (E2)
(C1) edge (C2)
(D1) edge (D2)
(B1) edge (B2);
\end{tikzpicture}
\]

Let $\bG^{sc}=\bSpin_{2n}\to \bG=\bSO_{2n}$. The multiplicative $\bK_2$-torsor of $\bSpin_{2n}$ is determined by a quadratic form $Y=Y^{sc}$, which is determined by its values on any simple coroots. From the map
\[
\bSpin_{2n} \times \bSpin_{2n} \to \bSpin_{4n},
\]
we obtain the following commutative diagram,
\[
\begin{tikzpicture}[scale=1.5]
\node (A1) at (0,1) {$1$};
\node (B1) at (1.5,1) {$F_s^\times \oplus  F_s^\times$};
\node (C1) at (3.5,1) {$\sce_{Q^{sc}} \oplus \sce_{Q^{sc}}$};
\node (D1) at (5.5,1) {$Y_+^{sc}\oplus Y_-^{sc}$};
\node (E1) at (7,1) {$1$};
\node (A2) at (0,0) {$1$};
\node (B2) at (1.5,0) {$F_s^\times$};
\node (C2) at (3.5,0) {$\sce_{Q^{\square,sc}}$};
\node (D2) at (5.5,0) {$Y^{\square,sc}$};
\node (E2) at (7,0) {$1$};
\path[->,font=\scriptsize,>=angle 90]
(A1) edge  (B1)
(B1) edge  (C1)
(C1) edge  (D1)
(D1) edge  (E1)
(A2) edge  (B2)
(B2) edge  (C2)
(C2) edge  (D2)
(D2) edge  (E2)
(C1) edge (C2)
(D1) edge (D2)
(B1) edge (B2);
\end{tikzpicture},
\]
where the first vertical map is given by multiplication. By our choice of Chevelley system of \'epinglage, under the second vertical map, the image of $s_{Q_{\pm}^{sc}}(\al^\vee)$ is $s_{Q^{\square,sc}}(\al^\vee)$ for a root $\al$ in $Y_{\pm}^{sc}$. This commutative diagram factors through pushing out by $\pr:F_s^\times \oplus F_s^\times \to F_s^\times$ and we have the following commutative diagram:
\[
\begin{tikzpicture}[scale=1.5]
\node (A1) at (0,1) {$1$};
\node (B1) at (1.5,1) {$F_s^\times$};
\node (C1) at (3.5,1) {$\pr_{\ast}(\sce_{Q^{sc}} \oplus \sce_{Q^{sc}})$};
\node (D1) at (5.5,1) {$Y_+^{sc}\oplus Y_-^{sc}$};
\node (E1) at (7,1) {$1$};
\node (A2) at (0,0) {$1$};
\node (B2) at (1.5,0) {$F_s^\times$};
\node (C2) at (3.5,0) {$\sce_{Q^{\square,sc}}$};
\node (D2) at (5.5,0) {$Y^{\square,sc}$};
\node (E2) at (7,0) {$1$};
\path[->,font=\scriptsize,>=angle 90]
(A1) edge  (B1)
(B1) edge  (C1)
(C1) edge  (D1)
(D1) edge  (E1)
(A2) edge  (B2)
(B2) edge  (C2)
(C2) edge  (D2)
(D2) edge  (E2)
(C1) edge (C2)
(D1) edge (D2)
(B1) edge (B2);
\end{tikzpicture}.
\]

We now want to construct the third BD invariant. That is, we need to construct a map
\[
\begin{tikzpicture}[scale=1.5]
\node (A1) at (0,1) {$1$};
\node (B1) at (1.5,1) {$F_s^\times$};
\node (C1) at (3,1) {$\sce_{Q^{\square,sc}}$};
\node (D1) at (4.5,1) {$Y^{\square,sc}$};
\node (E1) at (6,1) {$1$};
\node (A2) at (0,0) {$1$};
\node (B2) at (1.5,0) {$F_s^\times$};
\node (C2) at (3,0) {$\pr_{\ast}(\sce \oplus \sce)$};
\node (D2) at (4.5,0) {$Y^{\square}$};
\node (E2) at (6,0) {$1$};
\path[->,font=\scriptsize,>=angle 90]
(A1) edge  (B1)
(B1) edge  (C1)
(C1) edge  (D1)
(D1) edge  (E1)
(A2) edge  (B2)
(B2) edge  (C2)
(C2) edge  (D2)
(D2) edge  (E2)
(C1) edge (C2)
(D1) edge (D2)
(B1) edge (B2);
\end{tikzpicture}
\]
which the middle map $f^{\square}$ extends $\pr_{\ast}(\sce_{Q^{sc}}\oplus \sce_{Q^{sc}}) \to \pr_{\ast}(\sce \oplus \sce)$.
We also use the notation $f_+$ to denote the map
\[
\sce_{Q^{sc}} \xrightarrow{e\mapsto (e,1)} \sce_{Q^{sc}} \oplus \sce_{Q^{sc}} \to \pr_{\ast}(\sce_{Q^{sc}}\oplus \sce_{Q^{sc}}) \to \pr_{\ast}(\sce \oplus \sce).
\]
Similarly we define $f_-$.

The map $f^{\square}$ is determined by its images on $\{s_{Q^{\square,sc}}(\al^\vee)\mid \al^\vee \in \Delta_{\bG^{\square}}^{\vee}\}$. Since we require that
\[
f^{\square}(s_{Q^{\square,sc}}(\al^\vee)) =f_{\pm}(s_{Q_{\pm}^{sc}}(\al^\vee))
\]
for a root $\al$ in $Y_{\pm}^{sc}$, this trivially determines $f^{\square}$ by $f_+\oplus f_-$ except $f^{\square}(s_{Q^{\square,sc}}(e_n^{\vee}-e_{n+1}^{\vee}))$. We have to choose this value so that
\[
f^{\square}(s_{Q^{\square,sc}}(e_{n-1}^\vee + e_n^\vee))=f_+(s_{Q^{sc}}(e_{n-1}^\vee + e_n^\vee)).
\]
We find that it would be slightly more convenient to work with the following setup: we choose $f^{\square}(s_{Q^{\square,sc}}(e_1^{\vee}-e_{n+1}^{\vee}))$ so that
\[
f^{\square}(s_{Q^{\square,sc}}(e_{1}^\vee + e_2^\vee))=f_+(s_{Q^{sc}}(e_{1}^\vee + e_2^\vee)).
\]

\begin{Lem}
If we choose $f^{\square}(s_{Q^{\square,sc}}(e_1^{\vee}-e_{n+1}^{\vee}))$ is the unique element such that its projection to $Y\oplus Y$ is $e_1^{\vee}-e_{n+1}^{\vee}$ and $\mathrm{mul}(f^{\square}(s_{Q^{\square,sc}}(e_1^{\vee}-e_{n+1}^{\vee})))=1$, then $f^{\square}(s_{Q^{\square,sc}}(e_{1}^\vee + e_2^\vee))=f_+(s_{Q^{sc}}(e_{1}^\vee + e_2^\vee))$. In other words, $f^{\square}$ is an extension of $f\oplus f$.
\end{Lem}

\begin{proof}
Recall that
\[
e_{1}^\vee + e_2^\vee =  (e_{2}^\vee - e_1^\vee) + (e_1^\vee - e_{n+1}^\vee) + (e_{n+1}^\vee - e_{n+2}^\vee)+(e_{n+1}^\vee + e_{n+2}^\vee)+(e_1^\vee - e_{n+1}^\vee).
\]
This decomposition satisfies the condition $(\ast)$ in \eqref{eq:condition star}. By Lemma \ref{lem:additivity of sQ} (note that $2\mid Q(\al^\vee)$), $s_{Q^{\square,sc}}(e_{1}^\vee + e_2^\vee) =$
\[
s_{Q^{\square,sc}}(e_1^\vee - e_{n+1}^\vee)
s_{Q^{\square,sc}}(e_{n+1}^\vee + e_{n+2}^\vee)
s_{Q^{\square,sc}}(e_{n+1}^\vee - e_{n+2}^\vee)
s_{Q^{\square,sc}}(e_1^\vee - e_{n+1}^\vee)
s_{Q^{\square,sc}}(e_{2}^\vee - e_1^\vee).
\]
This implies that $f^{\square}(s_{Q^{\square,sc}}(e_{1}^\vee + e_2^\vee))=$
\[
f^{\square}(s_{Q^{\square,sc}}(e_1^\vee - e_{n+1}^\vee))
f^{\square}(s_{Q^{\square,sc}}(e_{n+1}^\vee + e_{n+2}^\vee) )
f^{\square}(s_{Q^{\square,sc}}(e_{n+1}^\vee - e_{n+2}^\vee) )
f^{\square}(s_{Q^{\square,sc}}(e_1^\vee - e_{n+1}^\vee) )
f^{\square}(s_{Q^{\square,sc}}(e_{2}^\vee - e_1^\vee) )
\]
or
\[
f^{\square}(s_{Q^{\square,sc}}(e_1^\vee - e_{n+1}^\vee))
f_-(s_{Q^{sc}}(e_{n+1}^\vee + e_{n+2}^\vee) )
f_-(s_{Q^{sc}}(e_{n+1}^\vee - e_{n+2}^\vee) )
f^{\square}(s_{Q^{\square,sc}}(e_1^\vee - e_{n+1}^\vee) )
f_+(s_{Q^{sc}}(e_{2}^\vee - e_1^\vee) ) .
\]
We now calculate its image under $\mathrm{mul}$. We observe that
\[
\mathrm{mul}\left(\dfrac{f_-(s_{Q^{sc}}(e_{n+1}^\vee - e_{n+2}^\vee)) f_-(s_{Q^{sc}}(e_{n+1}^\vee +e_{n+2}^\vee))}{f_+(s_{Q^{sc}}(e_{1}^\vee-e_{2}^\vee)) f_+(s_{Q^{sc}}(e_{1}^\vee +e_{2}^\vee))}\right)=1.
\]
From our choice of $f^{\square}(s_{Q^{\square,sc}}(e_1^\vee - e_{n+1}^\vee))$,
\[
\mathrm{mul}(f^{\square}(s_{Q^{\square,sc}}(e_1^\vee - e_{n+1}^\vee)))=1.
\]
Using these two facts, together with Lemma \ref{lem:formula on BD lift} part (2), we deduce that
\[
\mathrm{mul}(f^{\square}(s_{Q^{\square,sc}}(e_{1}^\vee + e_2^\vee)))=\mathrm{mul}(f_+(s_{Q^{sc}}(e_{1}^\vee + e_{2}^\vee))).
\]
By Lemma \ref{lem:very useful lemma}, this shows that
\[
f^{\square}(s_{Q^{\square,sc}}(e_{1}^\vee + e_2^\vee))=f_+(s_{Q^{sc}}(e_{1}^\vee +e_{2}^\vee)).
\]
This completes the proof.
\end{proof}

\begin{Rem}
It seems that the choice of $f^{\square}(s_{Q^{\square,sc}}(e_1^{\vee}-e_{n+1}^{\vee}))$ is quite delicate but it is not hard to see that this is almost the only choice. In the next section, we will show that $f^{\square}$ is also Galois invariant by carefully analyzing the Galois action on $\sce$. 
\end{Rem}

\subsubsection{Digression on the Chevalley system}\label{sec:Digression on Chevalley}

To show that $f^{\square}$ is $\Gamma$-equivariant, we need to have some understanding of the Galois action on $\mathcal{E}_{Q^{\square,sc}}$. The Galois action on $\sce_{Q^{\square,sc}}$ comes from the (possibly non-split) maximal $F$-torus $\bT^{\square,sc}$. We still denote this action as $\sigma$. We now would like to understand $\sigma(s_{Q^{\square,sc}}(e_1^\vee-e_{n+1}^\vee))$ for $\sigma \in \Gamma$. 

We start with some general facts about reductive groups over $F_s$. Let $\mathrm{Aut}(\bG_{F_s})$ be the automorphism group of $\bG_{F_s}$. Let $\mathrm{Inn}(\bG_{F_s})$ denotes the subgroup of inner automorphisms. Let $\tau\in \mathrm{Aut}(\bG_{F_s})$. Let $(\mathbf{B},\bT)$ be a choice of Borel subgroup and maximal torus which gives a based root datum $(X,\Phi,\Delta; Y,\Phi^\vee,\Delta^\vee)$. Then there exists $g_\tau\in \bG(F_s)$ such that $\Int(g_\tau)(\tau \mathbf{B})=\mathbf{B}$ and $\Int(g_\tau)(\tau \bT_{F_s})=\bT_{F_s}$. This induces an automorphism of $\Delta$. There is a split exact sequence
\[
1 \to \mathrm{Inn}(\bG_{F_s}) \to \mathrm{Aut}(\bG_{F_s}) \to \mathrm{Aut}(\Delta) \to 1.
\]
A splitting of this exact sequence is determined by a choice of $x_\al: \bG_a\simeq \bU_\al$ for $\al \in \Delta$.

We first would like to understand the action of $\Gamma$ on the root subgroup $x_{e_1-e_{n+1}}: \bG_m \simeq \bU_{e_1-e_{n+1}} \to \bG^{\square,sc}$. Note that $\bG^{\square,sc} \to \bG^{\square}$ restricts to an isomorphism on the root subgroup $\bU_{e_1-e_{n+1}}$. We will use the same notation for both root subgroups. It is sufficient to understand the action of $\Gamma$ on $x_{e_1-e_{n+1}}: \bG_m \simeq \bU_{e_1-e_{n+1}} \to \bG^{\square}$.

To have a good control on this, We have to relate the Galois action on $\bG \times \bG$ and $\bG^{\square}$. (We can also argue directly for $\bG^{sc}\times \bG^{sc}$ and $\bG^{\square,sc}$.)

Recall that $\bG$ is a subgroup of $\bGL_{W;D}$. Fix an isomorphism $D\otimes_F F_s\simeq \M_2(F_s)$. Put $z=\begin{pmatrix} 1 & 0 \\ 0 & 0 \end{pmatrix}\in \M_2(F_s)$ and set $W_z:=z(W\otimes_F F_s)$. The restriction $g\mapsto g|_{W_z}$ gives an isomorphism of $\bGL_{W;D}(F_s)$ onto the group $\bGL_{W_z}(F_s)$.

Let $\la \ , \ \ra_z$ be the restriction of $\la \ , \ \ra$ on $W_z$. Then $\la \ , \ \ra_z$ is an $F_s$-bilinear mapping with value in the one-dimensional $F_s$-vector space $zDz^\rho$, and it is non-degenerate and has the opposite symmetry as $\la \ , \ \ra$ under interchange of the two variables. The restriction $g\mapsto g|_{W_z}$ gives an isomorphism of $\bG_{F_s}$ onto the group $\bG_z:=\bG(W_z,\la \ , \ \ra_z)$. Let $\mathbf{T}_z$ be the image of $\mathbf{T}$ in $\mathbf{G}_z$. 

To summarize, we can identify the commutative diagram
\[
\begin{tikzpicture}[scale=1.5]
\node (A1) at (0,1) {$\bG_{F_s} \times \bG_{F_s}$};
\node (B1) at (3,1) {$\bG^{\square}_{F_s}$};
\node (A2) at (0,0) {$\bGL_{W;D,F_s} \times \bGL_{W;D,F_s}$};
\node (B2) at (3,0) {$\bGL_{W^{\square};D,F_s}$};
\path[->,font=\scriptsize,>=angle 90]
(A1) edge  (B1)
(A2) edge  (B2)
(A1) edge (A2)
(B1) edge (B2);
\end{tikzpicture}
\]
with
\[
\begin{tikzpicture}[scale=1.5]
\node (A1) at (0,1) {$\bG_z \times \bG_z$};
\node (B1) at (3,1) {$\bG_z^{\square}$};
\node (A2) at (0,0) {$\bGL_{W_z;F_s} \times \bGL_{W_z;F_s} $};
\node (B2) at (3,0) {$\bGL_{W_z^{\square};F_s}$};
\path[->,font=\scriptsize,>=angle 90]
(A1) edge  (B1)
(A2) edge  (B2)
(A1) edge (A2)
(B1) edge (B2);
\end{tikzpicture}
\]

Let $\mathbb{G}$ be the $F$-split group with a maximal split torus $\mathbb{T}$ so that the root system of $(\mathbb{G},\mathbb{T})$ is the same as the root system of $(\bG_{F_s},\bT_{F_s})$. The group $\bGL_{W_z;F_s}$ is split and has an $F$-structure $\bGL_{W_z;F}$. We realize $\mathbb{G}$ as a subgroup of $\bGL_{W_z;F}$. Then there exists $h\in \bGL_{W_z}(F_s)$ such that
\[
h\cdot \bG_z(F_s) \cdot h^{-1}= \mathbb{G}(F_s), \qquad h \cdot \bT_z(F_s)\cdot h^{-1}= \mathbb{T}(F_s).
\]
From this we can transfer the based root datum for $(\bG,\bT)$ to a based root datum of $(\mathbb{G},\mathbb{T})$. Thus, we obtain a basis of $Y_{\mathbb{G}}$ from a basis of $Y$. We write it as $(\mathbbm{e}_1,\cdots, \mathbbm{e}_n)$.

The action of $\Gamma$ on $(\bG,\bT)$ gives an action of $\Gamma$ on $(\mathbb{G},\mathbb{T})$:
\[
\mathbb{G}(F_s) \to \mathbb{G}(F_s), \qquad g\mapsto \tilde \sigma(g):=\Int(h) \circ \sigma\circ \Int(h^{-1})(g).
\]
The map $\sigma^{\dagger}:=\sigma^{-1} \circ (\Int(h) \circ \sigma\circ \Int(h^{-1})) \in \mathrm{Aut}(\mathbb{G}_{F_s})$. It is easy to see that $\sigma^{\dagger}=\Int(\sigma(h)h^{-1}) $. The element $\sigma^{\dagger}$ induces an action on the root system of $(\mathbb{G}, \mathbb{T})$, which we again denote as $\sigma^{\dagger}$. Then there exists a lift $w_\sigma$ of a Weyl group element $\mathbf{w}_\sigma$ of $(\mathbb{G}, \mathbb{T})$ such that $\mathbf{w}_\sigma(\sigma^{\dagger}(\Delta))=\Delta$.

A Chevalley system for $(\bG_{z}, \bT_{z})$ can also bee translated to a Chevalley system for $(\mathbb{G},\mathbb{T})$. The action of $\sigma$ on the root system of $(\bG_{z}, \bT_{z})$ and the pinning can be read from the action of $\sigma^\dagger$ on $(\mathbb{G},\mathbb{T})$.

We write the action down more explicitly. Let $y:\bG_m \to \bT_{z}$ be a cocharacter of $\bT_{z}$. Then the action of $\sigma$ on $y$ is given by:
\[
\bG_m  \to \bG_m \to \bT_{z} \to \bT_{z},\qquad t\mapsto \sigma^{-1}(t) \mapsto y(\sigma^{-1}(t)) \mapsto \sigma(y(\sigma^{-1}(t))).
\]
Let $x_\al:\bG_a \to \bU$ be a root subgroup. Then the action of $\sigma$ on $\alpha$ is given byL
\[
\bG_a  \to \bG_a \to \bU_{F_s} \to \bU_{F_s},\qquad t\mapsto \sigma^{-1}(t) \mapsto x_\al(\sigma^{-1}(t)) \mapsto \sigma(x_\al(\sigma^{-1}(t))).
\]
When transferring this action to $(\mathbb{G},\mathbb{T})$, these actions are given by the following:
\begin{itemize}
\item $y:\bG_m \to \mathbb{T}$ is sent to the element $\sigma^{\dagger}(y)$, defined as 
\[
\bG_m \to \mathbb{T},\qquad t \mapsto \tilde \sigma(y(\sigma^{-1}(t))).
\]
\item $x_\al:\bG_a\to \mathbb{U}$ is sent to $x_{\sigma^{\dagger}(\al)}$, defined as
\[
\bG_a \to \mathbb{U},\qquad t\mapsto \tilde\sigma(x_\alpha(\sigma^{-1}(t))).
\]
\end{itemize}
In this way, a based root datum for $(\bG,\bT)$ determines one for $(\mathbb{G},\mathbb{T})$. 

The group $\mathbb{G}$ is the connected component of some orthogonal group $\mathbb{O}_{2n}$. It is not hard to check that $\sigma(h)h^{-1}\in \mathbb{O}_{2n}(F_s)$. As a consequence, the automorphism of the Dynkin diagram given by $w_\sigma \circ \sigma^\dagger$ is either the identity or the isomorphism permuting $\mathbbm{e}_{n-1}-\mathbbm{e}_n$ and $\mathbbm{e}_{n-1}+\mathbbm{e}_n$. (In other words, triality does not appear in the case of $D_4$.) In either case, this isomorphism can be realized by the conjugation given by a lift $\tau_\sigma$ of a Weyl group element in $\bGL_{W_z}$.

We now have two different pinnings for $(\mathbb{G}, \mathbb{T})$. The first is  $x_{\alpha}: \mathbf{G}_a \to \mathbb{U}$ which is translated from $(\mathbf{G}, \mathbf{T})$. The other is $\Int(\tau_\sigma w_\sigma) \circ x_{\sigma^{\dagger}(\al)}$. They might not be the same. But we can choose $t_\sigma \in \mathbb{T}(F_s)$ such that
\[
\Int(t_\sigma \tau_\sigma w_\sigma) \circ x_{\sigma^{\dagger}(\al)} = x_{\alpha}, \qquad \al \in \Delta.
\]
This implies that $\sigma^{\dagger}=\Int(w_\sigma^{-1} \tau_\sigma^{-1} t_\sigma^{-1})$.

\subsubsection{The action on $\mathbf{G}^{\square}$}

We now have that
\[
\iota(h,h) \cdot \bT^{\square}(F_s) \cdot \iota(h,h)^{-1} \subset \iota(h,h) \cdot \bG^{\square}(F_s) \cdot \iota(h,h)^{-1}
\]
is $\mathbb{T}^{\square}(F_s)\subset \mathbb{G}^{\square}(F_s)$ for a maximal split torus $\mathbb{T}^{\square}$ inside a split group $\mathbb{G}^{\square}$. As in the case of $\bG$, here we realize $\mathbb{G}^{\square}(F_s)\subset \bGL_{W_z^{\square}}(F_s)$.
We can read the Galois action on $(\bG^{\square},\bT^{\square})$ from
\[
\sigma^{\dagger}= \Int( \iota(\sigma(h)h^{-1},\sigma(h)h^{-1})) \in \mathrm{Aut}(\mathbb{G}^{\square}(F_s)).
\]
This map preserves $\mathbb{T}^{\square}(F_s)$. From our discussion above, we know that
\[
\Int( \iota(\sigma(h)h^{-1},\sigma(h)h^{-1}))  = \Int(\iota(w_\sigma^{-1} \tau_\sigma^{-1} t_\sigma^{-1}, w_\sigma^{-1} \tau_\sigma^{-1} t_\sigma^{-1})).
\]

\begin{Lem}
For any $\sigma\in \Gamma$,
\[
\sigma^\dagger \circ x_{\mathbbm{e}_1 - \mathbbm{e}_{n+1}} = \Int(\iota(w_\sigma^{-1}, w_\sigma^{-1}))\circ x_{\mathbbm{e}_1 - \mathbbm{e}_{n+1}}.
\]
\end{Lem}

\begin{proof}
It is easy to check that $\Int(\iota(t_\sigma^{-1},t_\sigma^{-1}))$ acts trivially on $x_{\mathbbm{e}_1 - \mathbbm{e}_{n+1}}$. Moreover, $\Int(\iota(\tau_\sigma^{-1},\tau_\sigma^{-1}))$ acts trivially on $x_{\mathbbm{e}_1 - \mathbbm{e}_{n+1}}$ as well. This completes the proof. 
\end{proof}

We translate the above lemma back to the case of $(\mathbf{G}, \mathbf{T})$.  We deduce that $\sigma(x_{e_1-e_{n+1}})=\tilde w_\sigma\circ x_{e_1-e_{n+1}}$ for some lift $\tilde w_\sigma$ of a Weyl group element $\tilde{\mathbf{w}}_\sigma$ for $(\bG^{\square},\bT^{\square})$. In other words, the Galois action on $e_1-e_{n+1}$ is the same as the action by some Weyl group element. 

\begin{Lem}\label{lem:galois reduced to weyl}
For any $\sigma \in \Gamma$,
\[
\sigma(s_{Q^{\square,sc}}(e_1^\vee - e_{n+1}^\vee)) = \tilde w_\sigma \cdot s_{Q^{\square,sc}}(e_1^\vee - e_{n+1}^\vee) \cdot \tilde w_\sigma^{-1}=s_{Q^{\square,sc}}(\tilde{\mathbf{w}}_\sigma(e_1^\vee - e_{n+1}^\vee)).
\]
\end{Lem}

\begin{proof}
This follows from Lemma \ref{lem:weyl action on E I} and the discussion above. Note that from Lemma \ref{lem:divisible by 2 SO even}, we always have $\epsilon_{\alpha, \beta}^{Q(\beta^{\vee})}=1$. 
\end{proof}

\begin{Lem}
For $\sigma\in \Gamma$, there exists $i$ and a sign such that $\sigma(e_1-e_{n+1})$ is of the form $\pm(e_i-e_{n+i})$.
\end{Lem}

\begin{proof}
Suppose $\sigma(e_1)=\sum_j a_j e_j$ for some $a_j\in \bZ$. Since the action of $\sigma$ are the same for both copies of $\bG$, we have $\sigma(e_{n+1})=\sum_j a_j e_{n+j}$. We know
\[
\sigma(e_1-e_{n+1}) = \sum_j a_j(e_j-e_{n+j})
\]
must be a root of $\bG^{\square}$. Thus, it must be of the form $\pm(e_i-e_{n+i})$.
\end{proof}

\begin{Lem}\label{lem:mul are 1 under galois}
If $\sigma(e_1-e_{n+1})=e_i-e_{n+i}$ for some $i$, then 
\[
\sigma(s_{Q^{\square,sc}}(e_1^\vee - e_{n+1}^\vee)) / s_{Q^{\square,sc}}(e_1^\vee - e_{n+1}^\vee)\in \pr_{\ast}(\sce_{Q^{sc}}\oplus \sce_{Q^{sc}}),
\]
and 
\[
\mathrm{mul}(\sigma(s_{Q^{\square,sc}}(e_1^\vee - e_{n+1}^\vee)) / s_{Q^{\square,sc}}(e_1^\vee - e_{n+1}^\vee))=1.
\]
\end{Lem}

\begin{proof}
The first statement is straightforward. Using
\[
(e_i^\vee - e_{n+i}^\vee) +(e_1^\vee - e_i^\vee) = (e_1^\vee - e_{n+1}^\vee) + (e_{n+1}^\vee-e_{n+i}^\vee).
\]
and Lemma \ref{lem:additivity of sQ}, we have
\[
s_{Q^{\square,sc}}(e_i^\vee - e_{n+i}^\vee) \cdot s_{Q^{\square,sc}}(e_1^\vee - e_i^\vee) = s_{Q^{\square,sc}}(e_1^\vee - e_{n+1}^\vee) \cdot s_{Q^{\square,sc}}(e_{n+1}^\vee-e_{n+i}^\vee).
\]
Now the result follows the fact that $\mathrm{mul}(s_{Q^{\square,sc}}(e_1^\vee - e_i^\vee)/s_{Q^{\square,sc}}(e_{n+1}^\vee-e_{n+i}^\vee))=1$.
\end{proof}

\subsubsection{Galois equivariance}\label{sec:galois equi so even}

We are now ready to prove that $f^{\square}$ is Galois equivariant. 
\begin{Prop}
We have that 
\begin{equation}\label{eq:galois equi in so even}
f^{\square}(\sigma(s_{Q^{\square,sc}}(\al^\vee))) =\sigma(f^{\square}(s_{Q^{\square,sc}}(\al^\vee)))
\end{equation}
for $\al^\vee\in \Delta_{\bG^{\square}}^{\vee}$ and $\sigma \in \Gamma$. 
\end{Prop}

\begin{proof}
The only non-trivial case is $\al^\vee=e_1^{\vee} - e_{n+1}^{\vee}$.

Both sides in \eqref{eq:galois equi in so even} project to $\sigma(\al^\vee)$. Thus to show \eqref{eq:galois equi in so even}, it suffices to show that they are the same under the map $\mathrm{mul}$. It is easy to show that for any $\sigma\in \Gamma$,
\[
\mathrm{mul}(\sigma(f^{\square}(s_{Q^{\square,sc}}(e_1^{\vee} - e_{n+1}^{\vee})))=1.
\]

We now calculate $\mathrm{mul}(f^{\square}(\sigma(s_{Q^{\square,sc}}(e_1^{\vee} - e_{n+1}^{\vee}))) )$. We have three cases to consider. First, if $\sigma(e_1^{\vee} - e_{n+1}^{\vee})=\tilde{\mathbf{w}}_\sigma(e_1^\vee - e_{n+1}^\vee)=-(e_1^{\vee} - e_{n+1}^{\vee})$, then by Lemma \ref{lem:formula on BD lift} and \ref{lem:galois reduced to weyl},
\[
\sigma(s_{Q^{\square,sc}}(e_1^{\vee} - e_{n+1}^{\vee}))= s_{Q^{\square,sc}}(\tilde{\mathbf{w}}_\sigma(e_1^\vee - e_{n+1}^\vee))=s_{Q^{\square,sc}}(-(e_1^\vee - e_{n+1}^\vee))=s_{Q^{\square,sc}}(e_1^\vee - e_{n+1}^\vee)^{-1},
\]
and this implies that $\mathrm{mul}(f^{\square}(\sigma(s_{Q^{\square,sc}}(e_1^{\vee} - e_{n+1}^{\vee}))))=1$. 

We now assume that $\sigma(e_1-e_{n+1})=e_i-e_{n+i}$ for some $i$. Then by Lemma \ref{lem:mul are 1 under galois} and our choice of $f^{\square}(s_{Q^{\square,sc}}(e_1^{\vee} - e_{n+1}^{\vee}))$,
\[
\mathrm{mul}(f^{\square}(\sigma(s_{Q^{\square,sc}}(e_1^{\vee} - e_{n+1}^{\vee}))) =\mathrm{mul}(f^{\square}(s_{Q^{\square,sc}}(e_1^{\vee} - e_{n+1}^{\vee})))=1.
\]

Finally, we have to consider the case $\sigma(e_1-e_{n+1})=-(e_i-e_{n+i})$ for some $i\neq 1$. This can be proved by combining arguments in the previous two cases. This completes the proof.
\end{proof}

\subsection{Unitary groups}

We now consider the unitary group case. Recall that $\bG_{F_s}=\bGL_{n,F_s}$. We choose a standard basis of $\bT_{F_s}$ so that $Y=\mathrm{Span}\{e_1^{\vee},\cdots,e_n^{\vee}\}$ with the following set of simple roots:
\[
\{e_1^\vee-e_2^\vee,\cdots, e_{n-1}^\vee - e_n^\vee\}.
\]

Let $Q$ be a Weyl invariant quadratic form on $Y$. Then $Q$ is determined by the following two integers $p$ and $q$:
\[
B_Q(e_i^\vee,e_i^\vee)=2p,\qquad B_Q(e_i^\vee,e_j^\vee)=q \text{ for }i\neq j.
\]
Then for any coroot $\al^\vee$, $Q(\al^\vee)=2p-q$. Since we assume that $Q$ is decomposable, we have $q=0$.

The group $\bG^{\square}$ has a maximal $F$-torus $\bT^{\square}:=\bT\times \bT$. The cocharacter lattice of $\bT^{\square}$ over $F_s$ is $Y^{\square}=Y\oplus Y$. We choose standard basis so that $Y_+=\mathrm{Span}\{e_1^\vee, \cdots, e_n^\vee\}$ and $Y_-=\mathrm{Span}\{e_{n+1}^\vee, \cdots, e_{2n}^\vee\}$. The choice of simple roots is given as above. We choose the following set of simple roots for $\bG^{\square}$:
\[
\{
e_1^\vee - e_2^\vee ,\cdots, e_{2n-1}^\vee -e_{2n}^\vee.
\}
\]
Since we assume that $q=0$, $Q^{\square}$ is a direct sum $Q\oplus Q$ on $Y\oplus Y$.

\subsubsection{Construction of BD data}
Let $(Q,\sce,f)$ be the Galois equivariant BD data for $\obG$. We now construct a Galois invariant BD data $(Q^{\square},\sce^{\square},f^{\square})$ for $\bG^{\square}$.

Given a quadratic form $Q$ on $Y$ which is determined by an integer $p$ as above, we define $Q^{\square}:=Q\oplus Q$. The quadratic form is Weyl invariant and Galois invariant. 

\begin{Rem}
If $q\neq 0$, the quadratic form $Q\oplus Q$ on $Y\oplus Y$ is not $W$-invariant. We have to choose a different one. The argument in the rest of section will require some modification in order to handle this case.
\end{Rem}

We can again take the direct sum of two copies of $\sce$ and push it out via the multiplication map to obtain
\[
1\to F_s^\times \to \pr_{\ast}(\sce \oplus \sce ) \to Y^{\square} \to 1.
\]
The commutator is $[y_1,y_2]=(-1)^{B_{Q^{\square}}(y_1,y_2)}$. We take $\sce^{\square}=\pr_{\ast}(\sce \oplus \sce )$ and this is the second BD invariant.

We can now proceed as in the case of even orthogonal groups. We will not repeat the definitions of these notations here. We now have to define $f^{\square}: \sce_{Q^{\square,sc}} \to \pr_{\ast}(\sce \oplus \sce)$ such that its composition with the map $\pr_{\ast}(\sce_{Q^{sc}}\oplus \sce_{Q^{sc}}) \to \sce_{Q^{\square,sc}}$ gives
\[
f_+\oplus f_-: \pr_{\ast}(\sce_{Q^{sc}} \oplus \sce_{Q^{sc}})\to \pr_{\ast}(\sce \oplus \sce).
\]
The map is already determined on the image of $\pr_{\ast}(\sce_{Q^{sc}}\oplus \sce_{Q^{sc}})$ in $\sce_{Q^{\square,sc}}$. One only has to determine $f^{\square}(s_{Q^{\square,sc}}(e_n^{\vee}-e_{n+1}^{\vee}))$ so that the map is Galois equivariant.

\subsubsection{Galois equivariance}

We can verify Galois equivariance using the argument in Section \ref{sec:galois equi so even}. Here we give another proof using the Hilbert's theorem 90. We would like to show that there exists $f^{\square}(s_{Q^{\square,sc}}(e_n^{\vee}-e_{n+1}^{\vee}))$ such that
\[
f^{\square}(\sigma(s_{Q^{\square,sc}}(e_n^{\vee}-e_{n+1}^{\vee}))) = \sigma(f^{\square}(s_{Q^{\square,sc}}(e_n^{\vee}-e_{n+1}^{\vee})))
\]
for all $\sigma \in \Gamma$.

For ease of notations, we write $a=s_{Q^{\square,sc}}(e_n^{\vee}-e_{n+1}^{\vee})$. We take an arbitrary $f^{\square}$ and define a function $c:\Gamma \to F_s^{\times}$ as follows:
\[
c(\sigma)=\dfrac{f^{\square}(\sigma(a))}{\sigma(f^{\square}(a))}.
\]
We show that $c$ is a $1$-cocycle. In other words, we prove the following result.
\begin{Lem}
For any $\sigma_1,\sigma_2 \in \Gamma$,
\[
c(\sigma_1 \sigma_2) = f(\sigma_1) \cdot \sigma_1(f(\sigma_2)).
\]
\end{Lem}

\begin{proof}
We write
\[
c(\sigma_1 \sigma_2) = \dfrac{f^{\square}(\sigma_1\sigma_2(a))}{\sigma_1\sigma_2(f^{\square}(a))} = \dfrac{f^{\square}(\sigma_1\sigma_2(a))}{\sigma_1(f^{\square}(\sigma_2(a)))} \cdot \dfrac{\sigma_1(f^{\square}(\sigma_2(a)))}{\sigma_1\sigma_2(f^{\square}(a))}.
\]
It suffices to show that
\[
\dfrac{f^{\square}(\sigma_1\sigma_2(a))}{\sigma_1(f^{\square}(\sigma_2(a)))} = \dfrac{f^{\square}(\sigma_1(a))}{\sigma_1(f^{\square}(a))}
\]
or
\[
\dfrac{f^{\square}(\sigma_1\sigma_2(a))}{f^{\square}(\sigma_1(a))} = \dfrac{\sigma_1(f^{\square}(\sigma_2(a)))}{\sigma_1(f^{\square}(a))}
\]
Note that $\sigma_2(a)/a$ projects to $Y^{sc}\oplus Y^{sc}$. Thus, the left-hand side is
\[
f^{\square}(\sigma_1(\sigma_2(a)/a))=\sigma_1(f^{\square}(\sigma_2(a)/a)).
\]
This proves the result.
\end{proof}

The Hilbert's theorem 90 says that $H^1(\Gamma,F_s^\times)=1$. In other words, a $1$-cocycle must be a coboundary. This means that there exists $x\in F_s^\times$ such that $c(\sigma)=\sigma(x)/x$.

We now define
\[
\tilde f^{\square}(a)=f^{\square}(a)x.
\]
Then
\[
\tilde f^{\square}(\sigma(a))=f^{\square}(\sigma(a))x=\sigma(f^{\square}(a)) c(\sigma) \cdot x=\sigma(f^{\square}(a)) \sigma(x)=\sigma(\tilde f^{\square}(a)).
\]
This implies that $\tilde f^{\square}$ is $\Gamma$-equivariant.

\subsection{Special odd orthogonal groups}

We now treat the case of $\bG=\bSO_{2n+1}$. We first begin with some discussion of the BD data. We can write $Y=\mathrm{Span}\{e_1^{\vee},\cdots,e_n^{\vee}\}$ and let
\[
\{e_1^{\vee}-e_2^{\vee}, \cdots, e_{n-1}^{\vee}-e_n^{\vee},2e_n^{\vee}\}
\]
be the coroots of $\bSO_{2n+1}$. A $W$-invariant quadratic form $Q$ on $Y$ is determined by its value on a short coroot. Let $Q(\al_1^\vee)=a$. As in the even orthogonal case, we can similarly prove the following.

\begin{Lem}\label{lem:divisible by 2 SO odd}
We have that $2\mid a$.
\end{Lem}

Observe that $\bT\times \bT$ is an $F$-torus in $\bG^{\square}$ but not a maximal torus. Let $\bT^{\square}\supset \bT\times \bT$ be a maximal $F$-torus of $\bG^{\square}$.

We write  $Y_+=\mathrm{Span}\{e_1^{\vee},\cdots,e_n^{\vee}\}$, $Y_-=\mathrm{Span}\{e_{n+1}^{\vee},\cdots,e_{2n}^{\vee}\}$. We have
\[
\begin{aligned}
Y_+^{sc}=&\mathrm{Span}\{e_1^{\vee}-e_2^{\vee}, \cdots, e_{n-1}^{\vee}-e_n^{\vee},2e_n^{\vee}\}\\
Y_-^{sc}=&\mathrm{Span}\{e_{n+1}^{\vee}-e_{n+2}^{\vee}, \cdots, e_{2n-1}^{\vee}-e_{2n}^{\vee},2e_{2n}^{\vee}\}\\
Y^{\square,sc}=&\mathrm{Span}\{e_1^{\vee}-e_2^{\vee}, \cdots, e_{2n}^{\vee}-e_{2n+1}^{\vee},e_{2n}^{\vee}+e_{2n+1}^{\vee}\}.
\end{aligned}
\]

Let $Q^{\square}$ be $W$-invariant quadratic form on $Y^{\square}$.
\begin{Lem}
The restriction of $Q^{\square}$ to $Y_+\oplus Y_-$ is a direct sum of two quadratic forms $Q_+\oplus Q_-$. And we have $Q_+=Q_-$.

Conversely, given $Q=Q_+=Q_-$, then there is a unique $W$-invariant quadratic form $Q^{\square}$ which restricts to $Q_+\oplus Q_-$ on $Y_+\oplus Y_-$.
\end{Lem}

\begin{proof}
This is straightforward. 
\end{proof}

Let $\obG$ be a multiplicative $\bK_2$-torsor on $\bG$ with BD data $(Q,\sce,f)$. We now construct a BD data $(Q^{\square},\sce^{\square},f^{\square})$ for $\bG^{\square}$. We only explain the difference in this case but will not repeat all the details.

The construction of $Q^{\square}$ is straightforward since such a quadratic form is determined by its values on its short coroot.

We now define the second BD invariant. Note that $Y^{\square}/Y\oplus Y=\bZ\cdot e_{2n+1}^\vee$. Consider
\[
\sce\oplus \sce\oplus (F_s^\times \times \bZ)
\]
with the following multiplication:
\[
(e_1,e_2,(x,a))\cdot (e'_1,e'_2,(x',a')): =(e_1e_1',e_2e_2',(xx'(-1)^{B_{Q^{\square}}((y_1,y_2),a' \cdot e_{2n+1}^\vee)},a+a'))
\]
Here, $(y_1,y_2)$ is the image of $(e_1,e_2)$ under $\mathcal{E} \to Y$. This defines an exact sequence
\[
1\to F_s^\times \oplus F_s^\times \oplus F_s^\times \to \sce\oplus \sce\oplus (F_s^\times \times \bZ) \to Y^{\square}=Y\oplus Y\oplus \bZ\cdot e_{2n+1}^\vee \to 1.
\]
Pushing out by the product map $\pr: F_s^\times \oplus F_s^\times \oplus F_s^\times \to F_s^\times$ gives
\[
1\to F_s^\times \to \sce^{\square} \to Y^{\square} \to 1.
\]
One can verify that the commutator map is given by $(-1)^{B_{Q^{\square}}(y_1,y_2)}$. We also have a commutative diagram \[
\begin{tikzpicture}[scale=1.5]
\node (A1) at (0,1) {$1$};
\node (B1) at (1.5,1) {$F_s^\times$};
\node (C1) at (3.5,1) {$\pr_{\ast}(\sce_{Q^{sc}}\oplus \sce_{Q^{sc}})$};
\node (D1) at (5.5,1) {$Y^{sc}\oplus Y^{sc}$};
\node (E1) at (7,1) {$1$};
\node (A2) at (0,0) {$1$};
\node (B2) at (1.5,0) {$F_s^\times$};
\node (C2) at (3.5,0) {$\sce^{\square}$};
\node (D2) at (5.5,0) {$Y^{\square}$};
\node (E2) at (7,0) {$1$};
\path[->,font=\scriptsize,>=angle 90]
(A1) edge  (B1)
(B1) edge  (C1)
(C1) edge  (D1)
(D1) edge  (E1)
(A2) edge  (B2)
(B2) edge  (C2)
(C2) edge  (D2)
(D2) edge  (E2)
(C1) edge (C2)
(D1) edge (D2)
(B1) edge (B2);
\end{tikzpicture}
\]

To construct the third BD invariant, one has to construct $f^{\square}: \sce_{Q^{\square,sc}} \to \sce^{\square}$ which extends the map $\pr_{\ast}(\sce_{Q^{sc}} \oplus \sce_{Q^{sc}}) \to \sce^{\square}$. We only have to specify the values
\[
f^{\square}(s_{Q^{\square,sc}}(e_1^{\vee} - e_{n+1}^{\vee}))  \text{ and }
f^{\square}(s_{Q^{\square,sc}}(e_{2n}^{\vee} - e_{2n+1}^{\vee}))
\]
so that $f^{\square}$ is Galois equivariant. To choose the first value, we use the argument presented in the $\bSO_{2n}$ case. The choice for the second value is identical to the unitary case. This completes the proof. 

\section{$L$-functions}\label{sec:construction of Speh}

So far we only give a global zeta integral which represents an Euler product, but have not said anything regarding the $L$-functions obtained from the twisted doubling integrals. The construction relies on the construction of representations of type $(k,m)_D$. In the linear case, a good source of such representations are the generalized Speh representations. Here we present a conjectural picture.  Further investigations are necessary in order to gain a complete understanding of the local and global theory.

We make a couple of assumptions to simplify the situation in our discussion. We assume that $D=F$, so that only the group $\bGL_m$ will be involved. We also assume that the quadratic form $Q$ appearing in the BD data is decomposable. So we will not have any problems regarding parabolic induction. As we point out earlier, to treat the non-decomposable case, we need to have a suitable version of `metaplectic tensor product'.

Fix an integer $p$. This determines a $W$-invariant decomposable quadratic form on the cocharacter lattice for every $\bGL_m$. Let $n_Q=n/\gcd(n,Q(\al^\vee))$ for any $\al^\vee \in \Delta^\vee$ if $m\geq 2$; and $n_Q=n/\gcd(n,2a)$ for $a=Q(e_1^\vee)$ and $m=1$. The twisted doubling integrals relies on the following construction of the inducing data in the Eisenstein series. To be more precise, we would like to have
\begin{equation}\label{eq:global speh}
\tau \in \Irr_{cusp}^{u}(\GL_k(\ba)) \to \theta^{(n)}(\tau,m)\in \Irr^u(\oGL^{(n)}_{k m n_Q}(\ba)).
\end{equation}
For every local place $v$, the local analogue is given by
\[
\tau_v \in \Irr_{gen}^{u}(\GL_{k,v}) \to \theta^{(n)}(\tau_v,m)\in \Irr^u(\oGL^{(n)}_{k m n_Q,v}).
\]
Here, the superscript $u$ means that only unitary representations are considered and the subscript $gen$ means generic representations.

We expect the following list of properties:
\begin{itemize}
\item The construction is local-to-global compatible: $\theta^{(n)}(\tau,m)=\otimes_v' \theta^{(n)}(\tau_v,m)$ if $\tau=\otimes_v' \tau_v$.
\item The representation $\theta^{(n)}(\tau,m)$ is of type $(kn_Q,m)$.
\item For all $k, n$ and the multiplicative $\bK_2$-torsor on $\bGL_k$ determined by the integer $p$ as above, assume that there exists a `Shimura-type lift' from $\Irr(\oGL_k^{(n)}(\ba)) \to \Irr(\GL_k(\ba))$ which is also local to global compatible. If $\tau$ does not lie in the image of the Shimura lift from $\oGL^{(n')}_k(\ba)$ to $\GL_k(\ba)$ for any $n'\mid n$, then the lift is cuspidal. If it does, then this can be constructed using residues of Eisenstein series.
\end{itemize}

This is also discussed in \cite{Suzuki98} and \cite{Ginzburg}.

\begin{Rem}
We also expect more properties of this construction to be valid. Such properties are motivated by the local and global theory of the twisted doubling integrals. For instance, this construction should satisfies a multiplicativity with respect to the parabolic induction. This will be used in order to establish the multiplicativity of $\gamma$-factors. 
\end{Rem}

\begin{Rem}
Instead of \eqref{eq:global speh}, one might consider
\[
\tau \in \Irr_{cusp}^{u}(\oGL_k(\ba)) \to \theta^{(n)}(\tau,m)\in \Irr^u(\oGL^{(n)}_{k m n_Q}(\ba)),
\]
given by residues of Eisenstein series. However, the orbit of $\theta^{(n)}(\tau,m)$ might not be of $(kn_Q,m)$ due to the existence of cuspidal theta representations. We refer the reader to \cite{FG17} Section 3.2 for some discussion on this matter.
\end{Rem}

As a consequence, we expect that the twisted doubling integrals give an integral representation for the tensor product $L$-function for $\obG \times \bGL_k$. (In the unitary case, it would be $\obG \times \Res_{E/F}(\bGL_k)$.) Our formulation is slightly different from \cite{Kaplan}. Also note that in \cite{Gao18} the Langlands-Shahidi type $L$-functions appear in the constant terms of Eisenstein series, in which the tensor product $L$-function for $\obG\times \overline{\bGL}_k$ is obtained for split classical groups $\bG$, among many other interesting $L$-functions. It would be interesting to relate the $L$-functions obtained from these constructions. 

\bibliographystyle{alpha}


\bibliography{BD}

\begin{thebibliography}{CFGK19}

\bibitem[BD01]{BD01}
Jean-Luc Brylinski and Pierre Deligne.
\newblock Central extensions of reductive groups by {$\bold K_2$}.
\newblock {\em Publ. Math. Inst. Hautes \'Etudes Sci.}, (94):5--85, 2001.

\bibitem[BG92]{BG92}
Daniel Bump and David Ginzburg.
\newblock Symmetric square {$L$}-functions on {${\rm GL}(r)$}.
\newblock {\em Ann. of Math. (2)}, 136(1):137--205, 1992.

\bibitem[BT84]{BT84}
F.~Bruhat and J.~Tits.
\newblock Groupes r\'{e}ductifs sur un corps local. {II}. {S}ch\'{e}mas en
  groupes. {E}xistence d'une donn\'{e}e radicielle valu\'{e}e.
\newblock {\em Inst. Hautes \'{E}tudes Sci. Publ. Math.}, (60):197--376, 1984.

\bibitem[Bum97]{Bump97}
Daniel Bump.
\newblock {\em Automorphic forms and representations}, volume~55 of {\em
  Cambridge Studies in Advanced Mathematics}.
\newblock Cambridge University Press, Cambridge, 1997.

\bibitem[Cai21]{Cai21}
Yuanqing Cai.
\newblock Twisted doubling integrals for classical groups.
\newblock {\em Math. Z.}, 297(3-4):1075--1104, 2021.

\bibitem[CFGK16]{CFGK2016}
Yuanqing Cai, Solomon Friedberg, David Ginzburg, and Eyal Kaplan.
\newblock Doubling constructions for covering groups and tensor product
  {L}-functions.
\newblock {\em arXiv:1601.08240}, 2016.

\bibitem[CFGK19]{CFGK19}
Yuanqing Cai, Solomon Friedberg, David Ginzburg, and Eyal Kaplan.
\newblock Doubling constructions and tensor product {$L$}-functions: the linear
  case.
\newblock {\em Invent. Math.}, 217(3):985--1068, 2019.

\bibitem[FG17]{FG17}
Solomon Friedberg and David Ginzburg.
\newblock On the genericity of {E}isenstein series and their residues for
  covers of {$GL_m$}.
\newblock {\em Int. Math. Res. Not. IMRN}, (4):1000--1012, 2017.

\bibitem[Gan12]{Gan12}
Wee~Teck Gan.
\newblock Doubling zeta integrals and local factors for metaplectic groups.
\newblock {\em Nagoya Math. J.}, 208:67--95, 2012.

\bibitem[Gao18]{Gao18}
Fan Gao.
\newblock The {L}anglands-{S}hahidi {$L$}-functions for {B}rylinski-{D}eligne
  extensions.
\newblock {\em Amer. J. Math.}, 140(1):83--137, 2018.

\bibitem[GG18]{GG18}
Wee~Teck Gan and Fan Gao.
\newblock The {L}anglands-{W}eissman program for {B}rylinski-{D}eligne
  extensions.
\newblock {\em Ast\'erisque}, (398):187--275, 2018.
\newblock L-groups and the Langlands program for covering groups.

\bibitem[GGS17]{GGS17}
Raul Gomez, Dmitry Gourevitch, and Siddhartha Sahi.
\newblock Generalized and degenerate {W}hittaker models.
\newblock {\em Compos. Math.}, 153(2):223--256, 2017.

\bibitem[GGW18]{GGW18}
Wee~Teck Gan, Fan Gao, and Martin~H. Weissman.
\newblock L-groups and the {L}anglands program for covering groups: a
  historical introduction.
\newblock Number 398, pages 1--31. 2018.
\newblock L-groups and the Langlands program for covering groups.

\bibitem[Gin19]{Ginzburg}
David Ginzburg.
\newblock Tensor product ${L}$-functions on metaplectic covering groups of
  ${GL}_r$.
\newblock {\em arXiv:1908.07720}, 2019.

\bibitem[GSSar]{GSS}
Fan Gao, Freydoon Shahidi, and Dani Szpruch.
\newblock Gamma factors for genuine principal series of covering groups.
\newblock {\em Mem. Amer. Math. Soc.}, to appear.

\bibitem[Kap19]{Kaplan}
Eyal Kaplan.
\newblock Doubling constructions and tensor product {$L$}-functions: coverings
  of the symplectic group.
\newblock {\em arXiv:1902.00880}, 2019.

\bibitem[KP84]{KP84}
D.~A. Kazhdan and S.~J. Patterson.
\newblock Metaplectic forms.
\newblock {\em Inst. Hautes \'Etudes Sci. Publ. Math.}, (59):35--142, 1984.

\bibitem[Li14]{Li14}
Wen-Wei Li.
\newblock La formule des traces pour les rev\^{e}tements de groupes
  r\'{e}ductifs connexes. {I}. {L}e d\'{e}veloppement g\'{e}om\'{e}trique fin.
\newblock {\em J. Reine Angew. Math.}, 686:37--109, 2014.

\bibitem[Mez04]{Mezo04}
Paul Mezo.
\newblock Metaplectic tensor products for irreducible representations.
\newblock {\em Pacific J. Math.}, 215(1):85--96, 2004.

\bibitem[Moo68]{Moore68}
Calvin~C. Moore.
\newblock Group extensions of {$p$}-adic and adelic linear groups.
\newblock {\em Inst. Hautes \'{E}tudes Sci. Publ. Math.}, (35):157--222, 1968.

\bibitem[MW87]{MW87}
C.~M{\oe}glin and J.-L. Waldspurger.
\newblock Mod\`eles de {W}hittaker d\'eg\'en\'er\'es pour des groupes
  {$p$}-adiques.
\newblock {\em Math. Z.}, 196(3):427--452, 1987.

\bibitem[MW95]{MW95}
C.~M{\oe}glin and J.-L. Waldspurger.
\newblock {\em Spectral decomposition and {E}isenstein series}, volume 113 of
  {\em Cambridge Tracts in Mathematics}.
\newblock Cambridge University Press, Cambridge, 1995.
\newblock Une paraphrase de l'\'Ecriture [A paraphrase of Scripture].

\bibitem[PSR87]{PSR87}
I.~Piatetski-Shapiro and S.~Rallis.
\newblock {\em ${L}$-functions for the classical groups}, volume 1254 of {\em
  Lecture Notes in Math.}
\newblock Springer-Verlag, New York, 1987.

\bibitem[Suz98]{Suzuki98}
Toshiaki Suzuki.
\newblock Distinguished representations of metaplectic groups.
\newblock {\em Amer. J. Math.}, 120(4):723--755, 1998.

\bibitem[Tak16]{Takeda16}
Shuichiro Takeda.
\newblock Metaplectic tensor products for automorphic representation of
  {$\widetilde{GL}(r)$}.
\newblock {\em Canad. J. Math.}, 68(1):179--240, 2016.

\bibitem[Tak17]{Takeda17}
Shuichiro Takeda.
\newblock Remarks on metaplectic tensor products for covers of {${\rm GL}_r$}.
\newblock {\em Pacific J. Math.}, 290(1):199--230, 2017.

\bibitem[Wei11]{Weissman11}
Martin~H. Weissman.
\newblock Managing metaplectiphobia: covering {$p$}-adic groups.
\newblock In {\em Harmonic analysis on reductive, {$p$}-adic groups}, volume
  543 of {\em Contemp. Math.}, pages 237--277. Amer. Math. Soc., Providence,
  RI, 2011.

\bibitem[Wei16]{Weissman16a}
Martin~H. Weissman.
\newblock Covering groups and their integral models.
\newblock {\em Trans. Amer. Math. Soc.}, 368(5):3695--3725, 2016.

\bibitem[Yam14]{Yamana14}
Shunsuke Yamana.
\newblock L-functions and theta correspondence for classical groups.
\newblock {\em Invent. Math.}, 196(3):651--732, 2014.

\end{thebibliography}

\end{document}